\pdfoutput=1 %for arXiv submissions
\DeclareSymbolFont{AMSb}{U}{msb}{m}{n}
\documentclass[11pt,noamsfonts,a4paper]{amsart}
\usepackage[left=1in, right=1in, top=1in, bottom=1in]{geometry}
\usepackage{mathtools}
\usepackage{braket}
\usepackage[charter]{mathdesign}
\usepackage[tracking]{microtype}
\usepackage{tabularx}
\usepackage{etoolbox}
\usepackage{graphicx}
\usepackage{stmaryrd} % \mapsfrom

\usepackage{comment}

\usepackage{amsmath,amsthm}
\newtheoremstyle{pineapple}%
  {1em}{1em}%
  {\itshape}{}%
  {\bfseries}{. ---}
  {0.5em}{}

\newtheoremstyle{durian}%
  {1em}{1em}%
  {}{}%
  {\bfseries}{. ---}
  {0.5em}{}

\makeatletter
\def\swappedhead#1#2#3{%
  \thmnumber{\@upn{\the\thm@headfont#2\@ifnotempty{#1}{.~}}}%
  \thmname{#1}%
  \thmnote{ {\the\thm@notefont(#3)}}}
\makeatother

\makeatletter
\newcommand*\rel@kern[1]{\kern#1\dimexpr\macc@kerna}
\newcommand*\widebar[1]{%
  \begingroup
  \def\mathaccent##1##2{%
    \rel@kern{0.8}%
    \overline{\rel@kern{-0.8}\macc@nucleus\rel@kern{0.2}}%
    \rel@kern{-0.2}%
  }%
  \macc@depth\@ne
  \let\math@bgroup\@empty \let\math@egroup\macc@set@skewchar
  \mathsurround\z@ \frozen@everymath{\mathgroup\macc@group\relax}%
  \macc@set@skewchar\relax
  \let\mathaccentV\macc@nested@a
  \macc@nested@a\relax111{#1}%
  \endgroup
}
\makeatother

\makeatletter
\def\@sect#1#2#3#4#5#6[#7]#8{%
  \edef\@toclevel{\ifnum#2=\@m 0\else\number#2\fi}%
  \ifnum #2>\c@secnumdepth \let\@secnumber\@empty
  \else \@xp\let\@xp\@secnumber\csname the#1\endcsname\fi
  \@tempskipa #5\relax
  \ifnum #2>\c@secnumdepth
    \let\@svsec\@empty
  \else
    \refstepcounter{#1}%
    \edef\@secnumpunct{%
      \ifdim\@tempskipa>\z@ % not a run-in section heading
        \@ifnotempty{#8}{.~}%
      \else
        \@ifempty{#8}{.}{.~}%
      \fi
    }%
    \@ifempty{#8}{%
      \ifnum #2=\tw@ \def\@secnumfont{\bfseries}\fi}{}%
    \protected@edef\@svsec{%
      \ifnum#2<\@m
        \@ifundefined{#1name}{}{%
          \ignorespaces\csname #1name\endcsname\space
        }%
      \fi
      \@seccntformat{#1}%
    }%
  \fi
  \ifdim \@tempskipa>\z@ % then this is not a run-in section heading
    \begingroup #6\relax
    \@hangfrom{\hskip #3\relax\@svsec}{\interlinepenalty\@M #8\par}%
    \endgroup
    \ifnum#2>\@m \else \@tocwrite{#1}{#8}\fi
  \else
  \def\@svsechd{#6\hskip #3\@svsec
    \@ifnotempty{#8}{\ignorespaces#8\unskip
       \@addpunct.}%
    \ifnum#2>\@m \else \@tocwrite{#1}{#8}\fi
  }%
  \fi
  \global\@nobreaktrue
  \@xsect{#5}}
\makeatother

\makeatletter
\def\@seccntformat#1{%
  \protect\textup{\protect\@secnumfont
    \ifnum\pdfstrcmp{subsection}{#1}=0 \bfseries\fi% subsection # in \bfseries
    \csname the#1\endcsname
    \protect\@secnumpunct
  }%
}
\makeatother

\theoremstyle{pineapple}
\newtheorem{IntroTheorem}{Theorem}
\newtheorem*{IntroTheorem*}{Theorem}

\newtheorem{IntroConjecture}[IntroTheorem]{Conjecture}

\swapnumbers
\newtheorem{Theorem}[subsection]{Theorem}
\newtheorem{Lemma}[subsection]{Lemma}
\newtheorem{Proposition}[subsection]{Proposition}
\newtheorem{Corollary}[subsection]{Corollary}

\theoremstyle{durian}

\newtheorem{Example}[subsection]{Example}

\usepackage[dvipsnames]{xcolor}
\usepackage{hyperref} %always add last
\hypersetup{
  colorlinks = true,
  linkcolor = Fuchsia,
  urlcolor = Fuchsia,
  citecolor = ForestGreen,
  linkbordercolor = {white}}
\usepackage{tikz-cd}
\usetikzlibrary{arrows}

\tikzset{
  symbol/.style={
    draw=none,
    every to/.append style={
      edge node={node [sloped, allow upside down, auto=false]{$#1$}}}
  }
}

\usepackage{enumitem}
\setlist[1]{labelindent=\parindent}
\setlist[1]{labelsep=0.5em}
\setlist[enumerate,1]{label={\upshape (\roman*)}, ref={\upshape (\roman*)}}

\makeatletter
\newcommand{\leqnomode}{\tagsleft@true\let\veqno\@@leqno}
\newcommand{\reqnomode}{\tagsleft@false\let\veqno\@@eqno}
\makeatother

\usepackage[utf8]{inputenc}

\tikzset{>={Straight Barb[length=2pt,width=4pt]}, commutative diagrams/arrow style=tikz}

\makeatletter
\let\c@equation\c@subsection

\makeatother

\DeclareMathOperator{\res}{res}
\DeclareMathOperator{\ev}{ev}
\DeclarePairedDelimiter{\abs}{\lvert}{\rvert}
\DeclarePairedDelimiter{\norm}{\lVert}{\rVert}

\DeclareMathOperator{\Fr}{Fr}
\DeclareMathOperator{\id}{id}
\DeclareMathOperator{\Spec}{Spec}
\DeclareMathOperator{\Sym}{Sym}
\DeclareMathOperator{\Aut}{Aut}

\DeclareMathOperator{\image}{image}

\DeclareMathOperator{\pr}{pr}
\DeclareMathOperator{\Pic}{Pic}
\DeclareMathOperator{\mult}{mult}

\newcommand{\precdot}{\prec\mathrel{\mkern-5mu}\mathrel{\cdot}}

\makeatletter
\newcommand{\preceqdot}{\mathrel{\mathpalette\pr@ceqd@t\relax}}
\newcommand{\pr@ceqd@t}[2]{%
  \begingroup
  \sbox\z@{$#1\prec$}\sbox\tw@{$#1\preceq$}%
  \dimen@=\dimexpr\ht\tw@-\ht\z@\relax
  {\preceq}%
  \mkern-5mu
  \raisebox{\dimen@}{$\m@th#1\cdot$}%
  \endgroup
}
\makeatother

\makeatletter
\newcommand*{\coloneqq}{\mathrel{\rlap{%
           \raisebox{0.3ex}{$\m@th\cdot$}}%
           \raisebox{-0.3ex}{$\m@th\cdot$}}%
           =}
\newcommand{\eqqcolon}{=%
           \mathrel{\rlap{%
           \raisebox{0.3ex}{$\m@th\cdot$}}%
           \raisebox{-0.3ex}{$\m@th\cdot$}}}
\makeatother

\newcommand{\punct}[1]{\makebox[0pt][l]{\,#1}} %so that punctuation does not mess with tikz layout
\newcommand{\parref}[1]{{\bf\ref{#1}}}

\DeclareMathOperator{\rank}{rank}

\DeclareMathOperator{\Hom}{Hom}

\DeclareMathOperator{\codim}{codim}

\newcommand{\kk}{\mathbf{k}}

\newcommand{\sO}{\mathcal{O}}

\newcommand{\qaticci}{(q;\mathbf{a})\operatorname{\bf\!-tics}}
\newcommand{\hrefSP}[1]{\href{https://stacks.math.columbia.edu/tag/#1}{#1}}
\newcommand{\citeSP}[1]{\cite[\textbf{\hrefSP{#1}}]{stacks-project}}
\newcommand{\citeForms}[1]{\cite[\href{https://arxiv.org/pdf/2301.09929.pdf\#subsection.#1}{\textbf{#1}}]{qbic-forms}}
\newcommand{\citeFano}[1]{\cite[\href{https://arxiv.org/pdf/2307.06160.pdf\#subsection.#1}{\textbf{#1}}]{fano-schemes}}

\newcommand{\citeThesis}[1]{\cite[\href{https://arxiv.org/pdf/2205.05273.pdf\#subsection.#1}{\textbf{#1}}]{thesis}}

\makeatletter
\newcommand{\smallbullet}{} % for safety
\DeclareRobustCommand\smallbullet{%
  \mathord{\mathpalette\smallbullet@{0.75}}%
}
\newcommand{\smallbullet@}[2]{%
  \vcenter{\hbox{\scalebox{#2}{$\m@th#1\bullet$}}}%
}
\makeatother

\newcommand{\subsectiondash}[1]{\subsection{#1}\textbf{---}\;}
\newcommand{\PP}{\mathbf{P}}
\newcommand{\Types}{\mathbf{Prfl}}

\title[Geometry of \((q;a)\)-tic hypersurfaces]{Profiles, linear spaces, and unirationality of complete intersections}
\author{Raymond Cheng}

\address{SB MATH CAG \\
  EPFL \\
  Station 8 \\
  1015 Lausanne \\
  Switzerland
}
\email{raymond.cheng@epfl.ch}
\begin{document}
\begin{abstract}
Complete intersections may be unexpectedly simple over fields of positive characteristic: for instance, they may be unirational despite being of general type. One explanation is given by \emph{profiles}, structure that tracks the special shape of polynomials, refining the degree. The aim of this work is to show that complete intersections with small profile should be considered simple by generalizing two classical results on low degree complete intersections: First, the basic geometry of Fano schemes associated with complete intersections depends only on the profile, so that complete intersections with small profile contain many linear spaces. Second, a general complete intersection is unirational once its dimension is sufficiently large compared to its profile.
\end{abstract}
\maketitle
\setcounter{tocdepth}{1}

\thispagestyle{empty}
\section*{Introduction}
High degree hypersurfaces and complete intersections in projective space are,
by most measures, geometrically complicated. For instance, unlike the
simplest of algebraic varieties \cite{Kollar:Simplest}, high degree
hypersurfaces over the complex numbers contain very few---if any!---rational
subvarieties: see \cite{Ein,Voisin,Pacienza, RY:Clemens}. New phenomena appear,
however, in positive characteristic \(p > 0\), complicating the
correlation between degree and geometric complexity. The Fermat hypersurface
\[
X \coloneqq \{x_0^{q+1} + \cdots + x_n^{q+1} = 0\} \subset \PP^n
\]
of degree \(q + 1\), with \(q = p^e\) an integer power, is an exemplar:
Many properties of \(X\) are reminiscent of those typically expected in
quadric hypersurfaces: it is projectively self-dual and its Gauss map is a
homeomorphism \cite{Wallace, KP:Gauss}; its smooth hyperplane sections have
constant moduli \cite{Beauville}; it has many linear spaces \cite{fano-schemes};
and, if \(n \geq 3\), it is even unirational \cite{Shioda:Unirationality}!
Traditionally, these curiosities are explained \emph{ad hoc} by identifying the
equation of \(X\) with a Hermitian form for the quadratic extension
\(\mathbf{F}_{q^2} \mid \mathbf{F}_q\), as in \cite{BC:Hermitian, Hefez,
Shimada:Lattices}.

In contrast, I view the quadratic nature of \(X\) as the confluence between the
almost linear nature of \(q\)-powers in characteristic \(p\) and the special
form of the equation of \(X\). This perspective places this Fermat
hypersurface, a historically isolated example, in a broad context and suggests
a systematic way of identifying related phenomena. The approach taken here is
as follows:

A \emph{profile} with respect to \(q\) is an integer polynomial
\(a(t) = a_0 + a_1 t + \cdots + a_m t^m \in \mathbf{Z}_{\geq 0}[t]\)
subject to a technical condition, given in \parref{qatics-profiles}, which is
satisfied, for example, if \(a_j \leq q - 1\) for each \(j = 0,\ldots,m\). A
polynomial with coefficients a field \(\kk\) of characteristic \(p\) is said to
have profile \(a(t)\) with respect to \(q\)---or simply is called a
\emph{\((q;a)\)-tic polynomial}---if it is of the form
\[
f(x_0,\ldots,x_n) =
\sum\nolimits_{i \in I} \prod\nolimits_{j = 0}^m f_{i,j}(x_0,\ldots,x_n)^{q^j} \in
\kk[x_0,\ldots,x_n]
\]
where each \(f_{i,j}\) is a homogeneous polynomial of degree \(a_j\). The form
of this expression is invariant under linear changes of variables, so
\((q;a)\)-tic polynomials span a canonical linear subspace within the space of
polynomials of degree \(a(q) = a_0 + a_1 q + \cdots + a_m q^m\). The vanishing
locus in \(\PP^n\) of a \((q;a)\)-tic polynomial is a \emph{\((q;a)\)-tic
hypersurface}. The Fermat hypersurface above is an example with profile
\(a(t) = 1 + t\) with respect to \(q\). The purpose of this works is to develop
the following principle:
\emph{\((q;a)\)-tic hypersurfaces behave as if they were hypersurfaces of
degree \(a(1 + \varepsilon) \approx a_0 + a_1 + \cdots + a_m\) for some small
real number \(\varepsilon > 0\).}

One aspect of this principle is that hypersurfaces sharing the same profile
should be treated as a single family. Concretely, one looks for properties of
\((q;a)\)-tic hypersurfaces which may be expressed in terms of the profile
\(a\), ideally independently of, or at least uniformly in, the prime power
\(q\). A simple illustration of this is found in the geometry of
linear spaces:

\begin{IntroTheorem}\label{intro-fano}
Let \(a = \sum\nolimits_{j \geq 0} a_jt^j\) be a profile such that
\(a\neq 2t^m\). The Fano scheme \(\mathbf{F}_r(X)\) of \(r\)-planes in a
\((q;a)\)-tic hypersurface \(X \subset \PP^n\) is cut out of the Grassmannian
by a section of a vector bundle of rank
\[
(r+1)(n-r) - \delta(n,a,r)
\;\;\text{where}\;\;
\delta(n,a,r) \coloneqq
(r+1)(n-r) -
\prod\nolimits_{j \geq 0} \binom{a_j + r}{r}.
\]
If \(\delta(n,a,r) < 0\), then \(\mathbf{F}_r(X)\) is empty for \(X\)
general. If \(\delta(n,a,r) \geq 0\), then \(\mathbf{F}_r(X)\) is
\begin{enumerate}
\item\label{intro-fano.nonempty}
nonempty for every \(X\);
\item\label{intro-fano.dimension}
irreducible of dimension \(\delta(n,a,r)\) for \(X\) general;
\item\label{intro-fano.smooth}
smooth of dimension \(\delta(n,a,r)\) for \(X\) general when the profile
has constant term \(a_0 \neq 0\);
and
\item\label{intro-fano.connected}
connected for every \(X\) when \(\delta(n,a,r) > 0\).
\end{enumerate}
If \(\mathbf{F}_r(X)\) is of dimension \(\delta(n,a,r)\), then its
dualizing sheaf is given by a power of the Pl\"ucker line bundle:
\[
\omega_{\mathbf{F}_r(X)} \cong
\sO_{\mathbf{F}_r(X)}(\gamma(a,r,q) - n - 1)
\;\;\text{where}\;\;
\gamma(a,r,q) \coloneqq
\frac{a(q)}{r+1} \cdot \prod\nolimits_{j \geq 0} \binom{a_j + r}{r}.
\]
\end{IntroTheorem}

This is a summary of \parref{fano-equations}, \parref{fano-theorem}, and
\parref{fano-canonical} in the case of a hypersurface, all of which are
formulated and proven without the restriction \(a \neq 2t^m\) and 
more generally for \((q;\mathbf{a})\)-tic complete intersections, vanishing loci
of a regular sequence of \((q;a)\)-tic polynomials with \(a\) ranging over a
multi-set of profiles \(\mathbf{a}\). These generalize classical results on
Fano schemes of complete intersections and, once the techniques developed for
handling \((q;a)\)-tic equations are developed in
\S\S\parref{section-profiles}--\parref{section-qatic}, their proofs essentially
follow the strategy from the classical case found in \cite[\S2]{DM} and
\cite[\S V.4]{Kollar:Book}.

A second aspect of this principle is, of course, that \((q;a)\)-tic
hypersurfaces whose profile is small compared to its dimension should be
geometrically simple, regardless of the prime power \(q\). Theorem
\parref{intro-fano} is some evidence in this direction since the dimension of
the Fano scheme depends only on \(a\) and not on \(q\). To explain a second
result in this direction, recall classical results of Morin and Predonzan
from \cite{Morin, Predonzan} which show that a general complete intersection in
\(\PP^n\) is unirational whenever its total degree \(d\) is much smaller than
\(n\); see \cite[pp.44--46]{Roth} for a classical source, but also \cite{PS}
for a succinct exposition in modern language and \cite{Ramero} for an improved
bound. An analogue of this for \((q;\mathbf{a})\)-tic complete intersections
is:

\begin{IntroTheorem}\label{intro-unirationality}
Given a multi-profile \(\mathbf{a}\), there exists an integer
\(n_0 \coloneqq n_0(\mathbf{a})\), depending only on \(\mathbf{a}\), such that
for all \(n \geq n_0\), a general \((q;\mathbf{a})\)-tic complete intersection
in \(\PP^n\) is unirational.
\end{IntroTheorem}

This is proven in \S\parref{section-unirational} via an inductive argument,
using the constructions developed in
\S\S\parref{section-families}--\parref{section-residual}.  A similar result,
formulated and proven in a different language, seems to appear in
\cite{Shimada:Unirationality}; see also the related result in
\cite{Shimada:CI}.  The integer \(n_0(\mathbf{a})\) can be computed for small
\(\mathbf{a}\): see \parref{unirationality-estimates}. The result is already
interesting when \(\mathbf{a} = (d)\) consists of a single constant polynomial,
this is a statement about unirationality of hypersurfaces of degree \(d\), in
which case the integer \(n_0(d)\) appearing here is on the order of
\(2^{d2^d}\), significantly improving past constructions: see the companion
paper \cite{bounds}.

Briefly, the unirationality construction of Morin and Predonzan goes as
follows: given a general complete intersection \(X \subset \PP^n\) of
multi-degree \(\mathbf{d} = (d_1,\ldots,d_c)\), projection away from a general
\(r\)-plane in \(X\) yields a fibration \(\widetilde{X} \to \PP^{n-r-1}\) whose
generic fibre \(X'\) is itself a complete intersection of multi-degree
\(\mathbf{d}' = (d_1 - 1, \ldots, d_c - 1)\) in a \(\PP^{r+1}\). With an
appropriate choice of \(r\) and \(n\) depending on \(\mathbf{d}\), it is
possible to find a base change of \(X'\) that carries a large linear space,
allowing the argument to proceed inductively. This strategy does not adapt
in a straightforward manner for \((q;\mathbf{a})\)-tic complete intersections
because the generic fibre of the projection \(\widetilde{X} \to \PP^{n-r-1}\)
does not seem to have structure that is captured by the theory of profiles: see
\parref{families-not-strict-transform} for an explicit example.

Instead, Theorem \parref{intro-unirationality} is established with a
generalization of an old unirationality construction for cubic hypersurfaces
found in \cite[Appendix B]{CG} and \cite[\S2]{Murre} which is based on the
following observation: Given a hypersurface \(X \subset \PP^n\) of degree
\(d\), the space of \emph{penultimate tangents}
\[
X' = \{
(x,[\ell]) : 
\ell \subset \PP^n\;\text{a line intersecting \(X\) at \(x\) with multiplicity \(\geq d-1\)}
\}
\]
is generically a family of complete intersections of multi-degree
\(\mathbf{d}' = (d-2,d-3,\ldots,1)\) over \(X\). Moreover, for \(X\) general,
there is a dominant rational map \(X' \dashrightarrow X\) sending
\((x,[\ell])\) to the residual point of intersection \(x' = X \cap \ell - (d-1)x\).
By restricting \(X'\) to a sufficiently large and general linear space in
\(X\), unirationality can be established by induction on a poset consisting of
all multi-degrees. In particular, it is essential to prove the result for all
complete intersections.

Much effort is made to perform the constructions in
\S\S\parref{section-families}--\parref{section-unirational} globally, so that
they work well in families, albeit introducing additional technicalities. The
hope is to eventually make the generality conditions on \(X\) in Theorems
\parref{intro-fano} and \parref{intro-unirationality} effective \(n\)
sufficiently large. In characteristic \(0\), the works \cite{HMP, BR} show that
the Fano schemes \(\mathbf{F}_r(X)\) of \emph{every} smooth hypersurface \(X
\subset \PP^n\) of degree \(d\) is irreducible of its expected dimension once
\(n \geq 2\binom{d+r-1}{r} + r\). This may then be used in a refinement of the
Morin construction to show that \emph{every} smooth hypersurface of degree
\(d\) is unirational once \(n \geq 2^{d!}\). While it is possible that these
results may be extended to characteristics \(p > d\), Theorem
\parref{intro-fano} shows that these statements cannot hold verbatim when \(p
\leq d\) wherein there exists a nontrivial profile \(a\) with \(a(p) = d\): see
\parref{qatics-ordering-properties}\ref{qatics-ordering-properties.degrees}. An
optimistic guess would be that the appearance of these canonical sub-linear
systems is the only issue which arises, suggesting the following:

\begin{IntroConjecture}\label{intro-fano-conjecture}
Let \(a(t) = a_0 + \cdots + a_m t^m \in \mathbf{Z}_{\geq 0}[t]\) with
\(a_0 \neq 0\) and \(r \in \mathbf{Z}_{\geq 1}\). Then there exists
an integer \(n_0(a,r)\) such that for every prime
\(p > \max_i \{ a_i\}\), every
\(n \geq n_0(a,r)\), and every smooth \((p;a)\)-tic hypersurface \(X \subset \PP^n\),
the Fano scheme \(\mathbf{F}_r(X)\) is irreducible
of dimension \(\delta(n,a,r)\).
\end{IntroConjecture}

Note that \(n_0(a,r)\) does not depend on \(p\), so that
Conjecture \parref{intro-fano-conjecture} predicts a uniform phenomenon for
varying characteristic. Unfortunately, this statement cannot be quite right for
higher prime powers \(q\), since there will be a profile with respect to \(p\)
with larger expected dimension. A result in this direction would be a first
step toward an effective statement for unirationality.

\smallskip
\noindent\textbf{Further questions. --- }%
This work is primarily focused on the linear projective geometry of
\((q;a)\)-tic hypersurfaces, and so many basic questions and properties remain
to be explored: automorphisms, moduli, singularities, and so forth. Three main
directions that seem the most interesting are:

First, the principle that \((q;a)\)-tic hypersurfaces behave as if they were of
degree \(a(1 + \varepsilon)\) is taken to be a rough qualitative heuristic, and
it would be interesting to make this more precise and quantitative. For
instance, the work in \cite{fano-schemes, qbic-threefolds} shows that the
geometry of lines in hypersurfaces of profile \(a(t) = 1 + t\), or
\emph{\(q\)-bic hypersurfaces}, is reminiscent of that of cubics.

Second, the theory developed here is \emph{extrinsic} in that
\((q;a)\)-tic-ness is a structure with which an object may be equipped with.
Might there be an intrinsic characterization? One instance occurs with
\(q\)-bic hypersurfaces, wherein \cite{KKPSSVW} show that they are
characterized amongst hypersurfaces of the same degree as those having the
smallest \(F\)-pure threshold. From a different direction, many examples
suggest that non-\(F\)-splitness is related to the special form of the defining
equations: see \cite{Saito, BLRT, KKPSSVW:Cubics, MW} for a few examples.

Third, the definition in \S\parref{section-qatic} of \((q;a)\)-tic polynomials
generalize to provide canonical linear systems in
\(\Gamma(Y,\mathcal{L}^{\otimes a(q)})\) for any scheme \(Y\) and
\(\mathcal{L} \in \Pic Y\). The situation
here is the case \((Y,\mathcal{L}) = (\PP^n, \sO_{\PP^n}(1))\),
and other cases that have occurred in the literature include Deligne--Lusztig
varieties \cite{DL} and the Frobenius incidence correspondence of
\cite{Shimada:Incidence}. It would be interesting to study the special
properties of the divisors and linear systems constructed in this fashion.

\smallskip
\noindent\textbf{Outline. --- }%
\S\parref{section-profiles} develops the combinatorics of profiles and
\S\parref{section-qatic} gives a definition of \((q;a)\)-tic hypersurfaces and
studies their basic properties. Fano schemes of \((q;\mathbf{a})\)-tic complete
intersections are studied in \S\parref{section-fano}. Families of
\((q;\mathbf{a})\)-tic complete intersections are defined in
\S\parref{section-families}. The penultimate tangent construction for families
is made in \S\parref{section-tlines} and the corresponding residual point map
is studied in \S\parref{section-residual}. Finally, the unirationality result
is proven in \S\parref{section-unirational}.

\smallskip
\noindent\textbf{Notation. --- }%
Throughout, \(\kk\) denotes an algebraically closed field of characteristic \(p
> 0\) and \(q \coloneqq p^e\) is a positive integer power of \(p\). Unless
otherwise stated, \(V\) denotes a \(\kk\)-vector space of dimension \(n+1\).
Write \(\Fr \colon \kk \to \kk\) for the \(q\)-power Frobenius morphism and,
for any \(\kk\)-vector space \(W\), let
\(W^{[1]} \coloneqq \kk \otimes_{\Fr, \kk} W\) be its Frobenius twist.
Set \(W^{[0]} \coloneqq W\) and, for each integer \(i \geq 1\), inductively
define \(W^{[i+1]} \coloneqq (W^{[i]})^{[1]}\). Schemes are all taken to be
over \(\kk\) and \(\PP V\) is the projective space of lines in \(V\).

\smallskip
\noindent\textbf{Acknowledgements. --- }%
Thanks to Jan Lange, Matthias Sch\"utt, and Noah Olander for their interest and
helpful conversations related to this work. I was supported by a Humboldt
Postdoctoral Research Fellowship during the completion of this work.

\section{Profiles}\label{section-profiles}
A profile is a discrete invariant associated with a polynomial in positive
characteristic \(p > 0\) which records its special shape with respect to a
positive integer power \(q = p^e\) of the characteristic. Profiles are subject
to a technical condition which ensures a certain uniqueness in writing
polynomials in their distinguished form: see \parref{qatic-injective-mult}.
While this hypothesis is not strictly necessary at this point, it often becomes
important in applications. This section defines profiles and studies their
combinatorics, especially various order relations amongst profiles.

\subsectiondash{Profiles}\label{qatics-profiles}
Given an integer polynomial \(a(t) \coloneqq a_0 + a_1 t + \cdots + a_m t^m\)
with nonnegative coefficients, consider the tensor functor defined by
\[
\mathrm{S}^a(V^\vee) \coloneqq
\bigotimes\nolimits^m_{j = 0} \Sym^{a_j}(V^\vee)^{[j]}.
\]
Identifying \(\Sym^{a_j}(V^\vee)^{[j]}\) with the \(\kk\)-linear subspace of
\(\Sym^{a_j q^j}(V^\vee)\) consisting of \(q^j\)-powers provides a
multiplication map
\[
\operatorname{mult} \colon
\mathrm{S}^a(V^\vee) \to
\Sym^{a(q)}(V^\vee).
\]
Call \(a(t)\) a \emph{profile} with respect to the prime power \(q\) if this
multiplication map is injective for all finite-dimensional \(\kk\)-vector
spaces \(V\). Denote by
\[
\Types \coloneqq
\{
a \in \mathbf{Z}_{\geq 0}[t] :
\mult \colon \mathrm{S}^a(V^\vee) \to \Sym^{a(q)}(V^\vee)
\;\text{is injective for all finite-dimensional \(V\)}
\}
\]
the set of profiles with respect to \(q\). Given \(a \in \Types\) and a vector
space \(V\), \(\mathrm{S}^a(V^\vee)\) will often tacitly be identified with its
image in \(\Sym^{a(q)}(V^\vee)\) under the injective multiplication map.

Whether or not a given element of \(\mathbf{Z}_{\geq 0}[t]\) is a profile
depends on \(q\): For instance, \(a(t) \coloneqq t + q\) is not a profile with
respect to \(q\) since, for a \(2\)-dimensional vector space
\(V^\vee = \kk\cdot u \oplus \kk \cdot v\), the elements \(u^{[1]} \otimes v^q\) and
\(v^{[1]} \otimes u^q\) are distinct in \(\mathrm{S}^a(V^\vee)\), but
have the same image in \(\Sym^{2q}(V^\vee)\). The same polynomial \(t + q\)
is, however, a profile with respect to \(q^k\) for any \(k \geq 2\). In what
follows, however, \(q\) remains fixed, the dependence of \(\Types\) on \(q\)
will be suppressed.

\subsectiondash{Partial orderings on profiles}\label{qatic-poset}
Endow \(\Types\) with two partial orderings as follows: First, viewing a
profile as a sequence of nonnegative integers provides the product ordering
\(\preceq\), where for profiles
\(a \coloneqq \sum\nolimits_{j \geq 0} a_j t^j\) and
\(b \coloneqq \sum\nolimits_{j \geq 0} b_j t^j\),
\[
a \preceq b
\;\;\iff\;\;
a_j \leq b_j\;\;\text{for all \(j \geq 0\)}.
\]
Second, identifying a \(a \in \Types\) with the space \(\mathrm{S}^a(V^\vee)\)
provides the containment ordering \(\sqsubseteq\), where \(a \sqsubseteq b\) if and only if
\(a(q) = b(q) \eqqcolon d\) and
\[
\mathrm{S}^a(V^\vee) \subseteq
\mathrm{S}^b(V^\vee) \subseteq
\Sym^d(V^\vee)
\;
\text{for all finite-dimensional \(\kk\)-vector spaces \(V\),}
\]
where the tensor functors are identified with their image under the multiplication
map.

\subsectiondash{Examples and properties}\label{qatics-ordering-properties}
Regarding the partial orderings \(\preceq\) and \(\sqsubseteq\) on \(\Types\):
\begin{enumerate}
\item\label{qatics-ordering-properties.degrees}
The product ordering is strictly monotonic for \emph{numerical degrees}: If
\(a \prec b\), then \(a(q) < b(q)\). The containment ordering, on the other
hand, splits \(\Types\) into connected components
\[
\Types_d \coloneqq \{a \in \Types : a(q) = d\}
\]
indexed by numerical degrees. Each \(\Types_d\) has a unique
\(\sqsubseteq\)-maximal element given by the constant profile
\(a_{d\text{-}\mathrm{max}}(t) = d\), and a unique \(\sqsubseteq\)-minimal
element given by
\(a_{d\text{-}\mathrm{min}}(t) = a_0 + a_1t + \cdots + a_mt^m\) with
coefficients \(0 \leq a_j \leq q - 1\) arising from the base \(q\) expansion
of \(d\).
\item\label{qatics-ordering-properties.coefficient-sums}
More interestingly, both orderings are strictly monotonous for coefficient
sums. This is clear for \(\preceq\). To see that \(a \sqsubset b\) implies
\(a(1) < b(1)\), observe that
\(\mathrm{S}^a(V^\vee) \subset \mathrm{S}^b(V^\vee)\) for varying vector spaces
\(V\) gives the dimensional inequality
\[
\prod\nolimits_{j \geq 0} \binom{a_j + n}{a_j} <
\prod\nolimits_{j \geq 0} \binom{b_j + n}{b_j}
\;\;\text{for all \(n \geq 0\)}.
\]
Using that \(\binom{n}{k} \sim \frac{n^k}{k!}\) for large \(n\) and fixed
\(k\), taking logarithms, and doing away with constants shows that
for all \(n \gg 0\),
\[
\sum\nolimits_{j \geq 0} a_j \cdot \log(a_j + n) <
\sum\nolimits_{j \geq 0} b_j \cdot \log(b_j + n).
\]
By taking \(n\) even larger, \(\log(n)\) may be made arbitrarily close to
\(\log(c + n)\) for any constant \(c\). Since there are only finitely many
constants \(a_j\) and \(b_j\), this implies that \(a(1) < b(1)\).
\item\label{qatics-ordering-properties.not-total}
Neither \(\preceq\) nor \(\sqsubseteq\) give total orderings on \(\Types\).
This remains the case for \(\sqsubseteq\) even upon restriction to a connected
component \(\Types_d\): for instance,
\[
a(t) \coloneqq t^2 + (q+1)
\;\;\text{and}\;\;
b(t) \coloneqq (q+1)t + 1
\]
are profiles of numerical degree \(d = q^2+q+1\) which are
\(\sqsubseteq\)-incomparable.

\item\label{qatics-ordering-properties.reduced}
Call \(b \in \Types\) \emph{nonreduced} if its constant term is zero: \(b(0) =
0\). Any profile \(a\) preceding a nonreduced profile \(b\) in either ordering
is also nonreduced.
\item\label{qatics-ordering-properties.monoid}
Both orderings are compatible with addition in the following sense:
For \(a, b, c \in \Types\) such that \(a + c\) and \(b + c\) are also
profiles,
\[
\text{if}\;
a \sqsubseteq b,\;
\text{then}\;
a + c \sqsubseteq b + c,
\]
and similarly for \(\preceq\). This is straightforward to see for \(\preceq\).
For \(\sqsubseteq\), let \(V\) be a \(\kk\)-vector space and let \(f \in
\mathrm{S}^{a+c}(V^\vee)\). Then there exists an expansion of the form
\[
f =
\sum\nolimits_{i \in I}
g_i h_i
\;\;\text{with}\;
g_i \in \mathrm{S}^a(V^\vee)
\;\text{and}\;
h_i \in \mathrm{S}^c(V^\vee).
\]
Viewing the \(g_i\) as elements of \(\mathrm{S}^b(V^\vee)\) shows that
\(f \in \mathrm{S}^{b+c}(V^\vee)\).
\end{enumerate}

A third ordering \(\rightsquigarrow\) on \(\Types\) obtained by combining the
two will be also be useful: For \(a, b \in \Types\),
\[
a \rightsquigarrow b
\iff
\text{there exists \(a' \in \mathbf{Z}_{\geq 0}[t]\) such that
\(a \preceq a' \sqsubseteq b\)}.
\]
Here, the relations \(\preceq\) and \(\sqsubseteq\) are extended to
\(\mathbf{Z}_{\geq 0}[t]\) in the natural way: \(a \preceq a'\) means that each
coefficient of \(a'\) is at least that
of \(a\), and \(a' \sqsubseteq b\) means that the image of
\(\mathrm{S}^{a'}(V^\vee)\) under the multiplication map is contained in
\(\mathrm{S}^b(V^\vee)\) for all finite-dimensional \(\kk\)-vector spaces \(V\).

\begin{Lemma}\label{qatics-ordering}
The relation \(\rightsquigarrow\) is a partial ordering on \(\Types\).
\end{Lemma}

\begin{proof}
Reflexivity and antisymmetry follow directly from the corresponding properties
of \(\preceq\) and \(\sqsubseteq\). For transitivity, consider \(a,b,c \in \Types\)
such that \(a \rightsquigarrow b\) and \(b \rightsquigarrow c\).
By definition, this means that there are \(a', b' \in \mathbf{Z}_{\geq 0}[t]\)
satisfying
\[
a \preceq a' \sqsubseteq b
\;\;\text{and}\;\;
b \preceq b' \sqsubseteq c.
\]
Consider \(a'' \coloneqq a' + (b' - b)\). Since \(b \preceq b'\),
each coefficient of \(b' - b\) is nonnegative, and so \(a \preceq a''\).
Compatibility with addition from
\parref{qatics-ordering-properties}\ref{qatics-ordering-properties.monoid}
holds more generally for \(\sqsubseteq\) on \(\mathbf{Z}_{\geq 0}[t]\), and it
implies that \(a'' \sqsubseteq  b'\). Transitivity of \(\sqsubseteq\) then
gives \(a'' \sqsubseteq  c\), showing that \(a''\) witnesses
the relation \(a \rightsquigarrow c\).
\end{proof}

These partial orderings, \(\preceq\) in particular, make it simple to
formulate a criterion for when a nonnegative integer polynomial
\(a \in \mathbf{Z}_{\geq 0}[t]\) is a profile with respect to \(q\):

\begin{Lemma}\label{qatic-injective-mult}
Given \(a \in \mathbf{Z}_{\geq 0}[t]\), the following conditions are equivalent:
\begin{enumerate}
\item\label{qatic-injective-mult.mult}
\(\mult \colon \mathrm{S}^a(V) \to \Sym^{a(q)}(V)\) is injective
for every finite-dimensional \(\kk\)-vector space \(V\);
\item\label{qatic-injective-mult.two-dim}
\(\mult \colon \mathrm{S}^a(V) \to \Sym^{a(q)}(V)\) is injective for a
\(2\)-dimensional \(\kk\)-vector space \(V\); and
\item\label{qatic-injective-mult.poset}
the function \(\{b \in \mathbf{Z}_{\geq 0}[t]: 0 \preceq b \preceq a\} \to \mathbf{Z}\)
given by \(b \mapsto b(q)\) is injective.
\end{enumerate}
\end{Lemma}

\begin{proof}
\ref{qatic-injective-mult.mult} \(\Rightarrow\) \ref{qatic-injective-mult.two-dim}
is clear. For
\ref{qatic-injective-mult.two-dim} \(\Rightarrow\) \ref{qatic-injective-mult.poset},
choose a basis \(V = \kk \cdot u \oplus \kk \cdot v\).
Writing \(a = a_0 + a_1 t + \cdots + a_m t^m\), a basis for \(\mathrm{S}^a(V)\)
is given by
\begin{align*}
\mathrm{S}^a(\kk \cdot u \oplus \kk \cdot v) & =
\bigotimes\nolimits_{j = 0}^m \Sym^{a_j}(\kk \cdot u \oplus \kk \cdot v)^{[j]} =
\bigotimes\nolimits_{j = 0}^m
  \Big(\bigoplus\nolimits_{b_j = 0}^{a_j} \kk \cdot (u^{a_j - b_j} v^{b_j})^{[j]}\Big) \\
& =
\bigoplus\nolimits_{b_0 = 0}^{a_0} \cdots
\bigoplus\nolimits_{b_m = 0}^{a_m}
  \kk \cdot
  (u^{a_0 - b_0} v^{b_0}) \otimes
  (u^{a_1 -  b_1}v^{b_1})^{[1]} \otimes
  \cdots \otimes
  (u^{a_m - b_m} v^{b_m})^{[m]}.
\end{align*}
Multiplication sends the displayed basis element to \(u^{a(q) - b(q)}v^{b(q)}\)
where \(b \coloneqq b_0 + b_1 t + \cdots + b_m t^m\), so injectivity of
\(\mult \colon \mathrm{S}^a(V) \to \Sym^{a(q)}(V)\) is, in fact, equivalent to
injectivity of the function \(b \mapsto b(q)\).

Finally, for \ref{qatic-injective-mult.poset} \(\Rightarrow\)
\ref{qatic-injective-mult.mult}, suppose for sake of contradiction that there
is a \(V\) for which \(\mult \colon \mathrm{S}^a(V) \to \Sym^{a(q)}(V)\) is
not injective; choose such a \(V\) of minimal dimension and a nonzero element
\(\alpha \in \mathrm{S}^a(V)\) in the kernel. The previous argument shows that
\(\dim_{\kk} V \geq 2\), making it possible to choose a nonzero vector \(v \in
V\), a complementary subspace \(0 \neq U \subseteq V\), and a
splitting \(V = U \oplus \kk \cdot v\). Writing
\[
\Sym^{a_j}(V) = \Sym^{a_j}(U \oplus \kk \cdot v) =
\bigoplus\nolimits_{b_j = 0}^{a_j} \Sym^{a_j - b_j}(U) \cdot v^{b_j}
\]
provides \(\alpha\) with a unique expansion of the form
\[
\alpha
= C \cdot v^{a(q)} +
\sum\nolimits_{0 \preceq b \prec a}\Big(\sum\nolimits_{i \in I_b}
  (\beta_{i,0} \cdot v^{b_0}) \otimes
  (\beta_{i,1} \cdot v^{b_1})^{[1]} \otimes
  \cdots \otimes
  (\beta_{i,m} \cdot v^{b_m})^{[m]}\Big)
\]
for some \(C \in \kk\) and some \(\beta_{i,j} \in \Sym^{a_j - b_j}(U)\) for
each \(i \in I_b\) and \(0 \leq j \leq m\). Since \(\alpha \in \ker(\mult)\),
\[
0
= \mult(\alpha)
= C \cdot v^{a(q)} + \sum\nolimits_{0 \preceq b \prec a}\Big(\sum\nolimits_{i \in I_b}
  \beta_{i,0}\beta_{i,1}^{[1]} \cdots \beta_{i,m}^{[m]} \Big) \cdot v^{b(q)}.
\]
The hypothesis that
\(\{b \in \mathbf{Z}_{\geq 0}[t]: 0 \preceq b \preceq a\} \to \mathbf{Z} \colon b \mapsto b(q)\)
is injective implies that \(C = 0\) and that each parenthesized term must
vanish. Since \(\alpha\) is nonzero, some \(\beta_{i,j}\) is nonzero, and so
\(\mult \colon \mathrm{S}^{a - b}(U) \to \Sym^{a(q) - b(q)}(U)\) is not
injective for some \(b \prec a\). Multiplying by any element in
\(\mathrm{S}^b(U)\) with nonzero image in \(\Sym^{b(q)}(U)\) then implies
that \(\mult \colon \mathrm{S}^a(U) \to \Sym^{a(q)}(U)\) is also not injective.
This contradicts the minimality of \(V\), completing the proof.
\end{proof}

A simple but very useful consequence of this characterization used tacitly
throughout---in \parref{fano-pointed-line-equations} for instance---is the
following

\begin{Corollary}\label{profiles-less-than}
Let \(a, b \in \mathbf{Z}_{\geq 0}[t]\) with \(a \preceq b\). If \(b \in \Types\),
then also \(a \in \Types\). \qed
\end{Corollary}

\section{\texorpdfstring{\((q;\mathbf{a})\)}{(q;a)}-tic schemes}\label{section-qatic}
Each \(a \in \Types\) determines a canonical linear system
\[
\Gamma(Y,\mathcal{L}^{\otimes a})
\coloneqq
\image(
\mult \colon
\mathrm{S}^a\Gamma(Y,\mathcal{L}) \to
\Gamma(Y,\mathcal{L}^{\otimes a(q)}))
\]
for every \(\kk\)-scheme \(Y\) and \(\mathcal{L} \in \Pic Y\); note that
for general \((Y,\mathcal{L})\), the displayed multiplication map may not be
injective. Divisors in the corresponding linear series tend to be special
compared to the general divisor in the complete linear series associated with
\(\mathcal{L}^{\otimes a(q)}\) due to the special form of its equations. One
way to access these special properties is by carrying as additional structure
a lift of its defining section to the tensor product space
\(\mathrm{S}^a\Gamma(Y,\mathcal{L})\). This article is concerned primarily with
the case of projective \(n\)-space \(Y = \PP V\) and
\(\mathcal{L} = \sO_{\PP V}(1)\), and this section describes some of the basic
properties of complete intersections of such hypersurfaces.

\subsectiondash{Definitions}\label{qatics-definition}
Let \(a \in \Types\) be a profile.
\begin{itemize}
\item A \emph{\((q;a)\)-tic tensor} is an element
\(\alpha \in \mathrm{S}^a(V^\vee)\).
\item A \emph{\((q;a)\)-tic polynomial} is the image
\(f_\alpha \coloneqq \mult(\alpha)\) of a \((q;a)\)-tic tensor under
multiplication.
\item A \emph{\((q;a)\)-tic hypersurface} is the zero locus
\(X_\alpha \coloneqq \mathrm{V}(f_\alpha)\)
in \(\PP V\) of a nonzero \((q;a)\)-tic polynomial.
\end{itemize}
More generally, given a \emph{multi-profile}
\(\mathbf{a}\)---a finite multi-set consisting of elements of \(\Types\)---let
\[
\mathrm{S}^{\mathbf{a}}(V^\vee) \coloneqq
\bigoplus\nolimits_{a \in \mathbf{a}} \mathrm{S}^a(V^\vee)
\]
be the space of \emph{\((q;\mathbf{a})\)-tic tensors}. Given such a tensor
\(\boldsymbol{\alpha}\), the vanishing locus
\(X_{\boldsymbol{\alpha}} \coloneqq \mathrm{V}(f_\alpha : \alpha \in \boldsymbol{\alpha}) \subseteq \PP V\)
of its associated polynomials is called a
\emph{\((q;\mathbf{a})\)-tic scheme}; when the sequence of defining polynomials
form a regular sequence, \(X_{\boldsymbol{\alpha}}\) is called a
\emph{\((q;\mathbf{a})\)-tic complete intersection}.

\subsectiondash{Examples}\label{qatics-examples}
Given a profile \(a = a_0 + a_1 t + \cdots + a_m t^m\), a \((q;a)\)-tic
polynomial is of the form
\[
f = \sum\nolimits_{i \in I} g_{i,0} \cdot (g_{i,1})^q \cdots (g_{i,m})^{q^m}
\]
where \(g_{i,j}\) is a homogeneous polynomial of degree \(a_j\). Some simple
instances include:
\begin{enumerate}
\item\label{qatics-examples.constant}
The constant polynomial \(a(t) = d\) is a profile for any prime power \(q\),
and \((q;a)\)-tic polynomial is synonymous for degree \(d\) polynomial. In
particular, if \(a(t) = 1\) is a \emph{linear} profile, then \((q;a)\)-tic
polynomials are linear polynomials.
\item\label{qatics-examples.nonreduced}
If a profile \(a(t) = a_1t + \cdots + a_mt^m\) has zero constant term, so it is
\emph{nonreduced}, then any \((q;a)\)-tic polynomial is a
\(\kk\)-linear combination of \(q\)-power monomials, and its associated
hypersurface is geometrically nonreduced.
\item\label{qatics-examples.qbic}
\(a(t) = 1 + t + \cdots + t^m\) is a profile for any \(q\), and
\((q;a)\)-tic polynomials are of the form
\[
f(x_0,\ldots,x_n) =
\sum\nolimits_{i_0,\ldots,i_m = 0}^n c_{i_0,i_1,\ldots,i_m} \cdot
\prod\nolimits_{j = 0}^m (x_{i_j})^{q^j}
\;\;
\text{for scalars \(c_{i_0,i_1,\ldots,i_m} \in \kk\)}.
\]
In particular, when \(a(t) = 1 + t\), these are the \(q\)-bic hypersurfaces
from \citeFano{1.6}, and the underlying \((q;a)\)-tic tensor is but a \(q\)-bic
form in the sense of \citeForms{1.1}.
\item\label{qatics-examples.cubic}
\(a(t) = 2 + t\) is a profile whenever \(q \neq 2\), and \((q;a)\)-tic polynomials
are of the form
\[
f(x_0,\ldots,x_n) =
\sum\nolimits_{0 \leq i \leq j \leq n}
\sum\nolimits_{k = 0}^n c_{ijk} \cdot x_i x_j x_k^q
\;\;
\text{for scalars \(c_{ijk} \in \kk\)}.
\]
\item\label{qatics-examples.fermat}
The Fermat polynomial \(x_0^d + \cdots + x_n^d\) of degree \(d\) is
a \((q;a)\)-tic polynomial for any profile \(a\) of numerical degree
\(a(q) = d\). This provides a cheap proof of the fact that the general
\((q;a)\)-tic hypersurface is smooth whenever its numerical degree is coprime
to the characteristic \(p\); see \parref{qatic-general-is-smooth} for another
argument in the general case.
\end{enumerate}

A simple but extremely useful observation is that a \((q;\mathbf{a})\)-tic
structure is preserved upon passing to linear sections. The following statement
follows easily from functoriality of tensor functors:

\begin{Lemma}\label{qatic-linear-sections}
Let \(X \subseteq \PP V\) be a \((q;\mathbf{a})\)-tic scheme associated to a
tensor \(\boldsymbol{\alpha}\). If \(\PP U \subseteq \PP V\) is a linear
subspace, then \(X \cap \PP U\) is the \((q;\mathbf{a})\)-tic scheme
in \(\PP U\) associated with the tensor \(\boldsymbol{\alpha}\rvert_U\).
\qed
\end{Lemma}

\subsectiondash{\texorpdfstring{\((q;a)\)}{(q;a)}-tic Veronese}\label{qatic-veronese}
The linear system \(\mathrm{S}^a(V^\vee)\) of \((q;a)\)-tic polynomials, being
a tensor product of Frobenius twists of the complete linear systems
\(\Gamma(\PP V, \sO_{\PP V}(a_j))\), is base point free and so defines a
morphism
\[
\mathrm{Ver}_a \colon \PP V \to
\PP(\mathrm{S}^a(V^\vee)^\vee)
\]
called the \emph{\((q;a)\)-tic Veronese morphism}. This description shows that
\(\mathrm{Ver}_a\) canonically factors as
\[
\mathrm{Ver}_a \colon
\PP V \stackrel{(\mathrm{Ver}_{a_j})}{\longrightarrow}
\prod\nolimits_{j \geq 0} \PP(\Sym^{a_j}(V^\vee)^\vee) \stackrel{\prod \Fr^j}{\longrightarrow}
\prod\nolimits_{j \geq 0} \PP(\Sym^{a_j}(V^\vee)^{\vee,[j]}) \stackrel{\mathrm{Seg}}{\longrightarrow}
\PP(\mathrm{S}^a(V^\vee)^\vee)
\]
where the first map is the tuple whose \(j\)-th factor is the
\(a_j\)-th Veronese embedding; the second map is a product of powers of the
\(q\)-power \(\kk\)-linear Frobenius morphism, with \(\Fr^j\) acting on the
\(j\)-th factors; and the third map is the Segre embedding. This gives the
first statement of:

\begin{Lemma}\label{qatic-linear-system}
The \((q;a)\)-tic Veronese morphism
\(\mathrm{Ver}_a \colon \PP V \to \PP(\mathrm{S}^a(V^\vee)^\vee)\) is
universally injective. Furthermore, if \(a\) is reduced, then
\(\mathrm{Ver}_a\) is a closed immersion.
\end{Lemma}

\begin{proof}
Assume \(a\) is reduced, meaning its constant term satisfies \(a_0 \neq 0\),
and consider the factorization from \parref{qatic-veronese}. Since the Segre
embedding is a closed immersion, it suffices to show that
\[
\big(\prod\nolimits_{j \geq 0} \Fr^j\big) \circ
\big(\mathrm{Ver}_{a_j}\big)_{j \geq 0} \colon 
\PP V \to \prod\nolimits_{j \geq 0} \PP(\Sym^{a_j}(V^\vee)^{\vee,[j]})
\]
is an isomorphism onto its image. But this clear since projection onto the
\(0\)-th factor provides the \(a_0\)-th Veronese embedding of \(\PP V\), which
is an isomorphism onto its image whenever \(a_0 \neq 0\).
\end{proof}

\subsectiondash{Example}\label{qatic-twisted-qbic}
Consider \(a = t^r(1  + t^m)\) for integers \(r \geq 0\) and \(m > 0\), and
\(\PP V \cong \PP^1\). Then the \(a\)-th Veronese is the morphism
\(\PP^1 \to \PP^3\) given by
\[
\mathrm{Ver}_{t^r(1 + t^m)}(x:y) =
(x^{q^r(1 + q^m)}: (x^{q^m} y)^{q^r} : (x y^{q^m})^{q^r} : y^{q^r(1+q^m)}).
\]
When \(r = 0\), then \(\mathrm{Ver}_{1 + t^m} \colon \PP^1 \to \PP^3\) is an
isomorphism onto its image, providing what might be viewed as a generalization
of the twisted cubic, which may be obtained by taking \(q = 2\) and \(m = 1\).

\subsectiondash{Parameter space}\label{qatic-parameter-space}
Let \(\mathbf{a}\) be a multi-profile. A parameter space for
\((q;\mathbf{a})\)-tic schemes in \(\PP V\) is given by the multi-projective
space
\[
\qaticci_{\PP V} \coloneqq
\prod\nolimits_{a \in \mathbf{a}} \PP\mathrm{S}^a(V^\vee).
\]
The tautological line subbundles come together to form a tautological
\((q;\mathbf{a})\)-tic tensor
\[
\boldsymbol{\alpha}_{\mathrm{taut}} \colon
\bigoplus\nolimits_{a \in \mathbf{a}} \pr_a^*\sO_{\PP\mathrm{S}^a(V^\vee)}(-1) \to
\sO_{\qaticci_{\PP V}} \otimes \mathrm{S}^{\mathbf{a}}(V^\vee)
\]
which cuts out a tautological family \(\mathcal{X}_{\mathrm{taut}}\) in the
product \(\PP V \times \qaticci_{\PP V}\).

Standard arguments show that the second projection
\(\mathcal{X}_{\mathrm{taut}} \to \qaticci_{\PP V}\) is dominant if and only if
\(\#\mathbf{a} \leq n\). In this case, as usual, a property is said to hold
for a \emph{general} \((q;\mathbf{a})\)-tic scheme in \(\PP V\) if it holds for
each fibre of \(\mathcal{X}_{\mathrm{taut}}\) over a nonempty open subset of
\(\qaticci_{\PP V}\). In the following, call the multi-profile \(\mathbf{a}\)
\emph{reduced} if every \(a \in \mathbf{a}\) is reduced:

\begin{Proposition}\label{qatic-general-is-smooth}
If the multi-profile \(\mathbf{a}\) is reduced and \(\#\mathbf{a} \leq n\),
then a general \((q;\mathbf{a})\)-tic scheme in \(\PP V\) is smooth of
dimension \(n - \#\mathbf{a}\).
\end{Proposition}

\begin{proof}
Let \(\Delta \subseteq \qaticci_{\PP V}\) be the subset parameterizing
\((q;\mathbf{a})\)-tic schemes which are not smooth of dimension
\(n - \#\mathbf{a}\). This is closed subset as it is supported on the
image of the nonsmooth locus of the proper morphism
\(\mathcal{X}_{\mathrm{taut}} \to \qaticci_{\PP V}\). Thus it suffices to
show that there exists a single smooth \((q;\mathbf{a})\)-tic complete
intersection. This follows from Bertini's theorem as given in \citeSP{0FD6},
where the hypotheses are satisfied thanks to \parref{qatic-veronese} and
\parref{qatic-linear-system}.
\end{proof}

It would be interesting to study the discriminant locus \(\Delta\). For
\(q\)-bic hypersurfaces as in
\parref{qatics-examples}\ref{qatics-examples.qbic}, \(\Delta\) is cut out by a
determinantal equation, similar to the case of quadrics, and carries an
intricate stratification corresponding to singularity types of the
corresponding hypersurfaces: see \cite[\S6]{qbic-forms} for details. It would
also be interesting to compute numerical invariants of \(\Delta\), such as its
multi-degree, as is done in \cite{Benoist:CI-I} in the classical case.

\section{Linear spaces}\label{section-fano}
One of the most apparent special properties of a \((q;\mathbf{a})\)-tic
complete intersection \(X \subseteq \PP V\) is that it
contains many more linear spaces than would be expected given its numerical
degree. From the point of view of this article, this is because its
\emph{Fano schemes} \(\mathbf{F}_r(X)\)---the Hilbert schemes parameterizing
\(r\)-planes contained in \(X\)---inherit a \((q;\mathbf{a})\)-tic structure,
resulting in a smaller set of defining equations in the Grassmannian
\(\mathbf{G} \coloneqq \mathbf{G}(r+1,V)\) of \(r\)-planes in \(\PP V\).
Writing \(\mathcal{S}\) for the tautological subbundle of rank \(r+1\) on
\(\mathbf{G}\), this explicitly means the following:

\begin{Lemma}\label{fano-equations}
Let \(X \subseteq \PP V\) be a \((q;\mathbf{a})\)-tic scheme. Its
Fano scheme \(\mathbf{F}_r(X)\) of \(r\)-planes is cut out of the Grassmannian
\(\mathbf{G}\) by a section of
\(\mathrm{S}^{\mathbf{a}}(\mathcal{S}^\vee) \coloneqq
\bigoplus\nolimits_{a \in \mathbf{a}} \mathrm{S}^a(\mathcal{S}^\vee)\).
\end{Lemma}

\begin{proof}
Injectivity of the multiplication maps in the definition of profiles in
\parref{qatics-profiles} implies that the polynomials defining \(X\) vanish on
an \(r\)-plane \(\PP U \subseteq \PP V\) if and only if the corresponding
\((q;\mathbf{a})\)-tic tensor
\(\boldsymbol{\alpha} \in \mathrm{S}^{\mathbf{a}}(V^\vee)\)
vanishes along the restriction \(V^\vee \to U^\vee\). Thus
\(\mathbf{F}_r(X)\) is cut out by the section
\[
\boldsymbol{\alpha}\rvert_{\mathcal{S}} \colon
\sO_{\mathbf{G}} \to
\sO_{\mathbf{G}} \otimes \mathrm{S}^{\mathbf{a}}(V^\vee) \to
\mathrm{S}^{\mathbf{a}}(\mathcal{S}^\vee).
\qedhere
\]
\end{proof}

% \subsectiondash{Remark}\label{fano-warning}
% The argument tacitly uses the assumption from \parref{qatics-profiles} that
% multiplication maps associated with profiles are injective. In general,
% \(\PP U \subseteq X_{\boldsymbol{\alpha}}\) only implies that
% \(\boldsymbol{\alpha}\rvert_U\) is annihilated by
% the multiplication map
% \(\mult \colon \mathrm{S}^{\mathbf{a}}(U^\vee) \to \bigoplus\nolimits_{a \in \mathbf{a}} \Sym^{a(q)}(U^\vee)\).

This provides a lower bound on the dimension of the Fano scheme
\(\mathbf{F}_r(X)\) of \(r\)-planes associated with a \((q;\mathbf{a})\)-tic
scheme \(X \subseteq \PP V\) in an \(n\)-dimensional projective space: If
\(\mathbf{F}_r(X) \neq \varnothing\), then 
\[
\dim \mathbf{F}_r(X) \geq
\delta(n,\mathbf{a},r) \coloneqq
(r+1)(n-r) -
\sum\nolimits_{a \in \mathbf{a}} \prod\nolimits_{j \geq 0} \binom{a_j + r}{r}
\;\;\text{where}\;
a = \sum\nolimits_{j \geq 0} a_j t^j.
\]
Notably, the \emph{expected dimension} \(\delta(n,\mathbf{a},r)\) of the Fano
scheme, depends on \(\mathbf{a}\) but not on \(q\); equivalently, this means
that the expected dimension of \(\mathbf{F}_r(X)\) does not directly depend on
the degree of \(X\)!

The main result of this section is that \(\mathbf{F}_r(X)\) is of its expected
dimension for a general \((q;\mathbf{a})\)-tic scheme \(X\), and will be
furthermore smooth provided that the multi-profile \(\mathbf{a}\) is reduced.
As is well-known, the statement requires a slight modification when \(X\)
is a quadric, in which case the dimension estimate \(\delta(n,\mathbf{a},r)\)
is too big when \(r \approx \frac{1}{2}n\). To give a uniform
statement, define
\[
\delta_-(n,\mathbf{a},r) \coloneqq
\min\{\delta(n,\mathbf{a},r), n - 2r - \#\mathbf{a}\}.
\]
One may verify that, other than in the case \(\mathbf{a} = (2t^k) \cup \mathbf{a}'\)
where each \(a \in \mathbf{a}'\) is of the form \(t^m\),
\(\delta(n,\mathbf{a},r) \geq 0\) if and only if
\(\delta_-(n,\mathbf{a},r) \geq 0\),
and similarly for \(> 0\). The statement is now the following:

\begin{Theorem}\label{fano-theorem}
Let \(X \subseteq \PP V\) be a \((q;\mathbf{a})\)-tic scheme. If
\begin{enumerate}
\item\label{fano-theorem.empty}
\(\delta_-(n,\mathbf{a},r) < 0\), then \(\mathbf{F}_r(X)\) is empty for
general \(X\);
\item\label{fano-theorem.smooth}
\(\delta_-(n,\mathbf{a},r) \geq 0\), then \(\mathbf{F}_r(X)\) is
nonempty and has dimension \(\delta(n,\mathbf{a},r)\) for general \(X\);
and
\item\label{fano-theorem.connected}
\(\delta_-(n,\mathbf{a},r) > 0\), then \(\mathbf{F}_r(X)\) is connected
for all \(X\).
\end{enumerate}
Furthermore, if \(\mathbf{a}\) is reduced, then \(\mathbf{F}_r(X)\) is also
smooth for general \(X\) when \(\delta_-(n,\mathbf{a},r) \geq 0\).
\end{Theorem}

The proof of \parref{fano-theorem} occupies the bulk of this section: see, in
particular, \parref{fano-theorem-proof} where the intervening statements are
put together to complete the argument. Begin with two reductions:

\subsectiondash{Reductions}\label{fano-reductions}
First, it suffices to treat the case \(\mathbf{a}\) is reduced. If \(\mathbf{a}\)
is nonreduced, let \(\mathbf{a}'\) be the reduced multi-profile obtained by
maximally dividing out powers of \(t\) from each profile in \(\mathbf{a}\).
Given a \((q;\mathbf{a})\)-tic scheme \(X\), there is a canonical
\((q;\mathbf{a}')\)-tic scheme \(X'\) obtained by taking \(q\)-power roots of
the appropriate equations. This comes with a closed immersion \(X' \to X\)
which is furthermore a universal homeomorphism. This provides a universal
homeomorphism \(\mathbf{F}_r(X') \to \mathbf{F}_r(X)\) and so topological
properties of the two are the same.

Second, we may assume that \(1 \notin \mathbf{a}\). Otherwise, amongst
the defining equations of a \((q;\mathbf{a})\)-tic
scheme \(X \subseteq \PP V = \PP^n\) is a linear one. Eliminating that allows
us to view \(X \subseteq \PP^{n-1}\). Writing
\(\mathbf{a} = (1) \cup \mathbf{a}'\), a direct computation shows that
\(\delta_-(n,\mathbf{a},r) = \delta_-(n-1,\mathbf{a}',r)\), and so it suffices
to prove \parref{fano-theorem} for \(X\) viewed as a \((q;\mathbf{a}')\)-tic
scheme in \(\PP^{n-1}\).

\medskip
From now on, assume that the profile \(\mathbf{a}\) is reduced and
\(1 \notin \mathbf{a}\). The argument looks to lower bound the codimension of
the singular locus \(Z_r\) of the second projection from the incidence
correspondence
\[
\mathbf{Inc}_r \coloneqq
\mathbf{Inc}_{V,r,\mathbf{a}} \coloneqq
\big\{
  ([U], [\boldsymbol{\alpha}]) \in \mathbf{G}(r+1,V) \times \qaticci_{\PP V} :
  \PP U \subseteq X_{\boldsymbol{\alpha}}
\big\}
\]
to the parameter space of \((q;\mathbf{a})\)-tic tensors from
\parref{qatic-parameter-space}. Observe that
\(\pr_1 \colon \mathbf{Inc}_r \to \mathbf{G}(r+1,V)\) is surjective,
with the fibre over a point \([U]\) a multi-projective space given by
\[
\mathbf{Inc}_{r,[U]} =
\prod\nolimits_{a \in \mathbf{a}} \PP\mathrm{S}^a(V^\vee)_U
\]
where
\(\mathrm{S}^a(V^\vee)_U \coloneqq \ker(\mathrm{S}^a(V^\vee) \to \mathrm{S}^a`(U^\vee))\)
is the kernel of the restriction map. This has codimension
\(\sum\nolimits_{a \in \mathbf{a}} \prod_{j \geq 0} \binom{r + a_j}{r}\) in
the product and so gives:

\begin{Lemma}\label{fano-incidence-correspondence}
The incidence correspondence
\(\mathbf{Inc}_r\) is irreducible, proper, and smooth of dimension
\[
\pushQED{\qed}
\dim\mathbf{Inc}_r = \delta(n,\mathbf{a},r) + \dim\,\qaticci_{\PP V}.
\qedhere
\popQED
\]
\end{Lemma}

The first task is to explicitly describe the closed subset
\(Z_r \subseteq \mathbf{Inc}_r\) on which
\(\pr_2 \colon \mathbf{Inc}_r \to \qaticci_{\PP V}\) is
not smooth of expected dimension \(\delta(n,\mathbf{a},r)\): Let \(\PP U\) be
any \(r\)-plane contained in \(X \coloneqq X_{\boldsymbol{\alpha}}\). Write
\(\mathbf{a} = (a_1,\ldots,a_c)\) as a \(c\)-tuple of profiles
\(a_i = \sum\nolimits_{j \geq 0} a_{i,j} t^j\), and let
\(\alpha_i \in \mathrm{S}^{a_i}(V^\vee)\) be the components of the
\((q;\mathbf{a})\)-tic tensor \(\boldsymbol{\alpha}\) defining \(X\). Consider
the map
\begin{align*}
\rho_{U,f_{\boldsymbol{\alpha}}} \colon
  U^\vee \otimes (V/U) & \to
  \bigoplus\nolimits_{i = 1}^c \mathrm{H}^0(\PP U, \sO_{\PP U}(a_i(q))) \\
  \xi \otimes \bar{v} & \mapsto
  (
  \xi \cdot \partial_v f_{\alpha_1}\rvert_{\PP U},
  \ldots,
  \xi \cdot \partial_v f_{\alpha_c}\rvert_{\PP U}
  )
\end{align*}
which takes a pure tensor \(\xi \otimes \bar{v}\) to the \(c\)-tuple of
polynomials whose \(i\)-th entry is \(\xi\) times the directional derivative
\(\partial_v f_{\alpha_i}\) of the \(i\)-th equation of \(X\) along any lift
\(v \in V\) of \(\bar{v}\), all then restricted to \(\PP U\).
Since first-order derivatives act linearly through \(q\)-powers,
\(\rho_{U,f_{\boldsymbol{\alpha}}}\) in fact takes values within the subspace
of \((q;\mathbf{a})\)-tic polynomials. More precisely:

\begin{Lemma}\label{fano-tangent-map}
There exists a linear map
\(\rho_{U,\boldsymbol{\alpha}} \colon U^\vee \otimes (V/U) \to \mathrm{S}^{\mathbf{a}}(U^\vee)\)
which factors \(\rho_{U,f_{\boldsymbol{\alpha}}}\) through the multipication map
\(\mathrm{mult} \colon \mathrm{S}^{\mathbf{a}}(U^\vee) \to \bigoplus\nolimits_{i = 1}^c \mathrm{H}^0(\PP U, \sO_{\PP U}(a_i(q)))\).
\end{Lemma}

\begin{proof}
The map \(\rho_{U,\boldsymbol{\alpha}}\) is that which simply acts on the first
component of each tensor \(\alpha_i\). Explicitly, the pure tensor
\(\xi \otimes \bar{v}\) is mapped to the \(c\)-tuple with \(i\)-th term
\[
\xi \cdot \partial_v \alpha_i \rvert_{\PP U} =
\sum\nolimits_k
  (\xi \cdot \partial_v \alpha_{i,0,k}) \otimes
  (\alpha_{i,1,k})^{[1]} \otimes
  \cdots \otimes
  (\alpha_{i,m_i,k})^{[m_i]}\rvert_{\PP U}
\]
where each \(\alpha_{i,j,k} \in \Sym^{a_{i,j}}(V^\vee)\). The preceding
comments imply that
\(\rho_{U,f_{\boldsymbol{\alpha}}} = \operatorname{mult} \circ \rho_{U,\boldsymbol{\alpha}}\).
\end{proof}

If \(\PP U\) were contained in the smooth locus of \(X\), then the tangent
space to \(\mathbf{F}_r(X)\) at the point \([\PP U]\) is given by the space
of sections the normal bundle of \(\PP U \subseteq X\). The normal bundle
sequence, \parref{fano-tangent-map}, and injectivity of \(\mult\) together
given canonical identifications
\[
\mathrm{H}^0(\PP U, \mathcal{N}_{\PP U/X}) \cong
\ker\rho_{U,f_{\boldsymbol{\alpha}}} \cong
\ker\rho_{U,\boldsymbol{\alpha}}.
\]
Therefore \(\mathbf{F}_r(X)\) is smooth of dimension
\(\delta(n,\mathbf{a},r)\) at the point \([\PP U]\) if and only if the map
\(\rho_{U,\boldsymbol{\alpha}}\) is surjective. This property characterizes the
complement of \(Z_r\) in general:

\begin{Lemma}\label{fano-nonsmooth-locus}
\(\pr_2 \colon \mathbf{Inc}_r \to \qaticci_{\PP V}\) is smooth of relative
dimension \(\delta(n,\mathbf{a},r)\) at a point \(([U],[\boldsymbol{\alpha}])\)
if and only if
\(\rho_{U,\boldsymbol{\alpha}} \colon U^\vee \otimes (V/U) \to \mathrm{S}^{\mathbf{a}}(U^\vee)\)
is surjective. In particular,
\[
Z_r =
\big\{([U], [\boldsymbol{\alpha}]) \in \mathbf{Inc}_r :
\rho_{U,\boldsymbol{\alpha}} \colon U^\vee \otimes (V/U) \to \mathrm{S}^{\mathbf{a}}(U^\vee)
\;\text{is \emph{not} surjective}
\big\}.
\]
\end{Lemma}

\begin{proof}
Both \(\mathbf{Inc}_r\) and \(\qaticci_{\PP V}\) are smooth, so \(\pr_2\) is
smooth of relative dimension \(\delta(n,\mathbf{a},r)\) at
\(([U],[\boldsymbol{\alpha}])\) if and only if the map
\(\mathcal{T}_{\mathbf{Inc}_r} \to \pr_2^*\mathcal{T}_{\qaticci_{\PP V}}\)
on tangent bundles is surjective there. The discussion above
\parref{fano-incidence-correspondence} shows that the fibre of \(\pr_1\)
over \([U]\) maps isomorphically via \(\pr_2\) to a multi-projective
subspace in the parameter space corresponding to the vector space \(\mathrm{S}^{\mathbf{a}}(V^\vee)_U\):
\[
\pr_1^{-1}([U]) \cong
\prod\nolimits_{i = 1}^c \PP\big(\ker(\mathrm{S}^{a_i}(V^\vee) \to \mathrm{S}^{a_i}(U^\vee))\big) \subseteq
\prod\nolimits_{i = 1}^c \PP\mathrm{S}^{a_i}(V^\vee)
= \qaticci_{\PP V}
\]
Restricting the tangent map to the subbundle
\(\mathcal{T}_{\pr_1} \subseteq \mathcal{T}_{\mathbf{Inc}_r}\) and taking
fibres thus gives a sequence
\[
0 \to
\mathcal{T}_{\pr_1, ([U],[\boldsymbol{\alpha}])} \to
\mathcal{T}_{\qaticci_{\PP V}, [\boldsymbol{\alpha}]} \to
\mathrm{S}^{\mathbf{a}}(U^\vee) \to
0.
\]
Combined with the isomorphism
\(\pr_1^*\mathcal{T}_{\mathbf{G}} \cong \mathcal{T}_{\mathbf{Inc}_r}/\mathcal{T}_{\pr_1}\)
from the tangent bundle sequence of \(\pr_1\), this means that the tangent
map of \(\pr_2\) at \(([U],[\boldsymbol{\alpha}])\) induces a map
\begin{equation}\label{fano-nonsmooth-locus.map}
U^\vee \otimes (V/U) \cong
\mathcal{T}_{\mathbf{Inc}_r,([U],[\boldsymbol{\alpha}])}/\mathcal{T}_{\pr_1,([U],[\boldsymbol{\alpha}])} \to
\mathcal{T}_{\qaticci_{\PP V},[\boldsymbol{\alpha}]}/\mathcal{T}_{\pr_1,([U],[\boldsymbol{\alpha}])} \cong
\mathrm{S}^{\mathbf{a}}(U^\vee)
\tag{\(\star\)}
\end{equation}
and \(\pr_2\) is smooth of the expected dimension at
\(([U],[\boldsymbol{\alpha}])\) if and only if this map is surjective.

Identify this map via deformation theory: A pure tensor \(\xi \otimes \bar{v}\)
in \(U^\vee \otimes (V/U) \cong \mathcal{T}_{\mathbf{G},[U]}\) corresponds to
the first-order deformation of \(U \subseteq V\) given by the
\(\kk[\epsilon]\)-submodule
\[
U[\epsilon\xi \cdot \bar{v}] \coloneqq
\langle u + \epsilon\xi(u) \cdot v : u \in U\rangle
=
\{ u_1 + \epsilon(\xi(u_1) \cdot v + u_2) : u_1, u_2 \in U\}
\subseteq V \otimes_\kk \kk[\epsilon].
\]
Its preimage in \(\mathcal{T}_{\mathbf{Inc}_r,([U],[\boldsymbol{\alpha}])}\)
classifies first-order deformations
\(X_{\boldsymbol{\alpha} + \epsilon\boldsymbol{\beta}}\), where
\(\boldsymbol{\beta} \in \mathrm{S}^{\mathbf{a}}(V^\vee)\), which contain
\(\PP(U[\epsilon \xi \cdot \bar{v}])\). The tangent map
\(\mathcal{T}_{\mathbf{Inc}_r} \to \mathcal{T}_{\qaticci_{\PP V}}\)
is the forgetful map which extracts the \((q;\mathbf{a})\)-tic tensor
\(\boldsymbol{\beta}\) parameterizing the first-order deformation of
\(X_{\boldsymbol{\alpha}}\). The map \eqref{fano-nonsmooth-locus.map} thus
acts as \(\xi \otimes \bar{v} \mapsto \boldsymbol{\beta}\rvert_U\) for any
choice of such \(\boldsymbol{\beta}\). To express
\(\boldsymbol{\beta}\) in terms of \(\boldsymbol{\alpha}\), observe that
the condition that \(X_{\boldsymbol{\alpha} + \epsilon\boldsymbol{\beta}}\)
contains \(\PP(U[\epsilon\xi \cdot \bar{v}])\) means that, for all
\(u_1,u_2 \in U\),
\[
(\boldsymbol{\alpha} + \epsilon\boldsymbol{\beta})(u_1 + \epsilon(\xi(u_1) \cdot v + u_2))
= \boldsymbol{\alpha}(u_1 + \epsilon(\xi(u_1) \cdot v + u_2)) + \epsilon\boldsymbol{\beta}(u_1) = 0.
\]
Writing
\(\alpha_i = \sum_k
  \alpha_{i, 0, k} \otimes
  (\alpha_{i, 1, k})^{[1]} \otimes \cdots \otimes
  (\alpha_{i, m_i, k})^{[m_i]}\)
for each component of \(\boldsymbol{\alpha}\),
\[
\alpha_{i,j,k}^{[j]}(u_1 + \epsilon(\xi(u_1) \cdot v + u_2)) =
\begin{dcases*}
\alpha_{i,j,k}(u_1)^{q^j} & if \(j > 0\), and \\
\alpha_{i,0,k}(u_1) + \epsilon (\xi(u) \cdot \partial_v \alpha_{i,0,k}(u) + \partial_{u_2}\alpha_{i,0,k}(u)) & if \(j = 0\).
\end{dcases*}
\]
Expanding the tensor then gives
\[
\alpha_i(u_1 + \epsilon(\xi(u_1) \cdot v + u_2)) =
\alpha_i(u_1) + \epsilon(\xi(u_1) \cdot \partial_v\alpha_i(u_1) + \partial_{u_2}\alpha_i(u_1)).
\]
That \(\boldsymbol{\alpha}\rvert_U = 0\) means that the first term vanishes:
\(\alpha_i(u_1) = 0\) for all \(u_1 \in U\). A directional derivative of
a polynomial vanishing on \(U\) in a direction in \(U\) remains vanishing on
\(U\), so \(\partial_{u_2}\alpha_i(u_1) = 0\) for all \(u_1,u_2 \in U\). Put
together, this gives the result since
\[
\boldsymbol{\beta}(u) =
-\xi(u) \cdot \partial_v\boldsymbol{\alpha}(u) =
-\rho_{U,\boldsymbol{\alpha}}(\xi \otimes \bar{v})(u)
\;\;\text{for all \(u \in U\)}.
\qedhere
\]
\end{proof}

This description of \(Z_r\) is homogeneous in \([U]\), making it possible to
fix a subspace \(U\) and to simply estimate the codimension of the inclusion of
fibres \(Z_{r,[U]} \subset \mathbf{Inc}_{r,[U]}\). Proceed by additionally
parameterizing a bound on the image of \(\rho_{U,\boldsymbol{\alpha}}\):
Let \(\mathbf{H}\) be the projective space of hyperplanes in
\(\mathrm{S}^{\mathbf{a}}(U^\vee)\), and consider the correspondence
\[
Z_{r,[U]}' \coloneqq
\big\{([\boldsymbol{\alpha}],[\varphi]) \in \mathbf{Inc}_{r,[U]} \times \mathbf{H} :
\varphi \circ \rho_{U,\boldsymbol{\alpha}} \colon U^\vee \otimes (V/U) \to \kk
\;\text{is zero}
\big\}
\]
parameterizing \(X_{\boldsymbol{\alpha}}\) containing \(\PP U\) and a
hyperplane in \(\mathrm{S}^{\mathbf{a}}(U^\vee)\) containing the image of
\(\rho_{U,\boldsymbol{\alpha}}\). Projection to \(\mathbf{Inc}_{r,[U]}\)
maps this onto \(Z_{r,[U]}\). The fibre of projection to \(\mathbf{H}\) over
a point \(\varphi \colon \mathrm{S}^{\mathbf{a}}(U^\vee) \to \kk\) is a
product of projective spaces on the kernel of the map
\[
\Phi \colon \mathrm{S}^{\mathbf{a}}(V^\vee)_U  \to \Hom(U^\vee \otimes (V/U), \kk)
\quad\quad
\boldsymbol{\alpha} \mapsto \varphi \circ \rho_{U,\boldsymbol{\alpha}}.
\]
Writing \(\mathbf{a} - 1 = (a_1 - 1, \ldots, a_c - 1)\), the following gives a
stratification of the image \(Z_{r,[U]}'\) in \(\mathbf{H}\) over which the
fibres have the same dimension:

\begin{Lemma}\label{fano-H}
Let
\(\mu \colon \mathrm{S}^{\mathbf{a}-1}(U^\vee) \to \Hom(U^\vee, \mathrm{S}^{\mathbf{a}}(U^\vee))\)
be the map adjoint to multiplication, and set
\[
\mathbf{H}_k \coloneqq
\big\{
[\varphi] \in \mathbf{H} :
\rank\big(
  \varphi_* \circ \mu \colon
  \mathrm{S}^{\mathbf{a} - 1}(U^\vee) \to
  \Hom(U^\vee,\kk)\big) = k+1
\big\}\;\text{for each}\; 0 \leq k \leq r.
\]
Then \(Z_{r,[U]}' \times_{\mathbf{H}} \mathbf{H}_k\) is of
codimension \((k+1)(n-r)\) in \(\mathbf{Inc}_{r,[U]} \times \mathbf{H}_k\).
\end{Lemma}

\begin{proof}
Identify the assignment
\(\boldsymbol{\alpha} \mapsto \rho_{U,\boldsymbol{\alpha}}\) as the composition
of linear maps
\[
\rho_{U,-} \colon
\mathrm{S}^{\mathbf{a}}(V^\vee)_U \longrightarrow
\Hom(V/U, \mathrm{S}^{\mathbf{a}-1}(U^\vee)) \stackrel{\mu_*}{\longrightarrow}
\Hom(V/U, \Hom(U^\vee,\mathrm{S}^{\mathbf{a}}(U^\vee)))
\]
where the first arrow sends \(\boldsymbol{\alpha}\) to the linear map \(\bar{v}
\mapsto \partial_v \boldsymbol{\alpha}\rvert_U\), notation as in the proof of
\parref{fano-tangent-map}. Observe also that the first map here is surjective,
since a lift of a given
\(\bar{\boldsymbol{\beta}} \in \Hom(V/U, \mathrm{S}^{\mathbf{a}-1}(U^\vee))\) is
\[
\sum\nolimits_{j = r+1}^n \xi_i \cdot \boldsymbol{\beta}(\bar{v}_i)
\in \mathrm{S}^{\mathbf{a}}(V^\vee)_U
\]
where the \(\bar{v}_i\) form a basis of \(V/U\) with dual coordinate \(\xi_i\),
and \(\boldsymbol{\beta}(\bar{v}_i) \in \mathrm{S}^{\mathbf{a} - 1}(V^\vee)\)
is any lift of its barred counterpart. Together, these observations imply that
\begin{align*}
\rank(\Phi) & =
\rank(\varphi_* \circ \mu_* \colon \Hom(V/U, \mathrm{S}^{\mathbf{a}-1}(U^\vee)) \to \Hom(V/U, \Hom(U^\vee,\kk))) \\
& = \rank(\varphi_* \circ \mu \colon \mathrm{S}^{\mathbf{a}-1}(U^\vee) \to \Hom(U^\vee,\kk)) \cdot \dim_\kk(V/U).
\end{align*}
This means that the fibre of \(Z_{r,[U]}'\) over
\([\varphi] \in \mathbf{H}_k\) is of codimension \((k+1)(n-r)\) in
\(\mathbf{Inc}_{r,[U]} \times \{[\varphi]\}\), and this yields the statement.
\end{proof}

To relate this with the codimension of \(Z_{r,[U]}\) in \(\mathbf{Inc}_{r,[U]}\),
it remains to bound the dimension of \(\mathbf{H}_k\):

\begin{Lemma}\label{fano-strata-codimension}
\(\displaystyle
\dim\mathbf{H}_k \leq
(k+1)(r-k) + 
\sum\nolimits_{i = 1}^c\prod\nolimits_{j = 0}^{m_i} \binom{k + a_{i,j}}{k} - 1
\).
\end{Lemma}

\begin{proof}
The image of \(\varphi_* \circ \mu\) is contained in the \((k+1)\)-dimensional
subspace \(\Hom(U_0^\vee,\kk)\) if and only if
\(\varphi \colon \mathrm{S}^{\mathbf{a}}(U^\vee) \to \kk\) vanishes on the
image of the multiplication map
\[
(U/U_0)^\vee \otimes \mathrm{S}^{\mathbf{a}-1}(U^\vee) \to
\mathrm{S}^{\mathbf{a}}(U^\vee).
\]
The cokernel of this map is isomorphic to \(\mathrm{S}^{\mathbf{a}}(U_0^\vee)\),
and any such \(\varphi\) is determined by its values thereon. In other words,
the closure \(\mathbf{H}_k\) admits a surjection from the projective bundle on
\(\mathrm{S}^{\mathbf{a}}(-)\) applied to the dual tautological bundle on the
Grassmannian \(\mathbf{G}(k+1,U)\), yielding the dimension bound.
\end{proof}

\begin{Proposition}\label{fano-main-codimension-estimate}
\(\displaystyle
\codim(Z_r \subset \mathbf{Inc}_r) \geq
\delta_-(n,\mathbf{a},r) + 1
\).
\end{Proposition}

\begin{proof}
It suffices to show that the corresponding codimension estimate holds for each
fibre \(Z_{r,[U]}\) and \(\mathbf{Inc}_{r,[U]}\) over points
\([U] \in \mathbf{G}(r+1,V)\). Since \(Z_{r,[U]}\) is the image of
\(Z_{r,[U]}'\) under
the first projection,
\begin{align*}
\codim(Z_{r,[U]} \subset \mathbf{Inc}_{r,[U]})
& \geq \min_{0 \leq k \leq r}
  \codim(Z_{r,[U]}' \times_{\mathbf{H}} \mathbf{H}_k \subset
         \mathbf{Inc}_{r,[U]} \times \mathbf{H}_k)
- \dim \mathbf{H}_k \\
& \geq \min_{0 \leq k \leq r}
  (k+1)(n+k-2r)
- \sum\nolimits_{i = 1}^c \prod\nolimits_{j = 0}^{m_i} \binom{k+a_{i,j}}{k} + 1
\end{align*}
where the second inequality follows from \parref{fano-H} and
\parref{fano-strata-codimension}. Now view the rightmost quantity as a
polynomial \(k\); it is the difference between a quadratic polynomial and one
of degree
\[
\norm{\mathbf{a}}_\infty \coloneqq
\max\{\abs{a_i} \coloneqq a_{i,0} + a_{i,1} + \cdots + a_{i,m_i} : 1 \leq i \leq c\},
\]
the maximal coefficient sum of the profiles \(a_i \in \mathbf{a}\). The
identity \(\binom{k+d+1}{k} = \frac{k + d + 1}{d+1}\binom{k+d}{k}\) easily
implies that the all derivatives of
\(\sum\nolimits_{i = 1}^c \prod\nolimits_{j = 0}^{m_i} \binom{k + a_{i,j}}{k}\)
with respect to \(k\) are increasing in the parameters \(a_{i,j} \geq 0\).
Explicitly computing second derivatives for the multi-profiles \((3)\), \((1 +
t)\), and \((2,2)\) implies that whenever \(\norm{\mathbf{a}}_\infty \geq 3\)
or \(\norm{\mathbf{a}}_\infty = 2\) and \(\mathbf{a} \neq (2)\), the function
in the minimum is concave for \(k \geq 0\). When \(\mathbf{a} = (2)\),
the function is an increasing. In all cases, this means that the
minimum is achieved at the endpoints, either when \(k = 0\) or \(k = r\), so
\[
\codim(Z_{r,[U]} \subset \mathbf{Inc}_{r,[U]})
\geq \min\{n - 2r - c, \delta(n,\mathbf{a},r)\} + 1
= \delta_-(n,\mathbf{a},r) + 1.
\qedhere
\]
\end{proof}

Non-smooth points of \(\mathbf{F}_{r-1}(X)\) often contribute to non-smooth
points of \(\mathbf{F}_r(X)\), producing components of \(Z_r\) that are too
large. Writing \(\Delta_r\) for the image of \(Z_r\) under
\(\pr_2 \colon \mathbf{Inc}_r \to \qaticci_{\PP V}\), the following
statement says that members of
\(Z_r^\circ \coloneqq Z_r \setminus \pr_2^{-1}(\Delta_{r-1})\) are
parameterized by the piece of highest codimension from \parref{fano-H}:

\begin{Lemma}\label{fano-general-varphi}
Let \(([U],[\boldsymbol{\alpha}]) \in Z_r^\circ\).
If \([\varphi] \in \mathbf{H}\) is such that
\(\varphi \circ \rho_{U,\boldsymbol{\alpha}} = 0\), then
\([\varphi] \in \mathbf{H}_r\).
\end{Lemma}

\begin{proof}
Let \(([U], [\boldsymbol{\alpha}]) \in Z_r\) and choose
\([\varphi] \in \mathbf{H}_k\) such that
\(\varphi \circ \rho_{U,\boldsymbol{\alpha}} = 0\). If \(k < r\), then
by its definition from \parref{fano-H}, this means that
\(\varphi_* \circ \mu \colon \mathrm{S}^{\mathbf{a}-1}(U^\vee) \to \Hom(U^\vee, \kk)\)
is not surjective; choose a hyperplane \(\Hom(U_0^\vee, \kk)\) containing the
image. As in \parref{fano-strata-codimension}, this means that
\(\varphi\) vanishes on the image of the multiplication map
\[
(U/U_0)^\vee \otimes \mathrm{S}^{\mathbf{a}-1}(U^\vee) \to \mathrm{S}^{\mathbf{a}}(U^\vee)
\]
and that \(\varphi\) descends to a nonzero linear functional
\(\varphi_0 \colon \mathrm{S}^{\mathbf{a}}(U_0^\vee) \to \kk\) on the cokernel.
Consider now the point \(([U_0],[\boldsymbol{\alpha}]) \in \mathbf{Inc}_{r-1}\)
and the corresponding tangent map
\(
\rho_{U_0, \boldsymbol{\alpha}} \colon U_0^\vee \otimes (V/U_0) \to \mathrm{S}^{\boldsymbol{a}}(U_0^\vee)
\). Since \(X_{\boldsymbol{\alpha}}\) contains \(\PP U\), the tensor
\(\partial_u \boldsymbol{\alpha}\) for \(u \in U\) lifting a basis of \(U/U_0\)
vanishes on \(\PP U \supset \PP U_0\), and so \(\rho_{U_0, \boldsymbol{\alpha}}\)
factors through the map
\[
U_0^\vee \otimes (V/U) \to \mathrm{S}^{\boldsymbol{a}}(U_0^\vee).
\]
But this map is but a restriction of \(\rho_{U,\boldsymbol{\alpha}}\), and so
its image is contained in the kernel of \(\varphi_0\). Thus
\(\rho_{U_0,\boldsymbol{\alpha}}\) is not surjective, so 
\(([U_0], [\boldsymbol{\alpha}]) \in Z_{r-1}\) by \parref{fano-nonsmooth-locus},
meaning that
\(([U], [\boldsymbol{\alpha}]) \in \pr_2^{-1}(\Delta_{r-1})\).
\end{proof}

\subsectiondash{}\label{fano-theorem-proof}
It remains to put everything together to prove \parref{fano-theorem}:

For \ref{fano-theorem.empty}, if \(\delta_-(n,\mathbf{a},r) < 0\) and
\(\mathbf{a} \neq (2)\), then
\(\pr_2 \colon \mathbf{Inc}_r \to \qaticci_{\PP V}\) cannot be dominant by the
dimension computation in \parref{fano-incidence-correspondence}, and so
\(\mathbf{F}_r(X) = \varnothing\) for general \(X\). The case \(\mathbf{a} = (2)\)
is well-known.

For \ref{fano-theorem.smooth}, if \(\delta_-(n,\mathbf{a},r) \geq 0\), then it
follows from \parref{fano-general-varphi} that the intersection of
\(Z_r \setminus \pr_2^{-1}(\Delta_{r-1})\) with \(\mathbf{Inc}_{r,[U]}\) is
contained in the image of \(Z_{r,[U]}' \times_{\mathbf{H}} \mathbf{H}_r\)
for each \([U] \in \mathbf{G}(r+1,V)\). The argument of
\parref{fano-main-codimension-estimate} then implies that the codimension of
its closure in \(\mathbf{Inc}_r\) is at least \(\delta_-(n,\mathbf{a},r) + 1\),
and so
\[
\dim\big(\overline{Z_r \setminus \pr_2^{-1}(\Delta_{r-1})}\big)
\leq \dim\mathbf{Inc}_r - \delta_-(n,\mathbf{a},r) - 1
\leq \dim\,\qaticci_{\PP V} - 1.
\]
Therefore \(\Delta_r \setminus \Delta_{r-1}\) is not all
of \(\qaticci_{\PP V}\). Induction on \(r\)---the base case with \(r = 0\)
is the statement \parref{qatic-general-is-smooth} that the general
\((q;\mathbf{a})\)-tic scheme is smooth---then shows that \(\Delta_r\) is a
proper closed subset of \(\qaticci_{\PP V}\), meaning that \(\mathbf{F}_r(X)\)
is smooth of the expected dimension for general \(X\) by
\parref{fano-nonsmooth-locus}.

For \ref{fano-theorem.connected}, when \(\delta_-(n, \mathbf{a}, r) > 0\),
consider the Stein factorization
\[
\pr_2 \colon \mathbf{Inc}_r \to \mathbf{Inc}_r' \to \qaticci_{\PP V}
\]
of the second projection. Then \(Z_r\) contains the preimage of the branch
locus of \(\mathbf{Inc}_r' \to \qaticci_{\PP V}\). Since \(Z_r\) has
codimension at least \(2\) in \(\mathbf{Inc}_r\) by
\parref{fano-main-codimension-estimate}, purity of the branch locus, as in
\citeSP{0BMB}, implies that \(\mathbf{Inc}_r' \to \qaticci_{\PP V}\)
is finite \'etale; the target is a multi-projective space and so it is simply
connected, thus this is an isomorphism, and properties of the Stein
factorization imply that \(\pr_2\) has connected fibres. In other words,
\(\mathbf{F}_r(X)\) is connected for every \(X\).
\qed

\medskip
When \(\mathbf{F}_r(X)\) is of its expected dimension
\(\delta(n,\mathbf{a},r)\), \parref{fano-equations} shows that it is cut out in
\(\mathbf{G}\) by a regular section of
\(\mathrm{S}^{\mathbf{a}}(\mathcal{S}^\vee)\). Various simple numerical invariants
of the Fano scheme may then be determined via Schubert calculus, as is done
in \cite[\S\S3--4]{DM} for classical complete intersections, and
\cite[1.13--1.15]{fano-schemes} and 
\cite[1.11--1.13]{qbic-threefolds} in the \(q\)-bic case. For now, record
the fact that the dualizing sheaf \(\omega_{\mathbf{F}_r(X)}\) is a power
of the Pl\"ucker line bundle \(\sO_{\mathbf{F}_r(X)}(1)\):

\begin{Proposition}\label{fano-canonical}
If \(X \subseteq \PP^n\) is a \((q;\mathbf{a})\)-tic scheme such that
\(\dim\mathbf{F}_r(X) = \delta(n,\mathbf{a},r)\), then
\[
\omega_{\mathbf{F}_r(X)} \cong
\sO_{\mathbf{F}_r(X)}(\gamma(\mathbf{a},r,q) - n - 1)
\;\;\text{where}\;\;
\gamma(\mathbf{a},r,q) \coloneqq
\frac{1}{r+1}\sum\nolimits_{a \in \mathbf{a}} a(q) \cdot \prod\nolimits_{j \geq 0} \binom{a_j + r}{r}.
\]
\end{Proposition}

\begin{proof}
Duality theory, as in \citeSP{0AU3}, shows that the dualizing sheaf is given in
this case by
\(
\omega_{\mathbf{F}_r(X)} \cong
\omega_{\mathbf{G}}\rvert_{\mathbf{F}_r(X)} \otimes
\det\mathrm{S}^{\mathbf{a}}(\mathcal{S}^\vee)
\).
Tensor product formulae show that, for a profile
\(a = \sum\nolimits_{j \geq 0} a_j t^j\), 
\[
\det\mathrm{S}^a(\mathcal{S}^\vee)
= \det\big(\bigotimes\nolimits_{j \geq 0}
  \Sym^{a_j}(\mathcal{S}^\vee)^{[j]}\big)
\cong \bigotimes\nolimits_{j \geq 0}
  \det\big(
    \Sym^{a_j}(\mathcal{S}^\vee)\big)^{\otimes q^j \prod\nolimits_{k \geq 0} \binom{a_k + r}{r}/\binom{a_j + r}{r}}.
\]
Combined with fact that \(\det\Sym^{a_j}(\mathcal{S}^\vee)\) is the
\(\binom{a_j + r}{r + 1}\)-th power of \(\sO_{\mathbf{F}_r(X)}(1)\) gives
the result.
\end{proof}

\subsectiondash{\(r\)-planes through a point}\label{fano-planes-through-point}
Let \(X \subseteq \PP V\) be a \((q;\mathbf{a})\)-tic scheme and consider the
scheme
\[
\mathbf{F}_r(X,x)
\coloneqq \{[\PP U] \in \mathbf{F}_r(X) : x \in \PP U \subseteq X\}
\]
parameterizing \(r\)-planes in \(X\) through a given closed point \(x\). As
usual, this may be canonically identified as the Fano scheme of
\((r-1)\)-planes of a scheme \(X_{1,x} \subseteq \PP(V/L)\), where \(L\) is the
\(1\)-dimensional space underlying \(x\) and
\(X_{1,x} \coloneqq \mathbf{F}_1(X,x)\) is the scheme of lines in \(X\) through
\(x\). A classical fact, see \cite[\S2]{HRS} for example, is that if \(X
\subseteq \PP V\) is a scheme defined by equations of multi-degree \(\mathbf{d}
= (d_1,d_2,\ldots,d_c)\), then \(X_{1,x} \subseteq \PP(V/L)\) is defined by
equations of multi-degree
\[
\mathbf{d}_1 \coloneqq
(
  d_1, d_1 - 1, \ldots, 2, 1;\,
  d_2, d_2 - 1, \ldots, 2, 1;
  \ldots;\,
  d_c, d_c - 1, \ldots, 2, 1).
\]
A generalization of this to \((q;\mathbf{a})\)-tic schemes is as follows:

\begin{Proposition}\label{fano-pointed-line-equations}
Let \(X \subseteq \PP V \cong \PP^n\) be a \((q;\mathbf{a})\)-tic scheme.
For every closed point \(x \in X\), the scheme \(X_{1,x} = \mathbf{F}_1(X,x)\)
of lines in \(X\) through \(x\) is a \((q;\mathbf{a}_1)\)-tic scheme in
\(\PP^{n-1}\), where
\[
\mathbf{a}_1 \coloneqq
(b \in \Types : 0 \prec b \preceq a \;\text{with}\; a \in \mathbf{a}).
\]
\end{Proposition}

\begin{proof}
It is illustrative to work slightly more globally and to describe the scheme
\[
X_1 \coloneqq
\{(x,[\ell]) \in X \times \mathbf{F}_1(X) : x \in \ell \subseteq X\}
\]
of pointed lines locally relative to \(X\). Choose projective coordinates
\(\mathbf{x} \coloneqq (x_0:\cdots:x_n)\) and defining \((q;a)\)-tic equations
\(X = \mathrm{V}(f_a : a \in \mathbf{a})\). Over the standard affine open
subscheme \(\mathrm{D}(x_n)\), the space of pointed lines in \(\PP^n\)
restricts to a trivial \(\PP^{n-1}\)-bundle, and fibre coordinates
\(\mathbf{y} \coloneqq (y_0:\cdots:y_{n-1})\) may be chosen so that the point
\((\mathbf{x}, \mathbf{y}) \in \mathrm{D}(x_n) \times \PP^{n-1}\) represents
the line \(\ell\) parameterized by
\[
\varphi \colon
\PP^1 \to \PP^n \colon
(\xi:\eta) \mapsto
\xi\mathbf{x} + \eta\mathbf{y} \coloneqq
(\xi x_0 + \eta y_0 : \cdots : \xi x_{n-1} + \eta y_{n-1} : \xi x_n).
\]
Then \(\ell \subseteq X\) if and only if
\(\varphi^*(f_a) = f_a(\xi\mathbf{x} + \eta\mathbf{y}) = 0\) for each
\(a \in \mathbf{a}\). View \(f_a(\xi\mathbf{x} + \eta\mathbf{y})\) as a
polynomial in the auxiliary variables \((\xi:\eta)\); its coefficients are
polynomials \(f_{a,b}(\mathbf{x};\mathbf{y})\) which provide the equations for
\(X_1\) in \(\mathrm{D}(x_n) \times \PP^{n-1}\), and have the following
form:

\begin{Lemma}\label{tlines-explicit-expansion}
Let \(f_a(x_0,\ldots,x_n)\) be a \((q;a)\)-tic polynomial. Then there is a
unique expansion
\[
f_a(\xi x_0 + \eta y_0, \ldots, \xi x_{n-1} + \eta y_{n-1}, \xi x_n)
= \sum\nolimits_{0 \preceq b \preceq a}
f_{a,b}(x_0,\ldots,x_n; y_0,\ldots,y_{n-1}) \cdot \xi^{a(q) - b(q)}\eta^{b(q)}
\]
where the polynomials \(f_{a,b}(x_0,\ldots,x_n;y_0,\ldots,y_{n-1})\) are
homogeneous of bi-profile \((a-b,b)\).
\end{Lemma}

\begin{proof}
When the profile \(a = a_0\) is a constant, this is classical: simply group
terms with respect to the monomials in \(\xi\) and \(\eta\). For a general
profile \(a = a_0 + a_1t + \cdots + a_mt^m\), note that
\(f_a(\xi\mathbf{x} + \eta\mathbf{y})\) is a sum of polynomials of the form
\[
\prod\nolimits_{j = 0}^m f_{a_j}(\xi\mathbf{x} + \eta\mathbf{y})^{q^j} =
\prod\nolimits_{j = 0}^m\Big(
  \sum\nolimits_{b_j = 0}^{a_j} f_{a_j,b_j}(\mathbf{x};\mathbf{y}) \cdot \xi^{a_j-b_j}\eta^{b_j}
  \Big)^{q^j}
\]
where the \(f_{a_j}\) are homogeneous of degree \(a_j\), and \(f_{a_j,b_j}\)
is bihomogeneous of bidegree \((a_j-b_j,b_j)\). The product expands to
a sum of terms of the form
\[
\Big(\prod\nolimits_{j = 0}^m f_{a_j,b_j}(\mathbf{x};\mathbf{y})^{q^j}\Big) \cdot
\xi^{a(q) - b(q)} \eta^{b(q)}
\]
which is a polynomial with profile \(b \coloneqq b_0 + b_1t + \cdots + b_mt^m\)
in \(\mathbf{y}\), and profile \(a - b\) in \(\mathbf{x}\). Injectivity
of the multiplication maps implies via \parref{qatic-injective-mult} that any
coefficient of \(\xi^{a(q) - b(q)} \eta^{b(q)}\) is of bi-profile
\((a - b, b)\), from which the result follows.
\end{proof}

To complete the proof of \parref{fano-pointed-line-equations}, observe that
the bi-profile \((a,0)\) terms in \parref{tlines-explicit-expansion} are
\(f_{a,0}(\mathbf{x},\mathbf{y}) = f_a(\mathbf{x})\) simply the
equations of \(X\). Therefore, over \(X \cap \mathrm{D}(x_n)\), the polynomials
\(f_{a,b}(\mathbf{x};\mathbf{y})\) with \(0 \prec b \preceq a\) and \(a \in \mathbf{a}\)
present \(X_1\) as a \((q;\mathbf{a}_1)\)-tic scheme in a projective \((n-1)\)-space.
\end{proof}

Writing \(\mathbf{F}_r(X,x) \cong \mathbf{F}_{r-1}(X_{1,x})\) and combining
\parref{fano-pointed-line-equations} with \parref{fano-equations} and
\parref{fano-theorem} shows that a \((q;\mathbf{a})\)-tic scheme is covered by
\(r\)-planes as soon as \(\delta_-(n-1,\mathbf{a}_1,r-1) \geq 0\). A simple
criterion for \(\delta(n-1,\mathbf{a}_1,r-1) \geq 0\) may be obtained by
using the identity
\[
1 + \sum\nolimits_{0 \prec b \preceq a} \prod\nolimits_{j \geq 0} \binom{b_j + r - 1}{r - 1}
=
\prod\nolimits_{j \geq 0} \Big(\sum\nolimits_{b_j = 0}^{a_j} \binom{b_j + r - 1}{r - 1}\Big)
= \prod\nolimits_{j \geq 0} \binom{a_j + r}{r}.
\]

\begin{Corollary}\label{fano-covered-in-planes}
Let \(X \subseteq \PP V\) be a \((q;\mathbf{a})\)-tic scheme. If
\[
n \geq \max\big\{2r-1 + \#\mathbf{a}_1, r +
\frac{1}{r}
\sum\nolimits_{a \in \mathbf{a}}
\prod\nolimits_{j \geq 0} \binom{a_j + r}{r} -
\frac{1}{r}\#\mathbf{a}
\big\},
\]
then \(X\) is covered by \(r\)-planes.
\qed
\end{Corollary}

A more global version of \parref{fano-pointed-line-equations} will be given
in \S\parref{section-tlines} and will feature in the unirationality
construction in Theorem \parref{intro-unirationality}. The next section
prepares for this by clarifying what a family of \((q;\mathbf{a})\)-tic schemes
ought to be and by developing some tools for manipulating such families.

\section{Families}\label{section-families}
Consider a family \(\mathcal{X} \to S\) of complete intersections in a
projective bundle \(\pi \colon \PP\mathcal{V} \to S\) over the field \(\kk\)
which is cut out by a regular section
\(\sigma \colon \sO_{\PP\mathcal{V}} \to \mathcal{E}\) of a finite locally free
\(\sO_{\PP\mathcal{V}}\)-module. What additional structure
should be required to elevate
\(\mathcal{X}\) into a family of \((q;\mathbf{a})\)-tic complete intersections?
As a minimum, each fibre
\(\mathcal{X}_s\) ought to be a \((q;\mathbf{a})\)-tic complete intersection in
\(\PP\mathcal{V}_s\), meaning as in \parref{qatics-definition} that the
equations \(\sigma_s\) are induced by a \((q;\mathbf{a})\)-tic tensor
\(\boldsymbol{\alpha}_s\). This tensor is a crucial part of the structure that
defines a \((q;\mathbf{a})\)-tic complete intersection, so one ought to ask
that the \(\boldsymbol{\alpha}_s\) vary continuously across the family. The
following examples illustrate some of the subtleties involved:

\begin{Example}\label{families-varying-tensor-structure}
Fix a nonconstant profile
\(a \coloneqq a_0 + a_1t + \cdots + a_mt^m \in \Types\), let
\(S \coloneqq \mathbf{A}^2\) be the affine plane with coordinates
\((s_1,s_2)\), and choose homogeneous polynomials
\(f_1,f_2,g \in \sO_S[x_0,\ldots,x_n]\), where \(f_1\) and
\(f_2\) are general coprime \((q;a)\)-tic polynomials and where \(g\) is simply
general of degree \(d \coloneqq a(q)\). Consider the closed subscheme of
\(\PP^n_S\) defined by
the section
\[
\sigma = (f_1 + s_1 g, f_2 + s_2 g)^\vee \colon
\sO_{\PP^n_S} \to
\sO_{\PP^n_S}(d)^{\oplus 2}.
\]
Setting \(\mathbf{a} \coloneqq (d, a_0 + a_1t + \cdots + a_mt^m)\), then
\(\mathcal{X}\) is a family of codimension \(2\)
complete intersections in \(\PP^n_S\) with the property that each fibre is
a \((q;\mathbf{a})\)-tic complete intersection, but there is no continuously
varying family of \((q;\mathbf{a})\)-tic tensor defining \(\mathcal{X}\) in
a neighbourhood of the origin of \(S\).
\end{Example}

\begin{proof}
That each fibre is a \((q;\mathbf{a})\)-tic complete intersection is
straightforward: At a point \(s\) where \(s_i \neq 0\),
\(\mathcal{X}_s\) is cut out by the equations 
\(f_i + s_ig = s_1f_2 - s_2f_1 = 0\). Over the origin, \(\mathcal{X}_0\) is cut
out by \(f_1 = f_2 = 0\), and either choice of \(f_i\) being considered as a
\((q;a)\)-tic equation suffices. Suppose now that \(U \subseteq S\) is a
neighbourhood of \(0\) over which \(\mathcal{X}\rvert_U\) is defined in
\(\PP^n_U\) by a family of \((q;\mathbf{a})\)-tic tensors
\(\boldsymbol{\alpha}(s)\). The \((q;a)\)-tic component of the tensor takes the
form
\[
\boldsymbol{\alpha}(s)_a = xf_1 + yf_2
\;\text{where}\; x, y \in \Gamma(U, \sO_S)
\;\text{satisfies}\;
xs_1 + ys_2 = 0.
\]
Thus \(y\) is divisible by \(s_1\) and \(x\) is divisible by \(s_2\), and so
\(\boldsymbol{\alpha}(s)_a\) vanishes at the origin, contradicting the
assumption that it provides the \((q;a)\)-tic equation of \(\mathcal{X}\)
over all of \(U\).
\end{proof}

Excising the origin from the base ensures that a tensor does glue together to
become a section of a non-split vector bundle:

\begin{Example}\label{families-varying-tensor-structure-section}
Take the base \(S \coloneqq \mathbf{A}^2 \setminus \{0\}\) and continue with
the example \(\mathcal{X} \subset \PP^n_S\) of
\parref{families-varying-tensor-structure}. Consider the standard affine
open cover given by the complement \(U_i \coloneqq \mathrm{D}(s_i)\) of the
\(s_i\)-axis for \(i = 1,2\). Then
\[
\mathcal{X}\rvert_{U_i} = \mathrm{V}(f_i + s_ig, s_1f_2 - s_2f_1) \subset \PP^n_{U_i}
\]
is a presentation of \(\mathcal{X}\) as a \((q;\mathbf{a})\)-tic complete
intersection over \(U_i\). This presentation globalizes over \(S\) in the
following sense: Write \(\PP^n_S = \PP(\sO_S \otimes V)\) for a vector space
\(V\), and consider for \(i = 1, 2\) the map of locally free
\(\sO_{U_i}\)-modules
\[
\boldsymbol{\alpha}_i = (f_i + s_ig, s_1f_2 - s_2f_1)^\vee \colon
\sO_{U_i} \to
\big(\sO_{U_i} \otimes \Sym^d(V^\vee)\big) \oplus
\big(\sO_{U_i} \otimes \mathrm{S}^a(V^\vee)\big).
\]
These glue via the automorphism
on the intersection \(U_{1,2} \coloneqq U_1 \cap U_2\) given by
\[
\varphi_{1,2} \coloneqq
\begin{pmatrix}
s_2/s_1 & 1/s_1 \\
0 &  1
\end{pmatrix}
\in \Aut\Big(
\big(\sO_{U_{1,2}} \otimes \Sym^d(V^\vee)\big) \oplus
\big(\sO_{U_{1,2}} \otimes \mathrm{S}^a(V^\vee)\big)\Big).
\]
Thus there is a global section
\(\boldsymbol{\alpha} \colon \sO_S \to \mathcal{A}\)
such that \(\boldsymbol{\alpha}\rvert_{U_i} = \boldsymbol{\alpha}_i\), where
\(\mathcal{A}\) fits in an extension
\[
0 \to
\sO_S \otimes \Sym^d(V^\vee) \to
\mathcal{A} \to
\sO_S \otimes \mathrm{S}^a(V^\vee) \to
0.
\]
This tensor underlies \(\mathcal{X} \subset \PP^n_S\) in the sense that the
section \(\sigma\) defining \(\mathcal{X}\) from
\parref{families-varying-tensor-structure} factors as
\[
\sigma = \varepsilon \circ \pi^*\boldsymbol{\alpha} \colon
\sO_{\PP^n_S} \to \pi^*\mathcal{A} \to \sO_{\PP^n_S}(d)^{\oplus 2}
\]
where \(\varepsilon\) is the surjective morphism of \(\sO_{\PP^n_S}\)-modules locally
induced by the maps
\[
\pi_*\varepsilon\rvert_{U_1} =
\begin{pmatrix}
1 & 0 \\ s_2/s_1 & 1/s_1
\end{pmatrix}
\;\;\text{and}\;\;
\pi_*\varepsilon\rvert_{U_2} =
\begin{pmatrix}
s_1/s_2 & -1/s_2 \\
1 & 0
\end{pmatrix}
\]
between 
\(
(\sO_{U_i} \otimes_\kk \Sym^d(V^\vee)) \oplus
(\sO_{U_i} \otimes_\kk \mathrm{S}^a(V^\vee))
\to \sO_{U_i} \otimes_{\kk} \Sym^d(V^\vee)^{\oplus 2}\). Moreover, there is a
change of fibre coordinates of \(\PP^n_S\) in which \(\varepsilon\) is in
fact locally induced by the evaluation map of \(\pi^*\pi_*\sO_{\PP^n_S}(d)\).
Despite this, the extension in which \(\mathcal{A}\) fits is nontrivial:
Otherwise, projecting
\(\boldsymbol{\alpha}\) to
the \(\Sym^d\)-component provides a polynomial \(h\) cutting out \(\mathcal{X}\)
along with the equation \(s_1f_2 - s_2f_1\). Comparing generators of the ideal of
\(\mathcal{X}\) shows that
\(h = x(f_1 + s_1g) + y(f_2 + s_2g)\) for some
\[
x, y \in \Gamma(S,\sO_S) = \kk[s_1,s_2]
\;\text{satisfying}\; xs_1 + ys_2 \in \Gamma(S,\sO_S)^\times = \kk^\times.
\]
This is impossible, and so the extension is non-split. \qed
\end{Example}

Toward a definition of a family of \((q;\mathbf{a})\)-tic schemes, suppose one
is given a section of the form
\[
\boldsymbol{\alpha} \colon
\sO_S \to
\bigoplus\nolimits_{a \in \mathbf{a}} \mathrm{S}^a(\mathcal{V}^\vee).
\]
Adjunction along \(\pi \colon \PP\mathcal{V} \to S\) together with the relative
evaluation maps induce a canonical map
\[
\sigma \coloneqq \varepsilon \circ \pi^*\boldsymbol{\alpha} \colon
\sO_{\PP\mathcal{V}} \to
\bigoplus\nolimits_{a \in \mathbf{a}} \pi^*\mathrm{S}^a(\mathcal{V}^\vee) \to
\bigoplus\nolimits_{a \in \mathbf{a}} \sO_\pi(a(q))
\]
whose zero locus \(\mathcal{X}\) ought to be called a family of
\((q;\mathbf{a})\)-tic schemes over \(S\). However, there are families like
those in \parref{families-varying-tensor-structure-section} whose defining
section takes values in a vector bundle which is only locally of this form.
Additional data is therefore necessary to globalize this construction. The
solution taken here is to require a map \(\varepsilon\) globalizing the
evaluation map.

\subsectiondash{Definitions}\label{families-complete-intersections}
A \emph{family of \((q;\mathbf{a})\)-tic tensors} valued in a finite locally
free \(\sO_S\)-module \(\mathcal{V}\) is a section
\(\boldsymbol{\alpha} \colon \sO_S \to \mathcal{A}\), where \(\mathcal{A}\)
is locally on \(S\) isomorphic to
\[
\bigoplus\nolimits_{a \in \mathbf{a}} \mathrm{S}^a(\mathcal{V}^\vee).
\]
A \emph{family of \((q;\mathbf{a})\)-tic schemes} in a projective bundle
\(\pi \colon \PP\mathcal{V} \to S\) consists of the data of
\begin{itemize}
\item
a closed subscheme \(\mathcal{X} \subseteq \PP\mathcal{V}\) cut out by a
section \(\sigma \colon \sO_{\PP\mathcal{V}} \to \mathcal{E}\) of a vector bundle;
\item
a family of \((q;\mathbf{a})\)-tic tensors
\(\boldsymbol{\alpha} \colon \sO_S \to \mathcal{A}\); and
\item
a surjective morphism \(\varepsilon \colon \pi^*\mathcal{A} \to \mathcal{E}\) of
\(\sO_{\PP\mathcal{V}}\)-modules.
\end{itemize}
These data are subject to the conditions that:
\begin{enumerate}
\item\label{families-complete-intersections.factorize}
there is a factorization
\(\sigma = \varepsilon \circ \pi^*\boldsymbol{\alpha} \colon
\sO_{\PP\mathcal{V}} \to \pi^*\mathcal{A} \to \mathcal{E}\); and
\item\label{families-complete-intersections.eta}
there exists an open cover \(S = \bigcup\nolimits_{i \in I} U_i\) on which
the morphism \(\varepsilon \colon \pi^*\mathcal{A} \to \mathcal{E}\) is
isomorphic to the canonical map induced by evaluation along \(\pi\):
\[
\ev_\pi \colon
\bigoplus\nolimits_{a \in \mathbf{a}} \pi^*\mathrm{S}^a(\mathcal{V}^\vee) \to
\bigoplus\nolimits_{a \in \mathbf{a}} \sO_\pi(a(q)).
\]
\end{enumerate}
When \(\sigma\) is a regular section making \(\mathcal{X} \to S\) fibrewise a
complete intersection, furthermore call the triple
\((\mathcal{X}, \boldsymbol{\alpha}, \varepsilon)\) a \emph{family of
\((q;\mathbf{a})\)-tic complete intersections}. The family
\(\mathcal{X} \to S\) is often referred to as the family of
\((q;\mathbf{a})\)-tic schemes or complete intersections, leaving the
\(\boldsymbol{\alpha}\) and \(\varepsilon\) implicit.

\subsectiondash{Example}\label{families-blowups}
Since \((q;\mathbf{a})\)-tic structures are inherited upon passing to linear
sections by \parref{qatic-linear-sections}, families of linear sections of
\((q;\mathbf{a})\)-tic schemes provide a source of non-trivial examples.
To give a simple illustration, consider linear projection of a projective space
\(\PP V\) centred along a subspace \(\PP U\). This provides a rational map to
\(\PP(V/U)\) which is resolved on the blowup
\(b \colon \widetilde{\PP} V \to \PP V\) centred along \(\PP U\); as usual,
\(b\) exhibits \(\widetilde{\PP} V\) as the projective
bundle over \(\PP(V/U)\) whose underlying vector bundle \(\mathcal{V}\)
canonically arises via the diagram of short exact sequences
\[
\begin{tikzcd}
0 \rar
& \sO_{\PP(V/U)} \otimes U \rar \dar[equal]
& \mathcal{V} \rar \dar
& \sO_{\PP(V/U)}(-1) \rar \dar
& 0 \\
0 \rar
& \sO_{\PP(V/U)} \otimes U \rar
& \sO_{\PP(V/U)} \otimes V \rar
& \sO_{\PP(V/U)} \otimes V/U \rar
& 0\punct{.}
\end{tikzcd}
\]
If \(X \subseteq \PP V\) is a \((q;\mathbf{a})\)-tic scheme not contained in
\(\PP U\), then its \emph{total transform} \(\mathcal{X} \coloneqq b^{-1}(X)\)
in \(\widetilde{\PP} V\) is a family of \((q;\mathbf{a})\)-tic schemes over
\(S \coloneqq \PP(V/U)\): Dualizing the inclusion \(\mathcal{V} \subseteq
\sO_{\PP(V/U)} \otimes V\) and applying the tensor functor
\(\mathrm{S}^{\mathbf{a}}\) provides a restriction map
\[
\sO_{\PP(V/U)} \otimes \mathrm{S}^{\mathbf{a}}(V^\vee) \to
\mathrm{S}^{\mathbf{a}}(\mathcal{V}^\vee).
\]
Mapping a \((q;\mathbf{a})\)-tic tensor defining \(X\) along this provides
a family of \((q;\mathbf{a})\)-tic tensors \(\boldsymbol{\alpha}\) defining
\(\mathcal{X}\).

\subsectiondash{}\label{families-not-strict-transform}
If, furthermore, \(\PP U \subseteq X\), then its \emph{strict transform}
\(\widetilde{X}\) in \(\widetilde{\PP} V\) generally does \emph{not} appear to
carry useful additional structure from the point of view of this article. For
instance, consider the smooth \(q\)-bic surface
\[
X \coloneqq
\set{
(x_0:x_1:x_2:x_3) \in \PP^3 :
x_0^q x_1 + x_0 x_1^q + x_2^q x_3 + x_2 x_3^q = 0
}.
\]
Projection from the line
\(\PP U = (0:x_1:0:x_3)\) exhibits the strict transform \(\widetilde{X}\) as a
family of degree \(q\) plane curves, the general fibre of which is isomorphic
to
\[
C \coloneqq \set{(y_0:y_1:y_2) \in \PP^2 : y_0^q + y_1y_2^{q-1} = 0},
\]
see \citeThesis{2.5.3} for instance. 
In general, this equation does not belong to any proper subspace of
\(\Sym^q(\kk^{\oplus 3})\) of the form \(\mathrm{S}^a(\kk^{\oplus 3})\) for
\(a \in \Types\) of numerical degree \(q\). This perhaps suggests
that the class of \((q;\mathbf{a})\)-tic schemes is not sufficiently flexible.
It would be useful to develop methods to handle a larger class of schemes that
includes irreducible components of \((q;\mathbf{a})\)-tic schemes.

\subsectiondash{Transition functions}\label{families-transition-functions}
Injectivity of multiplication maps associated with profiles from
\parref{qatics-profiles} together with condition
\parref{families-complete-intersections}\ref{families-complete-intersections.eta}
implies that the map
\(\pi_*\varepsilon \colon \mathcal{A} \to \pi_*\mathcal{E}\) is an injection
which locally on \(S\) is isomorphic to the inclusion of \((q;a)\)-tic
polynomials amongst all polynomials of degree \(a(q)\), with \(a\) ranging over
profiles in \(\mathbf{a}\). This endows \(\mathcal{A}\) with some
additional structure: Choose an open covering
\(S = \bigcup\nolimits_{i \in I} U_i\) and trivializations of \(\mathcal{A}\)
and \(\pi_*\mathcal{E}\) so that \(\pi_*\varepsilon\rvert_{U_i}\) is the
identified with the inclusion
\[
\bigoplus\nolimits_{a \in \mathbf{a}} \mathrm{S}^a(\mathcal{V}^\vee)\rvert_{U_i} \subseteq
\bigoplus\nolimits_{a \in \mathbf{a}} \Sym^{a(q)}(\mathcal{V}^\vee)\rvert_{U_i}.
\]
For each pair of indices \(i,j \in I\), let \(\psi_{i,j}\) and \(\varphi_{i,j}\)
be the transition functions of \(\mathcal{A}\) and \(\pi_*\mathcal{E}\) with
these trivializations over \(U_{i,j} \coloneqq U_i \cap U_j\). Since
\(\mathcal{E}\) is locally a sum of the \(\sO_\pi(a(q))\), the
\(\varphi_{i,j}\) are induced by multiplication of polynomials; in particular,
for \(a, b \in \mathbf{a}\), the \((a,b)\)-component
\[
(\varphi_{i,j})_{a,b} \colon
\Sym^{a(q)}(\mathcal{V}^\vee)\rvert_{U_{i,j}} \to
\Sym^{b(q)}(\mathcal{V}^\vee)\rvert_{U_{i,j}}
\]
is induced by multiplication by a polynomial \(f_{i,j,a,b}\)
of degree \(b(q) - a(q)\). Since \(\mathcal{A}\) is identified as a subbundle
of \(\pi_*\mathcal{E}\), this means that the map
\[
(\psi_{i,j})_{a,b} \colon
\mathrm{S}^a(\mathcal{V}^\vee)\rvert_{U_{i,j}} \to
\mathrm{S}^b(\mathcal{V}^\vee)\rvert_{U_{i,j}}
\]
is also induced by multiplication by the same polynomial \(f_{i,j,a,b}\).
Keeping track of profiles shows that whether or not \((\psi_{i,j})_{a,b}\) must
vanish is related to the composite partial ordering \(\rightsquigarrow\) from
\parref{qatics-ordering}:

\begin{Lemma}\label{families-vanishing-transition-functions}
If  \((\psi_{i,j})_{a,b} \neq 0\), then \(a \rightsquigarrow b\).
\qed
\end{Lemma}

As a simple though useful consequence, this provides a \((q;\mathbf{a})\)-tic
tensor \(\boldsymbol{\alpha} \colon \sO_S \to \mathcal{A}\) with a canonical
decomposition into types. Namely, partition the collection \(\mathbf{a}\) of
profiles into
\[
\mathbf{a}_{\mathrm{lin}} \coloneqq \{a \in \mathbf{a} : a(t) = 1\},
\;\;
\mathbf{a}_{\mathrm{pow}} \coloneqq \{a \in \mathbf{a} : a(0) = 0\},\;\;
\mathbf{a}_{\mathrm{nlr}} \coloneqq \{a \in \mathbf{a} : a(t) \neq 1 \;\text{and}\; a(0) = 0\}
\]
those that are linear, those that are nonreduced, and those that are non-linear
and reduced. The linear and nonreduced components of \(\mathcal{A}\) give two
canonical quotients:

\begin{Lemma}\label{families-quotients}
Let \(\mathcal{X}\) be a family of \((q;\mathbf{a})\)-tic schemes in a
projective bundle \(\pi \colon \PP\mathcal{V} \to S\). There is a canonical
short exact sequence of locally free \(\sO_S\)-modules
\[
0 \to
\mathcal{A}_{\mathrm{nlr}} \to
\mathcal{A} \to
\mathcal{A}_{\mathrm{lin}} \oplus
\mathcal{A}_{\mathrm{pow}} \to
0
\]
where, for each \(\mathrm{type} \in \{\mathrm{nlr}, \mathrm{lin}, \mathrm{pow}\}\),
\(\mathcal{A}_{\mathrm{type}}\) is locally on \(S\) isomorphic to
\(\bigoplus_{a \in \mathbf{a}_{\mathrm{type}}} \mathrm{S}^a(\mathcal{V}^\vee)\).
An analogous and compatible exact sequence exists for \(\mathcal{E}\).
\end{Lemma}

\begin{proof}
From \parref{families-vanishing-transition-functions} together with
\parref{qatics-ordering-properties}\ref{qatics-ordering-properties.degrees}, it
follows that \((\psi_{i,j})_{a,1} = 0\) for each
\(a \in \mathbf{a} \setminus \mathbf{a}_{\mathrm{lin}}\), meaning that the
local summands of the form \(\mathcal{V}^\vee\) fit together as a quotient
bundle \(\mathcal{A} \to \mathcal{A}_{\mathrm{lin}}\). Similarly, if
\(a \in \mathbf{a} \setminus \mathbf{a}_{\mathrm{pow}}\) and
\(b \in \mathbf{a}_{\mathrm{pow}}\), then \((\psi_{i,j})_{a,b} = 0\) since
profiles preceding a nonreduced profile must also be nonreduced as observed in
\parref{qatics-ordering-properties}\ref{qatics-ordering-properties.reduced}.
Hence the local summands indexed by \(\mathbf{a}_{\mathrm{pow}}\) also fit
together to form a quotient \(\mathcal{A} \to \mathcal{A}_{\mathrm{pow}}\).
The latter argument also implies that all possible transition functions between
\(\mathcal{A}_{\mathrm{lin}}\) and \(\mathcal{A}_{\mathrm{pow}}\) must vanish,
from which the remaining conclusions follow.
\end{proof}

Dually, \parref{families-vanishing-transition-functions} implies that
equations whose profile is maximal for \(\rightsquigarrow\) in \(\mathbf{a}\)
fit together into a subbundle of \(\mathcal{A}\). One such class of maximal
profiles are those \(a \in \mathbf{a}\) with maximal coefficient sum \(a(1)\):
see
\parref{qatics-ordering-properties}\ref{qatics-ordering-properties.coefficient-sums}.
This is most useful when restricted to \(\mathcal{A}_{\mathrm{nlr}} \neq 0\):

\begin{Lemma}\label{families-maximal-subbundle}
Let \(\mathcal{X}\) be a family of \((q;\mathbf{a})\)-tic schemes in a projective
bundle \(\pi \colon \PP\mathcal{V} \to S\). If \(\mathbf{a}_{\mathrm{nlr}} \neq \varnothing\),
then there exists a nonzero subbundle of the form
\[
\mathrm{S}^a(\mathcal{V}^\vee) \otimes \mathcal{M} \subseteq
\mathcal{A}_{\mathrm{nlr}} \subseteq
\mathcal{A}
\]
for some locally free \(\sO_S\)-module \(\mathcal{M}\) and some
\(a \in \mathbf{a}_{\mathrm{nlr}}\) with maximal coefficient sum \(a(1)\).
\qed
\end{Lemma}

The canonical \emph{type decomposition} from \parref{families-quotients}
of a \((q;\mathbf{a})\)-tic tensor allows one to sometimes perform certain
simplifications to the equations defining the family \(\mathcal{X}\). For
instance---a construction which is of course much more generally applicable to
any family of projective schemes---the linear equations may be used to cut out
a projective subbundle containing \(\mathcal{X}\), thereby reducing the number
of equations to keep track of. Precisely:

\begin{Lemma}\label{families-remove-linear-equations}
Let \(\mathcal{X}\) be a family of \((q;\mathbf{a})\)-tic schemes
in a projective bundle \(\pi \colon \PP\mathcal{V} \to S\).
There exists a canonical subbundle \(\PP\mathcal{V}' \subseteq \PP\mathcal{V}\)
containing \(\mathcal{X}\) in which it is a family of
\((q;\mathbf{a} \setminus \mathbf{a}_{\mathrm{lin}})\)-tic schemes.
\end{Lemma}

\begin{proof}
Let \(\lambda \colon \mathcal{E} \to \mathcal{E}_{\mathrm{lin}}\) be the
quotient from \parref{families-quotients} in which the linear equations of
\(\mathcal{X} \subseteq \PP\mathcal{V}\) take values in, so that
\(\mathcal{E}_{\mathrm{lin}} \cong \sO_\pi(1) \otimes \pi^*\mathcal{M}\)
for some locally free \(\sO_S\)-module \(\mathcal{M}\) of rank
\(\#\mathbf{a}_{\mathrm{lin}}\). The subbundle
\(\PP\mathcal{V}' \subseteq \PP\mathcal{V}\) defined by
\(\sigma \circ \lambda\) then contains \(\mathcal{X}\), in which it is
a complete intersection cut out by the induced section
\(\sigma' \colon \sO_{\PP\mathcal{V}'} \to \mathcal{E}'\) valued in the
restriction of \(\ker(\lambda)\) to \(\PP\mathcal{V}'\).

To provide \(\mathcal{X}\) with the structure of a family
of \((q;\mathbf{a} \setminus \mathbf{a}_{\mathrm{lin}})\)-tic schemes in
\(\PP\mathcal{V}'\), begin with its \((q;\mathbf{a})\)-tic tensor
\(\boldsymbol{\alpha} \colon \sO_S \to \mathcal{A}\) and evaluation map
\(\varepsilon \colon \pi^*\mathcal{A} \to \mathcal{E}\) with respect to
\(\PP\mathcal{V}\). Since \(\varepsilon\) is locally given by evaluation along
\(\pi\), the kernel
\[
\mathcal{A}_0 \coloneqq
\ker(\pi_*(\lambda \circ \varepsilon) \colon \mathcal{A} \to \mathcal{V}^\vee \otimes \mathcal{M})
\]
is a subbundle of \(\mathcal{A}\) with the property that \(\varepsilon\)
restricted to \(\pi^*\mathcal{A}_0\) factors through \(\ker(\lambda)\). Compose
this with the restriction map \(\ker(\lambda) \to \mathcal{E}'\), push along
\(\pi\), and consider the \(\sO_S\)-module
\[
\mathcal{A}' \coloneqq
\image(\mathcal{A}_0 \to \pi_*\ker(\lambda) \to \pi_*\mathcal{E}').
\]
Locally on \(S\), the morphism defining the image is the composite
\[
\bigoplus\nolimits_{a \in \mathbf{a} \setminus \mathbf{a}_{\mathrm{lin}}} \mathrm{S}^a(\mathcal{V}^\vee) \subseteq
\bigoplus\nolimits_{a \in \mathbf{a} \setminus \mathbf{a}_{\mathrm{lin}}} \Sym^{a(q)}(\mathcal{V}^\vee) \to
\bigoplus\nolimits_{a \in \mathbf{a} \setminus \mathbf{a}_{\mathrm{lin}}} \Sym^{a(q)}(\mathcal{V}'^\vee)
\]
of inclusion \((q;a)\)-tic subbundles followed by the restriction maps
from \(\mathcal{V}\) to \(\mathcal{V}'\), which implies
that \(\mathcal{A}'\) is locally a sum of the \((q;a)\)-tic bundles
\(\mathrm{S}^a(\mathcal{V}'^\vee)\). To construct a
\((q;\mathbf{a}')\)-tic tensor valued in \(\mathcal{A}'\), consider the
commutative diagram
\[
\begin{tikzcd}[row sep=1em, column sep=1.5em]
\mathcal{A}_0 \rar \dar[symbol={\subseteq}]
& \mathcal{A}' \rar[hook]
& \pi_*\mathcal{E}' \dar[symbol={\subseteq}] \\
\mathcal{A} \rar[hook]
& \pi_*\mathcal{E} \rar
& \pi_*(\mathcal{E}\rvert_{\PP\mathcal{V}'})\punct{.}
\end{tikzcd}
\]
Composing the bottom row with \(\boldsymbol{\alpha} \colon \sO_S \to \mathcal{A}\)
yields the map \(\pi_*(\sigma\rvert_{\PP\mathcal{V}'})\), which has vanishing
linear components, and so it factors through \(\pi_*\mathcal{E}'\). Since
\(\mathcal{A}_0\) and \(\mathcal{A}\) have the same image in
\(\pi_*(\mathcal{E}\rvert_{\PP\mathcal{V}'})\), the commutative diagram implies
that this then lifts to the required map
\(\boldsymbol{\alpha}' \colon \sO_S \to \mathcal{A}'\).
\end{proof}

Perhaps a more interesting simplification is possible when the collection of
profiles \(\mathbf{a} = \mathbf{a}_{\mathrm{pow}}\) consists only of nonreduced
profiles, meaning that each equation of \(\mathcal{X}\) over \(S\) is,
geometrically, a \(q\)-power. Upon adjoining suitable roots of the coefficients,
it is possible to take \(q\)-th roots of all the equations of
\(\mathcal{X}\) to obtain a scheme \(\mathcal{X}'\) with multi-profile
\(\mathbf{a}/t \coloneqq (a/t : a \in \mathbf{a})\):

\begin{Lemma}\label{families-frobenius-descent}
Let \(\mathcal{X}\) be a family of \((q;\mathbf{a})\)-tic schemes with
\(\mathbf{a} = \mathbf{a}_{\mathrm{pow}}\). Then there exists a family of
\((q;\mathbf{a}/t)\)-tic schemes \(\mathcal{X}'\) and 
a universal homeomorphism \(\mathcal{X}' \to \mathcal{X}\) fitting in the
commutative square
\[
\begin{tikzcd}
\mathcal{X}' \rar \dar & \mathcal{X} \dar \\ S \rar["\Fr"] & S\punct{.}
\end{tikzcd}
\]
\end{Lemma}

\begin{proof}
Write \(\Fr_{\PP\mathcal{V}/S} \colon \PP\mathcal{V} \to \PP\mathcal{V}^{[1]}\)
for the \(S\)-linear relative \(q\)-power Frobenius morphism: the morphism
which takes \(q\)-powers of the fibre coordinates and leaves coefficients
fixed, induced by the commutative diagram
\[
\begin{tikzcd}
\PP\mathcal{V} \rar["\Fr_{\PP\mathcal{V}/S}"'] \ar[dr, "\pi"', bend right=20] \ar[rr, "\Fr"', bend left=30]
& \PP\mathcal{V}^{[1]} \rar \dar["\pi^{[1]}"]
& \PP\mathcal{V} \dar["\pi"] \\
& S \rar["\Fr"]
& S
\end{tikzcd}
\]
where the right hand square witnesses
\(\PP\mathcal{V}^{[1]} = \PP\mathcal{V} \times_{S, \Fr} S\). The main point now
is that \(\mathcal{E}\) admits a descent
along \(\Fr_{\PP\mathcal{V}/S}\), meaning there is an
\(\sO_{\PP\mathcal{V}^{[1]}}\)-module \(\mathcal{E}'\) which pulls back to
\(\mathcal{E}\). To construct \(\mathcal{E}'\), adopt the notation from
\parref{families-transition-functions} so that
\[
\mathcal{E}\rvert_{\pi^{-1}(U_i)} \cong
\bigoplus\nolimits_{a \in \mathbf{a}} \sO_\pi(a(q))\rvert_{\pi^{-1}(U_i)}.
\]
That \(\mathbf{a} = \mathbf{a}_{\mathrm{pow}}\) means that each
\(\sO_\pi(a(q))\) is the pullback of \(\sO_{\pi^{[1]}}(a(q)/q)\) along the
relative Frobenius, and so \(\mathcal{E}\) admits Frobenius descents locally
over \(S\). As explained in \parref{families-transition-functions}, the
transition functions of \(\mathcal{E}\) over \(U_{i,j}\) have
\((a,b)\)-components induced by multiplication by a polynomial:
\[
f_{i,j,a,b} \colon
\sO_\pi(a(q))\rvert_{\pi^{-1}(U_{i.j})} \to
\sO_\pi(b(q))\rvert_{\pi^{-1}(U_{i,j})}.
\]
As in the proof of \parref{families-quotients}, multiplication by
\(f_{i,j,a,b}\) maps
\(\mathrm{S}^a(\mathcal{V}^\vee)\) to \(\mathrm{S}^b(\mathcal{V}^\vee)\).
Since \(a\) and \(b\) both have vanishing constant term, if \(f_{i,j,a,b}\) is
nonzero, it also must have profile with vanishing constant term; in other words,
each monomial appearing must be a \(q\)-power in the fibre coordinates. This means
that the \(f_{i,j,a,b}\) admit \(U_{i,j}\)-linear Frobenius descents \(g_{i,j,a,b}\)
to \(\PP\mathcal{V}^{[1]}\rvert_{U_{i,j}}\). This glues the local Frobenius
descents to give an \(\sO_{\PP\mathcal{V}^{[1]}}\)-module \(\mathcal{E}'\)
descending \(\mathcal{E}\) along \(\Fr_{\PP\mathcal{V}/S}\).

Pushing the canonical adjunction map
\(\mathcal{E}' \to \Fr_{\PP\mathcal{V}/S, *}\mathcal{E}\)
along \(\pi^{[1]}\) now yields an injection which is locally given by the
inclusion
\[
\bigoplus\nolimits_{a \in \mathbf{a}} \Sym^{a(q)/q}(\mathcal{V}^{[1],\vee}) \subseteq
\bigoplus\nolimits_{a \in \mathbf{a}} \Sym^{a(q)}(\mathcal{V}^\vee).
\]
Identifying the \(\sO_S\)-module \(\mathcal{A}\) in which the
\((q;\mathbf{a})\)-tic tensor \(\boldsymbol{\alpha}\) takes values
locally as
\[
\bigoplus\nolimits_{a \in \mathbf{a}} \mathrm{S}^a(\mathcal{V}^\vee) \cong
\bigoplus\nolimits_{a \in \mathbf{a}} \mathrm{S}^{a/t}(\mathcal{V}^{[1],\vee})
\]
shows that the injection \(\pi_*\varepsilon\) factors as
\[
\pi_*\varepsilon \colon
\mathcal{A}
\to \pi^{[1]}_*\mathcal{E}'
\to \pi_*\mathcal{E}.
\]
Adjunction along \(\pi^{[1]}\) thus provides a relative evaluation map
\(\varepsilon' \colon \pi^{[1],*}\mathcal{A} \to \mathcal{E}'\) and the
pullback of \(\boldsymbol{\alpha}\) produces a section
\(\sigma' \colon \sO_{\PP\mathcal{V}^{[1]}} \to \mathcal{E}'\): its vanishing
locus \(\mathcal{X}' \subseteq \PP\mathcal{V}^{[1]}\) is now the sought-after
\((q;\mathbf{a}/t)\)-tic scheme.
\end{proof}

\subsectiondash{Planing}\label{tlines-planing}
An \emph{\(r\)-planing} of a family \(\mathcal{X}\) of \((q;\mathbf{a})\)-tic
schemes in \(\pi \colon \PP\mathcal{V} \to S\) is the data of a closed
subscheme \(\mathcal{P} \coloneqq \PP\mathcal{U} \subseteq \mathcal{X}\)
corresponding to a subbundle \(\mathcal{U} \subseteq \mathcal{V}\) of rank
\(r+1\); in other words, \(\mathcal{P}\) is a family of \(r\)-planes contained
in \(\mathcal{X}\). The pair \(\mathcal{P} \subseteq \mathcal{X}\) is
referred to as a \emph{family of \(r\)-planed \((q;\mathbf{a})\)-tic schemes}
over the base scheme \(S\). Observe that an \(r\)-planing \(\mathcal{P}\) of
\(\mathcal{X}\) may be transformed into an \(r'\)-planing for any
\(0 \leq r' \leq r\) upon base change to the Grassmannian bundle
\(\mathbf{G}(r'+1,\mathcal{U}) \to S\) and taking the projective bundle on
the tautological subbundle of rank \(r'+1\).

\subsectiondash{Classifying maps}\label{families-classifying-map}
Let \(\mathcal{P} \subseteq \mathcal{X}\) be an \(r\)-planed family of
\((q;\mathbf{a})\)-tic schemes in a \(\PP^n\)-bundle over a base \(S\). At
least on open subsets of the base \(S\), the family \(\mathcal{X}\) and the
pair \(\mathcal{P} \subseteq \mathcal{X}\) induce classifying maps to the
parameter spaces
\begin{align*}
\qaticci_{\PP^n} & \coloneqq
\prod\nolimits_{a \in \mathbf{a}} \PP\mathrm{S}^a(\kk^{\oplus n+1})\;\text{and} \\
\mathbf{Inc}_{n,r,\mathbf{a}} & \coloneqq
\Set{
  ([U], [\boldsymbol{\alpha}]) \in \mathbf{G}(r+1,n+1) \times \qaticci_{\PP^n} :
  \PP U \subseteq X_{\boldsymbol{\alpha}}},
\end{align*}
encountered already in \parref{qatic-parameter-space} and
\parref{fano-theorem}. To describe these, it is convenient to assume that \(S\)
is integral, so that the classifying maps may be seen as rational maps
from \(S\) to one of the two parameter spaces; often, such a rational map will
stand in for a choice of classifying map constructed below.

Assume henceforth that the base scheme \(S\) is integral. For a nonempty
open subscheme \(S^\circ \subseteq S\), write \(\mathcal{V}^\circ\),
\(\mathcal{P}^\circ\), and \(\mathcal{X}^\circ\) for the restrictions of
\(\mathcal{V}\), \(\mathcal{P}\), and \(\mathcal{X}\) over \(S^\circ\). Pick
\(S^\circ\) over which \(\mathcal{V}^\circ\) is trivial and \(\varepsilon\) may
be identified with the evaluation map as in
\parref{families-complete-intersections}\ref{families-complete-intersections.eta}.
Fixing suitable trivializations identifies
\(\mathcal{X}^\circ\) as a \((q;\mathbf{a})\)-tic scheme in the projective
space \(\PP\mathcal{V}^\circ \cong \PP^n \times S^\circ\) defined by a
\((q;\mathbf{a})\)-tic tensor
\[
\boldsymbol{\alpha}^\circ \colon
\sO_{S^\circ} \to
\bigoplus\nolimits_{a \in \mathbf{a}} \mathrm{S}^a(\mathcal{V}^{\circ,\vee}).
\]
This tensor defines a morphism \(S^\circ \to \qaticci_{\PP^n}\) and, together
with \(\mathcal{P}^\circ\), a morphism
\(S^\circ \to \mathbf{Inc}_{n,r,\mathbf{a}}\). These are the
\emph{classifying morphisms} for the families \(\mathcal{X}^\circ\) and
\(\mathcal{P}^\circ \subseteq \mathcal{X}^\circ\), respectively.

\subsectiondash{Generic families}\label{families-genericity}
Different choices of trivialization produce classifying morphisms which differ
by automorphisms of the parameter spaces, so dominance of a classifying map
does not depend on any of the choices above; after all, dominance means
informally that the family contains the general member of the parameter space
in question. By way of terminology, call the family
\(\mathcal{P} \subseteq \mathcal{X}\) \emph{generic} if the
associated classifying map \(S \dashrightarrow \mathbf{Inc}_{n,r,\mathbf{a}}\) is
dominant.

Genericity propagates along many constructions. As a first example,
consider a family \(\mathcal{P} \subseteq \mathcal{X}\) of \(r\)-planed
\((q;\mathbf{a})\)-tic schemes in which some of the equations of
\(\mathcal{X}\) are linear: that is,
\(\mathbf{a}_{\mathrm{lin}} \neq \varnothing\) in the notation of
\parref{families-quotients}. Then \(\mathcal{P} \subseteq \mathcal{X}\) may be
considered as a family in the smaller projective bundle \(\PP\mathcal{V}'\) cut
out by those linear equations, as in \parref{families-remove-linear-equations}.
Since an equation in a projective subspace may be extended---in many ways!---to
an equation in a larger ambient projective space, it is easy to convince
oneself that if \(\mathcal{P} \subseteq \mathcal{X}\) is generic as a family in
\(\PP\mathcal{V}\), then it remains generic as a family in \(\PP\mathcal{V}'\).
A more careful argument is given in the following:

\begin{Lemma}\label{families-generic-in-subbundle}
Let \(\mathcal{P} \subseteq \mathcal{X}\) be a generic family of \(r\)-planed
\((q;\mathbf{a})\)-tic schemes in \(\pi \colon \PP\mathcal{V} \to S\) over an
integral base. Suppose that there is
a subbundle \(\pi' \colon \PP\mathcal{V}' \to S\) containing \(\mathcal{X}\)
as a family of \((q;\mathbf{a}')\)-tic schemes
where \(\varnothing \neq \mathbf{a}' \subsetneq \mathbf{a}\)
and \(\mathbf{a} \setminus \mathbf{a}' \subseteq \mathbf{a}_{\mathrm{lin}}\).
Then \(\mathcal{P} \subseteq \mathcal{X}\) is also generic viewed as a family
in \(\PP\mathcal{V}'\).
\end{Lemma}

\begin{proof}
It suffices to treat the universal case: Write \(\PP^n = \PP V\),
\(c \coloneqq \#\mathbf{a} \setminus \mathbf{a}'\), and let
\[
S \coloneqq
\{s \in \mathbf{Inc}_{V,r,\mathbf{a}} : \rank \phi_s \geq c\} \subseteq
\mathbf{G}(r+1,V) \times (q;\mathbf{a}')\operatorname{\bf\!-tics}_{\PP V} \times
\prod\nolimits_{i = 1}^c \PP V^\vee
\]
be the open subscheme of the incidence correspondence where maximal rank
is attained for the map
\[
\phi \colon
\sO_{\mathbf{Inc}_{V,r,\mathbf{a}}} \otimes_\kk V \to \bigoplus\nolimits_{i = 1}^c \mathcal{L}_i,
\;\text{where}\;
\mathcal{L}_i \coloneqq \pr_i^*\sO_{\PP V^\vee}(1),
\]
obtained by pulling back the evaluation maps
\(\sO_{\PP V^\vee} \otimes V \to \sO_{\PP V^\vee}(1)\) from the \(c\) factors
in the right-most product. Thus the restriction \(\mathcal{P} \subseteq \mathcal{X}\)
of the tautological family of \(r\)-planed \((q;\mathbf{a})\)-tic schemes
to \(S\) is contained in the projective subbundle on
\(\mathcal{V}' \coloneqq \ker(\phi\rvert_S)\) as a family of
\((q;\mathbf{a}')\)-tic schemes. To construct a classifying map
\(S \dashrightarrow \mathbf{Inc}_{V',r,\mathbf{a}'}\) as in
\parref{families-classifying-map} to the parameter space of
\(r\)-planed \((q;\mathbf{a}')\)-tic schemes in \(\PP^{n-c} = \PP V'\),
choose an open subscheme \(S^\circ \subseteq S\) on which
\(\mathcal{V}'\) is trivial. Dualizing the composition
\[
\sO_{S^\circ} \otimes V' \cong
\mathcal{V}'\rvert_{S^\circ} \subseteq
\sO_{S^\circ} \otimes V
\]
of a trivialization with the inclusion induces a map which takes the
tautological family of \((q;\mathbf{a})\)-tic tensors on \(V\) to a family of
\((q;\mathbf{a}')\)-tic
tensors on \(V'\), and hence a morphism \(S^\circ \to
\mathbf{Inc}_{V',r,\mathbf{a}'}\). It remains to see that some such morphism is
dominant.

Consider an explicit choice of \(S^\circ\): Upon choosing coordinates
\(\PP^n = \PP V\), the map \(\phi\) may be viewed as a \(c \times (n+1)\)
matrix \((a_{i,j})\), where the \(c\) linear equations defining \(\mathcal{X}\)
are given by
\[
\ell_i(x_0,\ldots,x_n) = a_{i,0} x_0 + \cdots + a_{i,n} x_n
\;\text{for \(i = 1,\ldots,c\)}.
\]
Let \(S^\circ \subseteq S\) be the open subscheme on which the rightmost
\(c \times c\) minor of \(\phi\) is non-vanishing. With a suitable
\(\sO_{S^\circ}\)-linear change of coordinates, the linear equations \(\ell_i\)
may be transformed to
\[
\ell_i(x_0,\ldots,x_n) = x_{n-c+i} - \sum\nolimits_{j = 0}^{n-c} b_{i,j} x_j
\;\text{for \(i = 1,\ldots,c\)}
\]
where the \(b_{i,j} \in \Gamma(S^\circ, \sO_{S^\circ})\). This provides a
splitting of the surjection \(\phi\rvert_{S^\circ}\); the corresponding
retraction provides a trivialization of \(\mathcal{V}'\rvert_{S^\circ}\)
identifying it with the first \(n-c+1\) summands of \(\sO_{S^\circ}^{\oplus
n+1}\), and the induced classifying morphism \(S^\circ \to
\mathbf{Inc}_{V',r,\mathbf{a}'}\) takes a \((q;\mathbf{a})\)-tic tensor on
\(\PP^n\) to that on \(\PP^{n-c}\) obtained by eliminating the last \(c\)
coordinates as prescribed by the linear equation \(\ell_i\). Described in this
way, it is straightforward that this classifying morphism is surjective:
Consider, for example, the closed subscheme \(T\) defined by \(b_{i,j} = 0\)
for all \(1 \leq i \leq c\) and \(0 \leq j \leq n-c\). Then any point of
\(\mathbf{Inc}_{V',r,\mathbf{a}'}\) can be lifted to a point of \(T\) just by
viewing the \((q;\mathbf{a}')\)-tic tensor in \(n-c+1\) variables as a
\((q;\mathbf{a}')\)-tic tensor in \(n+1\) variables; then augment this to a
\((q;\mathbf{a})\)-tic tensor by taking into account the linear equations
\(\ell_i\).
\end{proof}

\section{Highly tangent lines}\label{section-tlines}
Lines tangent to order \(\geq k\) to a projective scheme \(X \subseteq \PP^n\)
may be parameterized by a scheme of the form
\[
\mathbf{Tan}_k(X) \coloneqq
\{(x,[\ell]) \in X \times \mathbf{G}(2,n+1) : \mult_x(\ell \cap X) \geq k\}.
\]
Familiar cases include: When \(k = 1\), this is the restriction to \(X\) of the
variety of pointed lines in \(\PP^n\), and is the projective bundle on
\(\mathcal{T} \coloneqq \mathcal{T}_{\PP^n}(-1)\rvert_X\). When \(k = 2\), this
is the projectivized tangent bundle of \(X\). When \(X\) is a general
hypersurface of degree \(d\), it is well-known that the general fibre of the
projection \(\mathbf{Tan}_k(X) \to X\) to the \(x\)-coordinate is a complete
intersection of type \((k-1,k-2,\ldots,2,1)\) in an \((n-1)\)-dimensional
projective space.

If \(X\) is a general complete intersection of codimension \(c \geq 2\),
\(\mathbf{Tan}_k(X)\) usually consists of several components and will not
fibre in complete intersections over \(X\). Nonetheless, upon writing
\(X\) as an intersection \(H_1 \cap \cdots \cap H_c\) of hypersurfaces \(H_i\)
of degree \(d_i\), each of the schemes
\[
\mathbf{Tan}_k(X; H_i) \coloneqq
\{(x,\ell) :
\mult_x(\ell \cap H_i) \geq k\;\text{and}\;
\ell \subset H_j\;\text{for \(1 \leq j \leq c\) and \(j \neq i\)}\},
\]
parameterizing \(k\)-fold tangent lines to \(H_i\) and pointed lines in the
remaining \(H_j\), are distinguished components of \(\mathbf{Tan}_k(X)\) which
project onto \(X\) with complete intersection general fibres. The aim of this
section is to make sense of this construction and structure for a family of
\((q;\mathbf{a})\)-tic schemes.

Rather than discussing \(\mathbf{Tan}_k(X)\) for general \(k\), this section
focuses on the case of particular relevance, namely, \(k = d - 1\)
where \(d \coloneqq \deg X\):
\[
\mathbf{PenTa}(X)
\coloneqq \mathbf{Tan}_{d-1}(X)
\coloneqq \{(x,[\ell]) \in X \times \mathbf{G}(2,n+1) : \mult_x(\ell \cap X) \geq d-1\}.
\]
Lines parameterized by \(\mathbf{PenTa}(X)\) are called \emph{penultimate
tangents}. In words, this scheme of penultimate tangents parameterizes pointed
lines \(x \in \ell \subset \PP^n\) such that either \(\ell\) is contained in
\(X\) or else their intersection, viewed as a Cartier divisor on \(\ell\), is
of the form \(\ell \cap X = (d-1)x + x'\) for a \emph{residual point}
\(x' \in X\). The interest in this case is that extracting the residual point
often provides a rational map
\(\res \colon \mathbf{PenTa}(X) \dashrightarrow X\), which will be
studied in \S\parref{section-residual}.

\subsectiondash{Local situation}\label{tlines-explicit-construction}
When \(S = \Spec \kk\) and \(X \subseteq \PP^n\) is a \((q;\mathbf{a})\)-tic
scheme, the situation is easy to describe explicitly. With the notation in the
proof of \parref{fano-pointed-line-equations}, the scheme \(X_1\) of pointed
lines in \(X\) is defined over \(\mathrm{D}(x_n) \subseteq \PP^n\) as
\[
X_1\rvert_{\mathrm{D}(x_n)} =
\{
  (\mathbf{x}, \mathbf{y}) \in \mathrm{D}(x_n) \times \PP^{n-1} :
  f_{a,b}(\mathbf{x};\mathbf{y}) = 0
  \;\text{for}\;
  a \in \mathbf{a}
  \;\text{and}\;
  0 \preceq b \preceq a
\}.
\]
A point \((\mathbf{x},\mathbf{y}) \in \mathrm{D}(x_n) \times \PP^{n-1}\)
corresponds to a parameterized line
\((\xi:\eta) \mapsto \xi\mathbf{x} + \eta\mathbf{y}\), so the point
\(\mathbf{x}\) is the image of \((1:0)\). Writing \(H \subseteq \PP^n\) for the
hypersurface cut out by a \((q;a)\)-tic defining polynomial \(f_a\), the scheme
\(\mathbf{Tan}_k(H)\) of pointed lines tangent to
\(H\) to order \(\geq k\) is defined by \(f_{a,b}(\mathbf{x};\mathbf{y}) = 0\)
for \(0 \preceq b \preceq a\) satisfying \(b(q) \leq k\). Therefore, over
\(X \cap \mathrm{D}(x_n)\), the scheme \(\mathbf{PenTa}(X;H)\) of penultimate
tangents may be obtained from \(X_1\) by omitting the two equations
\(f_{a,a}(\mathbf{x};\mathbf{y}) = f_{a,a-1}(\mathbf{x};\mathbf{y}) = 0\),
presenting it as a family of \((q;\mathbf{a}')\)-tic where
\(\mathbf{a}' \coloneqq \mathbf{a}_1 \setminus (a,a-1)\); this is most
useful when \(a - 1\) is a nonzero profile, equivalently when the profile \(a\)
is nonlinear and reduced.

The next few paragraphs make sense of these observations in families.

\subsectiondash{Pointed lines}\label{tlines-pointed-lines}
Given a projective bundle \(\pi \colon \PP\mathcal{V} \to S\), write
\(\mathbf{L}\mathcal{V}\) for its space of pointed lines: namely, the
incidence correspondence between it and its relative Grassmannian
\(\mathbf{G}\mathcal{V}\) of lines. Projection to the line exhibits
\(\mathbf{L}\mathcal{V}\) as the universal line over
\(\mathbf{G}\mathcal{V}\), whereas projection
\(\pr_x \colon \mathbf{L}\mathcal{V} \to \PP\mathcal{V}\) to the point
identifies \(\mathbf{L}\mathcal{V}\) as the projective bundle
on the relative tangent bundle \(\mathcal{T}_\pi \otimes \sO_\pi(-1)\) of
\(\pi \colon \PP\mathcal{V} \to S\). These moduli descriptions imply that the
tautological bundles of \(\mathbf{L}\mathcal{V}\) fit into a canonical
short exact sequence
\[
0 \to
\sO_{\pr_x}(1) \to
\mathcal{S}^\vee \to
\pr_x^*\sO_\pi(1) \to
0
\]
where \(\mathcal{S}\) is the subbundle of rank \(2\). Pushing along \(\pr_x\)
yields the relative dual Euler sequence
\[
0 \to
\Omega^1_\pi \otimes \sO_\pi(1) \to
\pi^*\mathcal{V}^\vee \to
\sO_\pi(1) \to
0.
\]

Consider now a family \(\mathcal{X}\) of \((q;\mathbf{a})\)-tic schemes in
\(\pi \colon \PP\mathcal{V} \to S\). Its space of pointed
lines
\[
\mathcal{X}_1 \coloneqq
\{(x, \ell) \in \mathcal{X} \times_S \mathbf{F}_1(\mathcal{X}/S) : x \in \ell\}
\hookrightarrow \mathbf{L}\mathcal{V}\rvert_{\mathcal{X}}
\]
naturally embeds into the projective bundle on
\(\mathcal{T} \coloneqq \mathcal{T}_\pi \otimes \sO_\pi(-1)\rvert_{\mathcal{X}}\)
in which it has a \((q;\mathbf{a}_1)\)-tic structure, globalizing
\parref{fano-pointed-line-equations}:

\begin{Proposition}\label{tlines-pointed-lines-structure}
Let \(\mathcal{X}\) be a family of \((q;\mathbf{a})\)-tic schemes in a
projective bundle \(\pi \colon \PP\mathcal{V} \to S\). Then its space of
pointed lines \(\mathcal{X}_1\) admits a canonical
\((q;\mathbf{a}_1)\)-tic structure
in \(\rho \colon \PP\mathcal{T} \to \mathcal{X}\), where
\[
\mathcal{T} \coloneqq \mathcal{T}_\pi \otimes \sO_\pi(-1)\rvert_{\mathcal{X}},
\;\;
\mathbf{a}_1 \coloneqq
(b \in \Types : 0 \prec b \preceq a\;\text{for}\;a \in \mathbf{a}),
\;\;\text{and}\;\;
\mathcal{A}_1 \coloneqq \ker(\varepsilon \colon \pi^*\mathcal{A} \to \mathcal{E})\rvert_{\mathcal{X}}.
\]
\end{Proposition}

\begin{proof}
The task is to construct equations of \(\mathcal{X}_1\) in \(\PP\mathcal{T}\)
and to provide them with a \((q;\mathbf{a}_1)\)-tic structure over
\(\mathcal{X}\). First, the equations of \(\mathcal{X}_1\) in all of
\(\mathbf{L}\mathcal{V}\) are the pullback of those of the Fano scheme
\(\mathbf{F}_1(\mathcal{X}/S)\) in \(\mathbf{G}\mathcal{V}\), and the
latter are given by \(\pr_{\ell, *}\pr_x^*\sigma\), where \(\pr_x\) and
\(\pr_{\ell}\) are the projections out of
\(\mathbf{L}\mathcal{V}\), and
\(\sigma \colon \sO_{\PP\mathcal{V}} \to \mathcal{E}\) are the equations of
\(\mathcal{X}\) in \(\PP\mathcal{V}\). Writing
\(\sigma = \varepsilon \circ \pi^*\boldsymbol{\alpha}\) as in
\parref{families-complete-intersections}\ref{families-complete-intersections.factorize}
provides a factorization
\[
\pr_{\ell,*}\pr_x^*\sigma \colon
\sO_{\mathbf{G}\mathcal{V}} \stackrel{\gamma^*\boldsymbol{\alpha}}{\longrightarrow}
\gamma^*\mathcal{A} \stackrel{\varepsilon_0}{\longrightarrow}
\mathcal{E}_0 \subseteq
\pr_{\ell,*}\pr_x^*\mathcal{E}
\]
where \(\gamma \colon \mathbf{G}\mathcal{V} \to S\) is the structure map, and
\(\varepsilon_0 \colon \gamma^*\mathcal{A} \to \mathcal{E}_0\) is a canonical
map which is locally isomorphic to a direct sum of the evaluation
maps \(\mathrm{S}^a(\mathcal{V}^\vee) \to \mathrm{S}^a(\mathcal{S}^\vee)\),
globalizing \parref{fano-equations}: this uses the fact that
\(\pr_{\ell,*}\pr_x^*\sO_\pi(d) \cong \Sym^d(\mathcal{S}^\vee)\), the local
form of \(\varepsilon\) from
\parref{families-complete-intersections}\ref{families-complete-intersections.eta},
and the assumption that the multiplication map for \((q;\mathbf{a})\)-tic
tensors is injective. In summary, \(\mathbf{F}_1(\mathcal{X}/S)\) is cut out in
\(\mathbf{G}\mathcal{V}\) by the section
\(\sigma_0 \coloneqq \varepsilon_0 \circ \gamma^*\boldsymbol{\alpha} \colon \sO_{\mathbf{G}\mathcal{V}} \to \mathcal{E}_0\),
and so \(\mathcal{X}_1\) is cut out in
\(\mathbf{L}\mathcal{V}\) by the pullback \(\pr_\ell^*\sigma_0\).

Next, consider the restriction of \(\pr_\ell^*\sigma_0\) to \(\PP\mathcal{T}\):
On the one hand, evaluation along \(\pr_\ell\) provides a canonical map
\(
\xi \colon
\pr_\ell^*\mathcal{E}_0 \subseteq
\pr_\ell^*\pr_{\ell,*}\pr_x^*\mathcal{E} \to
\pr_x^*\mathcal{E}
\)
making the square
\[
\begin{tikzcd}
\pr_\ell^*\gamma^*\mathcal{A} \rar[equal] \dar["\pr_\ell^*\varepsilon_0"'] &
\pr_x^*\pi^*\mathcal{A} \dar["\pr_x^*\varepsilon"] \\
\pr_\ell^*\mathcal{E}_0 \rar["\xi"] &
\pr_x^*\mathcal{E}
\end{tikzcd}
\]
commute; in particular, this implies that
\(\pr_x^*\sigma = \xi \circ \pr_{\ell}^*\sigma_0\). Locally, this is the sum of
the canonical surjections
\(\mathrm{S}^a(\mathcal{S}^\vee) \to \pr_x^*\sO_\pi(a(q))\) induced by applying
\(\mathrm{S}^a\) to the first exact sequence of
\parref{tlines-pointed-lines}. On the other hand,
\(\PP\mathcal{T} = \mathbf{L}\mathcal{V}\rvert_{\mathcal{X}}\)
is cut out of \(\mathbf{L}\mathcal{V}\) by \(\pr_x^*\sigma\).
Together, this means that the restricted section
\(\pr_\ell^*\sigma_0\rvert_{\PP\mathcal{T}}\) induces a map
\[
\sigma_1 \colon
\sO_{\PP\mathcal{T}} \to
\mathcal{E}_1 \coloneqq \ker(\xi \colon \pr_\ell^*\mathcal{E}_0 \to \pr_x^*\mathcal{E})\rvert_{\PP\mathcal{T}}.
\]
which cuts out \(\mathcal{X}_1\) in \(\PP\mathcal{T}\).

Finally, for the \((q;\mathbf{a}_1)\)-tic structure, let
\(\mathcal{A}_1 \coloneqq \ker(\varepsilon \colon \pi^*\mathcal{A} \to \mathcal{E})\rvert_{\mathcal{X}}\)
and observe that, since
\(\sigma = \varepsilon \circ \pi^*\boldsymbol{\alpha}\) vanishes on
\(\mathcal{X}\), the restriction \(\pi^*\boldsymbol{\alpha}\rvert_{\mathcal{X}}\)
induces a map \(\boldsymbol{\alpha}_1 \colon \sO_{\mathcal{X}} \to \mathcal{A}_1\).
Pulling up along \(\rho \colon \PP\mathcal{T} \to \mathcal{X}\), the composite
\[
\rho^*\mathcal{A}_1 \subseteq
\pr_x^*\pi^*\mathcal{A}\rvert_{\PP\mathcal{T}} =
\pr_\ell^*\gamma^*\mathcal{A}\rvert_{\PP\mathcal{T}} \stackrel{\pr_\ell^*\varepsilon_0}{\longrightarrow}
\pr_\ell^*\mathcal{E}_0\rvert_{\PP\mathcal{T}} \stackrel{\xi}{\longrightarrow}
\pr_x^*\mathcal{E}\rvert_{\PP\mathcal{T}}
\]
vanishes since it is the restriction of \(\pr_x^*\varepsilon\), and so it
induces a canonical map
\(\varepsilon_1 \colon \rho^*\mathcal{A}_1 \to \mathcal{E}_1\). Tracing
through the construction and using
\(\sigma_0 = \varepsilon_0 \circ \gamma^*\boldsymbol{\alpha}\) then shows that
\(\sigma_1 = \varepsilon_1 \circ \rho^*\boldsymbol{\alpha}_1\).

It remains to describe the local structure of \(\varepsilon_1\). Begin with the
open cover \(S = \bigcup\nolimits_{i \in I} U_i\) and trivializations of
\(\varepsilon\) provided by
\parref{families-complete-intersections}\ref{families-complete-intersections.eta}.
Refine the open cover given by the \(\pi^{-1}(U_i)\) into an open covering
\(\mathcal{X} = \bigcup\nolimits_{j \in J} V_j\) on
which, additionally, the dual Euler sequence splits:
\[
\pi^*\mathcal{V}^\vee\rvert_{V_j} \cong
\mathcal{T}^\vee\rvert_{V_j} \oplus \sO_\pi(1)\rvert_{V_j}.
\]
Fix such a choice for each \(j \in J\). The restriction of \(\mathcal{A}_1\) to
\(V_j\) is then isomorphic to
\[
\bigoplus\nolimits_{a \in \mathbf{a}}
\ker\Big(
\mathrm{S}^a(\mathcal{T}^\vee \oplus \sO_\pi(1)) \to
\sO_\pi(a(q))
\Big)\rvert_{V_j}
\cong
\bigoplus\nolimits_{a \in \mathbf{a}}
\bigoplus\nolimits_{0 \prec b \preceq a}
\mathrm{S}^b(\mathcal{T}^\vee) \otimes \sO_\pi(a(q) - b(q))\rvert_{V_j}
\]
where the maps on the left are projection onto the
\(\mathrm{S}^a(\sO_\pi(1)) = \sO_\pi(a(q))\) factor. Now choose over \(V_j\)
a splitting of the first exact sequence in \parref{tlines-pointed-lines}
compatible with the evaluation map \(\pr_x^*\pi^*\mathcal{V}^\vee \to \mathcal{S}^\vee\).
Over each \(V_j\), the map
\(\varepsilon_1 \colon \rho^*\mathcal{A}_1 \to \mathcal{E}_1\) is isomorphic to
the sum of the evaluation maps
\[
\mathrm{S}^b(\ev_\rho) \otimes \id \colon
\rho^*\mathrm{S}^b(\mathcal{T}^\vee) \otimes \sO_\pi(a(q) - b(q)) \to
\sO_\rho(b(q)) \otimes \rho^*\sO_\pi(a(q) - b(q))
\]
as \(a\) ranges over \(\mathbf{a}\) and \(0 \prec b \preceq a\). This verifies
that \((\mathcal{X}_1, \boldsymbol{\alpha}_1, \varepsilon_1)\) indeed is a
family of \((q;\mathbf{a}_1)\)-tic schemes in \(\rho \colon \PP\mathcal{T} \to \mathcal{X}\),
as desired.
\end{proof}

In light of the discussion in \parref{tlines-explicit-construction}, it will be
useful to describe how to access, within the bundle
\(\mathcal{E}_1\) of equations for \(\mathcal{X}_1\) in \(\PP\mathcal{T}\)
constructed in \parref{tlines-pointed-lines-structure}, those of degree
\(a(q)\) and \(a(q) - 1\), at least when \(\mathcal{X}\) is a family of
\((q;a)\)-tic hypersurfaces with \(a\) nonlinear and reduced:

\begin{Lemma}\label{tlines-last-two-equations}
Let \(\mathcal{X}\) be a family of \((q;a)\)-tic hypersurfaces where \(a(t)\)
is nonlinear and reduced. The degree \(a(q)\) and \(a(q) - 1\) equations of the
family \(\mathcal{X}_1\) of \((q;\mathbf{a}_1)\)-tic schemes in
\(\rho \colon \PP\mathcal{T} \to \mathcal{X}\) constructed in
\parref{tlines-pointed-lines-structure} lie in a subbundle of the form
\(\mathcal{S}^\vee \otimes \sO_\rho(a(q) - 1) \subseteq \mathcal{E}_1\).
\end{Lemma}

\begin{proof}
Since \(\mathcal{X}\) is a family of \((q;a)\)-tic hypersurfaces, up to twisting
by line bundles from \(S\), its structure is given by a
tensor \(\alpha \colon \sO_S \to \mathrm{S}^a(\mathcal{V}^\vee)\) and the
evaluation map
\(\varepsilon \colon \pi^*\mathrm{S}^a(\mathcal{V}^\vee) \to \sO_\pi(a(q))\).
The construction of \parref{tlines-pointed-lines-structure} then shows that the
\((q;\mathbf{a}_1)\)-tic polynomials defining \(\mathcal{X}_1\) in
\(\rho \colon \PP\mathcal{T} \to \mathcal{X}\) take values in the bundle
\[
\mathcal{E}_1 =
\ker\big(\mathrm{S}^a(\mathcal{S}^\vee) \to \rho^*\sO_\pi(a(q))\big)
\]
where \(\mathcal{S}^\vee\) denotes the tautological rank \(2\) subbundle on
\(\PP\mathcal{T}\), and the morphism is the canonical quotient that arises
upon applying \(\mathrm{S}^a\) to the short exact sequence of
\(\sO_{\PP\mathcal{T}}\)-modules
\[
0 \to \sO_\rho(1) \to \mathcal{S}^\vee \to \rho^*\sO_\pi(1) \to 0
\]
as in \parref{tlines-pointed-lines}. The line subbundle \(\sO_\rho(1)\) gives
the fibre coordinates of \(\rho \colon \PP\mathcal{T} \to \mathcal{X}\), so
the degree \(a(q)\) and \(a(q) - 1\) equations of \(\mathcal{X}_1\) lie in the
deepest rank \(2\) subbundle of \(\mathcal{E}_1\) with respect to the filtration
obtained by applying \(\mathrm{S}^a\) to the short exact sequence. Since
\[
\mathrm{S}^a(\mathcal{S}^\vee) = \Sym^{a_0}(\mathcal{S}^\vee) \otimes
\Big(\bigotimes\nolimits_{j \geq 1} \Sym^{a_j}(\mathcal{S}^\vee)^{[j]}\Big),
\]
this is the deepest rank \(2\) subbundle of \(\Sym^{a_0}(\mathcal{S}^\vee)\)
twisted by the deepest line subbundle of each
\(\Sym^{a_j}(\mathcal{S}^\vee)^{[j]}\) for \(j \geq 1\). General facts about
rank \(2\) bundles shows that this is
\[
\Big(\mathcal{S}^\vee \otimes \sO_\rho(a_0 - 1)\Big) \otimes
\Big(\bigotimes\nolimits_{j \geq 1} \sO_\rho(a_j q^j)\Big) \cong
\mathcal{S}^\vee \otimes \sO_\rho(a(q) - 1).
\qedhere
\]
\end{proof}

\subsectiondash{Penultimate tangents}\label{tlines-penta-family}
Let \(\mathcal{X}\) be a family of \((q;\mathbf{a})\)-tic schemes in
\(\pi \colon \PP\mathcal{V} \to S\), and let \(\mathcal{X}_1\) be its
scheme of pointed lines equipped with its structure as a family of
\((q;\mathbf{a}_1)\)-tic schemes in \(\rho \colon \PP\mathcal{T} \to \mathcal{X}\)
as constructed in \parref{tlines-pointed-lines-structure}. Assume that
\(\mathbf{a}_{\mathrm{nlr}} \neq \varnothing\) and choose a subbundle
\(\mathrm{S}^a(\mathcal{V}^\vee) \otimes \mathcal{M} \subseteq \mathcal{A}_{\mathrm{nlr}}\)
as in \parref{families-maximal-subbundle}. To perform the construction
\parref{tlines-explicit-construction} in families, view the projective bundle
\(\mu \colon \PP\mathcal{M} \to S\) as the linear system parameterizing
\((q;a)\)-tic hypersurfaces containing \(\mathcal{X}\) corresponding to the
equations of \(\mathrm{S}^a(\mathcal{V}^\vee) \otimes \mathcal{M}\). On the
product \(\PP\mathcal{T} \times_S \PP\mathcal{M}\), the chosen subbundle of
\(\mathcal{A}\) together with the tautological line subbundle \(\sO_\mu(-1)\),
and the computation of \parref{tlines-last-two-equations}
distinguish subbundles
\[
\pr_1^*\mathcal{E}_1 \supseteq
\ker\big(\mathrm{S}^a(\mathcal{S}^\vee) \to \rho^*\sO_\pi(a(q))\big)
\boxtimes \sO_\mu(-1)
\supseteq
\big(\mathcal{S}^\vee \otimes \sO_\rho(a(q)-1)\big) \boxtimes \sO_\mu(-1)
\]
giving, respectively, the equations of the \((q;a)\)-tic hypersurfaces
parameterized by \(\PP\mathcal{M}\), and the corresponding degree
\(a(q)\) and \(a(q) - 1\) equations of 
\(\mathcal{X}_1 \times_S \PP\mathcal{M}\) over \(\PP\mathcal{M}\). The composition
\[
\widebar{\sigma}_1 \colon
\sO \to
\widebar{\mathcal{E}}_1 \coloneqq
\pr_1^*\mathcal{E}_1/
(\mathcal{S}^\vee \otimes \sO_\rho(a(q)-1)) \boxtimes \sO_\mu(-1)
\]
of \(\pr_1^*\sigma_1\) with the quotient map
\(\pr_1^*\mathcal{E}_1 \to \widebar{\mathcal{E}}_1\) defines the desired
component of penultimate tangents.

To keep notation consistent with what follows, assume additionally that there
is given an \(r\)-planing \(\mathcal{P} \subseteq \mathcal{X}\) as in
\parref{tlines-planing}. Let
\[
S' \coloneqq \mathcal{P} \times_S \PP\mathcal{M}
\;\;\text{and}\;\;
\PP\mathcal{V}' \coloneqq \PP\mathcal{T} \times_{\mathcal{X}} S' =
\PP\mathcal{T}\rvert_{\mathcal{P}} \times_S \PP\mathcal{M}
\]
where the second projection \(\pi' \colon \PP\mathcal{V}' \to S'\) exhibits the
fibre product as the projective bundle on
\(\mathcal{V}' \coloneqq \nu^*(\mathcal{T}\rvert_{\mathcal{P}})\) where
\(\nu \colon S' \to \mathcal{P}\) is the structure map. View
\(\PP\mathcal{V}'\) as a closed subscheme of
\(\PP\mathcal{T} \times_S \PP\mathcal{M}\) and let
\(\sigma' \colon \sO_{\PP\mathcal{V}'} \to \mathcal{E}'\) be the restriction of
the section \(\widebar{\sigma}_1\) above. The \emph{scheme of penultimate tangents}
associated with \(\mathcal{X}\) and the subbundle
\(\mathrm{S}^a(\mathcal{V}^\vee) \otimes \mathcal{M}\) over \(\mathcal{P}\) is
the vanishing locus
\[
\mathcal{X}' \coloneqq
\mathbf{PenTa}(\mathcal{X};\mathrm{S}^a(\mathcal{V}^\vee) \otimes \mathcal{M})\rvert_{\mathcal{P}} \coloneqq
\mathrm{V}(\sigma' \colon \sO_{\PP\mathcal{V}'} \to \mathcal{E}') \subseteq \PP\mathcal{V}'
\]
of the section \(\sigma'\) in \(\PP\mathcal{V}'\). All of this data fits into a
commutative diagram of schemes
\[
\begin{tikzcd}
\mathcal{X}_1\rvert_{\mathcal{P}} \times_{\mathcal{P}} S' \rar[symbol={\subseteq}] \dar &[-2em]
\mathcal{X}' \rar[symbol={\subseteq}] &[-2em]
\PP\mathcal{V}' \rar["\pi'"'] \dar["\pr_1"'] &
S' \dar["\nu"] \\
\mathcal{X}_1\rvert_{\mathcal{P}} \ar[rr,hook] &&
\PP\mathcal{T}\rvert_{\mathcal{P}} \rar["\rho"] &
\mathcal{P} \rar[symbol={\subseteq}] &[-2em]
\mathcal{X} \rar[symbol={\subseteq}] &[-2em]
\PP\mathcal{V} \rar["\pi"] &
S\punct{.}
\end{tikzcd}
\]

The defining equations \(\sigma'\) of \(\mathcal{X}'\) are essentially a subset
of the equations \(\sigma_1\) defining \(\mathcal{X}_1\), whence the
containment relation in the top left. This relationship between
\(\mathcal{X}'\) and \(\mathcal{X}_1\) further means that the
\((q;\mathbf{a}_1)\)-tic structure on \(\mathcal{X}_1\) induces a
\((q;\mathbf{a}')\)-tic structure on \(\mathcal{X}'\):

\begin{Proposition}\label{tlines-lower-degrees}
In the above setting, the scheme \(\mathcal{X}'\) of penultimate tangents
admits the structure of a family of \((q;\mathbf{a}')\)-tic schemes in
\(\pi' \colon \PP\mathcal{V}' \to S'\), where
\(\mathbf{a}' \coloneqq \mathbf{a}_1 \setminus (a, a-1)\).
\end{Proposition}

\begin{proof}
In describing the \((q;\mathbf{a}')\)-tic structure of \(\mathcal{X}'\),
we may replace \(S\) by \(\PP\mathcal{M}\) to simplify the setting of
\parref{tlines-penta-family} so as to assume that we have chosen a subbundle
\(\mathrm{S}^a(\mathcal{V}^\vee) \subseteq \mathcal{A}\) corresponding
to a family of \((q;a)\)-tic hypersurfaces containing
\(\mathcal{X}\). Thus \(S' = \mathcal{P} \subseteq \mathcal{X}\) and
\(\PP\mathcal{V}' = \PP\mathcal{T}\rvert_{\mathcal{P}}\).
The preimage under
\(\varepsilon_1 \colon \rho^*\mathcal{A}_1 \to \mathcal{E}_1\) of the
rank \(2\) subbundle from \parref{tlines-last-two-equations} corresponding to
the degree \(a(q)\) and \(a(q) - 1\) equations of the chosen family of
\((q;a)\)-tic equations over \(\mathcal{X}\) is the pullback of a subbundle
\(\mathcal{B} \subseteq \mathcal{A}_1\) which fits in an extension
\[
0 \to
\mathrm{S}^a(\mathcal{T}^\vee) \to
\mathcal{B} \to
\mathrm{S}^{a-1}(\mathcal{T}^\vee) \otimes \sO_\pi(1) \to
0
\]
arising from the functor \(\mathrm{S}^a\) applied to the relative dual Euler
sequence in \parref{tlines-pointed-lines}. Writing
\(\theta \colon \mathcal{A}_1 \to \widebar{\mathcal{A}}_1\) for the
corresponding quotient, there is thus an induced map
\(\bar{\varepsilon}_1 \colon \widebar{\mathcal{A}}_1 \to \widebar{\mathcal{E}}_1\).
The final arguments of \parref{tlines-pointed-lines-structure} may then be
adapted to show that the maps
\[
\boldsymbol{\alpha}' \coloneqq
(\theta \circ \boldsymbol{\alpha}_1)\rvert_{S'} \colon
\sO_{S'} \to \mathcal{A}' \coloneqq \widebar{\mathcal{A}}_1\rvert_{S'}
\;\;\text{and}\;\;
\varepsilon' \coloneqq \bar{\varepsilon}_1\rvert_{\PP\mathcal{V}'} \colon
\pi'^*\mathcal{A}' \to \mathcal{E}'
\]
define the \((q;\mathbf{a}')\)-tic structure of \(\mathcal{X}'\) in
\(\pi' \colon \PP\mathcal{V}' \to S'\).
\end{proof}

\subsectiondash{Induced planings}\label{tlines-induced-planing}
Let \(\mathcal{X}\) be a family of \((q;\mathbf{a})\)-tic schemes in a
projective bundle \(\pi \colon \PP\mathcal{V} \to S\). An \(r\)-planing
\(\mathcal{P} \subseteq \mathcal{X}\) induces a canonical \((r-1)\)-planing of
the scheme \(\mathcal{X}_1\rvert_{\mathcal{P}}\) of pointed lines over
\(\mathcal{P}\) from \parref{tlines-pointed-lines}: The twisted
tangent bundle of \(\mathcal{P}\) provides a rank \(r\) subbundle
\[
\mathcal{T}_{\mathcal{P}} \otimes \sO_\pi(-1) \subseteq
\mathcal{T}_{\PP\mathcal{V}} \otimes \sO_\pi(-1)\rvert_{\mathcal{P}} =
\mathcal{T}\rvert_{\mathcal{P}}
\]
whose associated projective bundle \(\mathcal{P}_1\) is contained in
\(\mathcal{X}_1\rvert_{\mathcal{P}}\); geometrically, \(\mathcal{P}_1\)
parameterizes pointed lines \((x, \ell)\) where \(\ell \subseteq \mathcal{P}\).
Any scheme \(\mathcal{X}'\) of penultimate tangents over \(\mathcal{P}\)
as constructed in \parref{tlines-penta-family} also inherits 
an \((r-1)\)-planing: simply set
\(\mathcal{P}' \coloneqq \mathcal{P}_1 \times_{\mathcal{P}} S'\) and observe
that 
\(\mathcal{X}_1\rvert_{\mathcal{P}} \times_{\mathcal{P}} S' \subseteq \mathcal{X}'\).

\subsectiondash{Genericity of the induced families}\label{tlines-genericity-induced-families}
The remainder of this section is concerned with genericity properties of the
families \(\mathcal{P}_1 \subseteq \mathcal{X}_1\rvert_{\mathcal{P}}\) of
pointed lines over \(\mathcal{P}\), and \(\mathcal{P}' \subseteq \mathcal{X}'\)
of penultimate tangents associated with a generic family, in the sense of
\parref{families-genericity}, of \(r\)-planed \((q;\mathbf{a})\)-tic schemes
\(\mathcal{P} \subseteq \mathcal{X}\) over an integral base \(S\). First, to
see that the families
\(\mathcal{P}_1 \subseteq \mathcal{X}_1\rvert_{\mathcal{P}}\) and
\(\mathcal{P}' \subseteq \mathcal{X}'\) are themselves generic, consider
the tautological situation over the incidence correspondence of \(r\)-planes
and \((q;\mathbf{a})\)-tic schemes:

\begin{Lemma}\label{unirational-classify-pointed-lines}
Let \(\mathcal{P} \subseteq \mathcal{X}\) be the tautological family of
\(r\)-planed \((q;\mathbf{a})\)-tic schemes over the incidence correspondence
\(S \coloneqq \mathbf{Inc}_{n,r,\mathbf{a}}\). The classifying map
\[
[\mathcal{P}_1 \subseteq \mathcal{X}_1\rvert_{\mathcal{P}}] \colon
\mathcal{P} \dashrightarrow \mathbf{Inc}_{n-1,r-1,\mathbf{a}_1}
\]
for the associated family of pointed lines is dominant.
\end{Lemma}

\begin{proof}
Consider the fibre of
\(\mathcal{P} \subseteq \mathcal{X} \subseteq \PP^n_S\) over a fixed
closed point \(x \in \PP^n\):
\[
\mathcal{P}_x =
\{([U],[\boldsymbol{\alpha}]) \in \mathbf{Inc}_{n,r,\mathbf{a}} :
x \in \PP U \subseteq X_{\boldsymbol{\alpha}}\}.
\]
There is a choice of classifying morphism whose domain of definition intersects
\(\mathcal{P}_x\), and such that its restriction
\(\Psi \colon \mathcal{P}_x \dashrightarrow \mathbf{Inc}_{n-1,r-1,\mathbf{a}_1}\)
acts as
\(
([U],[\boldsymbol{\alpha}]) \mapsto ([U/L], [\boldsymbol{\alpha}_x])
\),
where \(x = \PP L\) and \(\boldsymbol{\alpha}_x\) is the
\((q;\mathbf{a}_1)\)-tic tensor 
constructed in \parref{fano-pointed-line-equations}
defining \(X_{\boldsymbol{\alpha},1,x}\) in \(\PP^{n-1} = \PP(V/L)\). Since
\(
\mathbf{F}_{r-1}(X_{\boldsymbol{\alpha}_x}) \cong
\mathbf{F}_r(X_{\boldsymbol{\alpha}},x)
\)
as observed in \parref{fano-planes-through-point}, \(\Psi\) is dominant
if and only if the map
\[
\{
  [\boldsymbol{\alpha}] \in \qaticci_{\PP^n} : x \in X_{\boldsymbol{\alpha}}
\}
\dashrightarrow
(q;\mathbf{a}_1)\operatorname{\bf\!-tics}_{\PP^{n-1}} \colon
[\boldsymbol{\alpha}] \mapsto [\boldsymbol{\alpha}_x]
\]
is dominant. This is a product of rational maps determined by the linear maps
\[
\{\alpha \in \mathrm{S}^a(V^\vee) : \alpha\rvert_L = 0\} \to
\bigoplus\nolimits_{0 \prec b \preceq a} \mathrm{S}^b((V/L)^\vee) \colon
\alpha \mapsto (\alpha_b)_{0 \prec b \preceq a}
\]
where \(\alpha_b\) is the \(b\)-homogeneous component of \(\alpha\) upon
expansion at \(x\). It is straightforward to see from the computations of
\parref{fano-pointed-line-equations} and \parref{tlines-explicit-expansion}
that, upon choosing coordinates so that \(x = (0:\cdots:0:1)\),
the \(\alpha_b\) are uniquely determined from \(\alpha\) by the relation
\[
\alpha(x_0,\ldots,x_n) =
\sum\nolimits_{0 \prec b \preceq a} \alpha_b(x_0,\ldots,x_{n-1}) \cdot x_n^{a(q) - b(q)}.
\]
Combined with \parref{qatic-injective-mult}, it follows that the map
\(\alpha \mapsto (\alpha_b)_{0 \prec b \preceq a}\) is an isomorphism, and so
the corresponding map \([\boldsymbol{\alpha}] \mapsto [\boldsymbol{\alpha}_x]\)
on multi-projective spaces is dominant.
\end{proof}

For an analogous statement for penultimate tangents, observe that the
\((q;\mathbf{a})\)-tic tensor
\(\boldsymbol{\alpha} \colon \sO_S \to \mathcal{A}\) defining the tautological
family \(\mathcal{X} \subseteq \PP^n_S\) takes values in a split bundle of
the form
\[
\mathcal{A} = \bigoplus\nolimits_{a \in \mathbf{a}} \mathrm{S}^a(\mathcal{V}^\vee)
\]
where \(\mathcal{V}\) itself is a trivial \(\sO_S\)-module of rank \(n+1\).
Choose a summand \(\mathrm{S}^a(\mathcal{V}^\vee) \subseteq \mathcal{A}\) with
nonlinear and reduced profile \(a\), apply the constructions of
\parref{tlines-penta-family} and \parref{tlines-induced-planing}, and let
\(\mathcal{P}' \subseteq \mathcal{X}'\) be the resulting family of
\((r-1)\)-planed \((q;\mathbf{a}')\)-tic schemes, with
\(\mathbf{a}' = \mathbf{a}_1 \setminus (a,a-1)\), over
\(S' \coloneqq \mathcal{P}\).

\begin{Lemma}\label{unirational-classify-penultimate-tangents}
Let \(\mathcal{P} \subseteq \mathcal{X}\) be the tautological family of
\(r\)-planed \((q;\mathbf{a})\)-tic schemes over the incidence correspondence
\(S \coloneqq \mathbf{Inc}_{n,r,\mathbf{a}}\). For any choice of subbundle
\(\mathrm{S}^a(\mathcal{V}^\vee) \subseteq \mathcal{A}\) above, the classifying
map
\[
[\mathcal{P}' \subseteq \mathcal{X}'] \colon
S' \dashrightarrow \mathbf{Inc}_{n-1,r-1,\mathbf{a}'}
\]
for the associated family of penultimate tangents is dominant.
\end{Lemma}

\begin{proof}
The classifying map factors as the dominant classifying map
\(S' \dashrightarrow \mathbf{Inc}_{n-1,r-1,\mathbf{a}_1}\) from
\parref{unirational-classify-pointed-lines}, followed by the morphism
\(\pi \colon
\mathbf{Inc}_{n-1,r-1,\mathbf{a}_1} \to
\mathbf{Inc}_{n-1,r-1,\mathbf{a}'}
\)
which projects out the profile \(a\) and \(a-1\) components corresponding to
the tensor of the chosen \((q;a)\)-tic subbundle
\(\mathrm{S}^a(\mathcal{V}^\vee) \subseteq \mathcal{A}\). Writing
\(\PP^{n-1} = \PP\widebar{V}\), the fibres of \(\pi\) are isomorphic to
bi-projective spaces of the form
\[
\pi^{-1}([\widebar{U}], [\boldsymbol{\alpha}']) \cong
\PP(
\ker(\mathrm{S}^a(\widebar{V}^\vee) \to \mathrm{S}^a(\widebar{U}^\vee))
\times
\PP(
\ker(\mathrm{S}^{a-1}(\widebar{V}^\vee) \to \mathrm{S}^{a-1}(\widebar{U}^\vee)
)
\]
parameterizing the missing components of \(\boldsymbol{\alpha}'\).
Thus \(\pi\) is surjective, and the result follows.
\end{proof}

Consider an arbitrary family \(\mathcal{P} \subseteq \mathcal{X}\) of
\(r\)-planed \((q;\mathbf{a})\)-tic schemes over an integral base \(S\).
If the family \(\mathcal{P} \subseteq \mathcal{X}\) is generic in the sense
of \parref{families-genericity}, then so too are the families
\(\mathcal{P}_1 \subseteq \mathcal{X}_1\rvert_{\mathcal{P}}\) of pointed lines
over \(\mathcal{P}\) with its structure from
\parref{tlines-pointed-lines-structure}, and---whenever
\(\mathbf{a}_{\mathrm{nlr}} \neq \varnothing\)---any family
\(\mathcal{P}' \subseteq \mathcal{X}'\) of penultimate tangents as constructed
in \parref{tlines-penta-family} and \parref{tlines-lower-degrees}:

\begin{Proposition}\label{unirationality-generic-family}
Let \(\mathcal{P} \subseteq \mathcal{X}\) be a generic family of \(r\)-planed
\((q;\mathbf{a})\)-tic schemes.
\begin{enumerate}
\item\label{unirationality-generic-family.pointed-lines}
\(\mathcal{P}_1 \subseteq \mathcal{X}_1\rvert_{\mathcal{P}}\) 
is generic as a family of \((r-1)\)-planed \((q;\mathbf{a}_1)\)-tic schemes.
\item\label{unirationality-generic-family.penta}
If \(\mathbf{a}_{\mathrm{nlr}} \neq \varnothing\), then any
\(\mathcal{P}' \subseteq \mathcal{X}'\) of is a generic family of
\((r-1)\)-planed \((q;\mathbf{a}')\)-tic schemes.
\end{enumerate}
\end{Proposition}

\begin{proof}
The family \(\mathcal{P} \subseteq \mathcal{X}\) is locally on \(S\) the pull
back via a classifying map of the tautological family over
\(\mathbf{Inc}_{n,r,\mathbf{a}}\). The invariant construction of
\(\mathcal{P}_1 \subseteq \mathcal{X}_1\rvert_{\mathcal{P}}\) from
\parref{tlines-pointed-lines-structure} and \parref{tlines-induced-planing}
means that it, too, is pulled back from the corresponding construction over
the tautological family, and so \ref{unirationality-generic-family.pointed-lines}
follows from \parref{unirational-classify-pointed-lines}.

Assume now that \(\mathbf{a}_{\mathrm{nlr}} \neq \varnothing\), choose a
subbundle
\(\mathrm{S}^a(\mathcal{V}^\vee) \otimes \mathcal{M} \subseteq \mathcal{A}\)
for some nonlinear and reduced profile \(a\), and let \(\mathcal{P}' \subseteq \mathcal{X}'\)
be the corresponding family of penultimate tangents over \(S' = \mathcal{P} \times_S \PP\mathcal{M}\)
as constructed in \parref{tlines-penta-family}. The base change of the original
family \(\mathcal{P} \subseteq \mathcal{X}\) to \(\PP\mathcal{M}\) will remain
generic, so as in the proof of \parref{tlines-lower-degrees}, replace
\(S\) by \(\PP\mathcal{M}\) to assume that \(\mathcal{M} \cong \sO_S\) and
\(S' = \mathcal{P}\). Then, once again, locally over \(S\), this family
\(\mathcal{P}' \subseteq \mathcal{X}'\) is pulled back from the corresponding
construction over the incidence correspondence, and so
\ref{unirationality-generic-family.penta} follows from
\parref{unirational-classify-penultimate-tangents}.
\end{proof}

Combined with numerical assumptions on the ambient projective bundle dimension
\(n\) and the integer \(r\), genericity of the family
\(\mathcal{P}' \subseteq \mathcal{X}'\) yields geometric genericity properties
of the families \(\mathcal{X}_1\rvert_{\mathcal{P}} \to \mathcal{P}\) and 
\(\mathcal{X}' \to S\). The most useful for what follows is that whenever \(n\)
is sufficiently large depending on \(r\) and \(\mathbf{a}\), the general fibre
of either family has its expected dimension:

\begin{Proposition}\label{tlines-generic-complete-intersection}
Let \(\mathcal{P} \subseteq \mathcal{X}\) be a generic family of \(r\)-planed
\((q;\mathbf{a})\)-tic schemes in a \(\PP^n\)-bundle over an integral base \(S\).
If
\[
n \geq \max\big\{2r-1 + \#\mathbf{a}_1, r +
\frac{1}{r}
\sum\nolimits_{a \in \mathbf{a}}
\prod\nolimits_{j \geq 0} \binom{a_j + r}{r} -
\frac{1}{r}\#\mathbf{a}
\big\},
\]
then the general fibre of \(\mathcal{X}_1\rvert_{\mathcal{P}} \to \mathcal{P}\)
is a \((q;\mathbf{a}_1)\)-tic complete intersection in \(\PP^{n-1}\). Similarly,
the general fibre of any family \(\mathcal{X}' \to S'\) of penultimate tangents
is a \((q;\mathbf{a}')\)-tic complete intersection in \(\PP^{n-1}\).
\end{Proposition}

\begin{proof}
The hypothesis on \(n\) implies that \(n - 1 - \#\mathbf{a}_1 \geq 0\), so the
tautological family over
\((q;\mathbf{a}_1)\operatorname{\bf\!-tics}_{\PP^{n-1}}\) is generically a
complete intersection. Since the classifying map \(\mathcal{P} \dashrightarrow
\mathbf{Inc}_{n-1,r-1,\mathbf{a}_1}\)
is dominant by \parref{unirational-classify-pointed-lines}, the result would
follow if the projection
\(\mathbf{Inc}_{n-1,r-1,\mathbf{a}_1} \to (q;\mathbf{a}_1)\operatorname{\bf\!-tics}_{\PP^{n-1}}\)
were surjective. By \parref{fano-theorem}\ref{fano-theorem.smooth}, this
is the case whenever \(\delta_-(n-1,\mathbf{a}_1,r-1) \geq 0\),
and this inequality is equivalent to the hypothesis on \(n\), as seen in
\parref{fano-covered-in-planes}. The same argument applies for penultimate tangents,
using the facts that
\(S' \dashrightarrow \mathbf{Inc}_{n-1,r-1,\mathbf{a}'}\) is dominant by
\parref{unirationality-generic-family}\ref{unirationality-generic-family.penta}
and that \(\delta_-(n-1,\mathbf{a}',r-1) \geq \delta_-(n-1,\mathbf{a}_1,r-1)\).
\end{proof}

\section{Residual point map}\label{section-residual}
As mentioned at the beginning of \S\parref{section-tlines}, extracting the
residual point of intersection provides a rational map from a scheme of
penultimate lines back to the original projective scheme. The aim of this
section is to construct this rational map in the relative setting, for the
family \(\mathcal{X}'\) of penultimate lines constructed in
\parref{tlines-penta-family}, and to study when the resulting map
\(\res \colon \mathcal{X}' \dashrightarrow \mathcal{X}\) is dominant.

To begin, let \(\mathcal{P} \subseteq \mathcal{X}\) be a family of
\(r\)-planed \((q;\mathbf{a})\)-tic schemes in a projective bundle
\(\pi \colon \PP\mathcal{V} \to S\). Assume that
\(\mathcal{A}_{\mathrm{nlr}} \neq 0\), choose a subbundle
\(\mathrm{S}^a(\mathcal{V}^\vee) \otimes \mathcal{M} \subseteq \mathcal{A}_{\mathrm{nlr}}\)
as in \parref{families-maximal-subbundle}, and let
\(\mathcal{X}' \to S'\) be the associated family of penultimate tangents as in
\parref{tlines-penta-family}. To give an explicit description of the residual
point map, represent geometric points of \(\mathcal{X}'\) as quintuples
\((s, x, \ell, H, Y)\) where
\begin{itemize}
\item \(s\) is a geometric point of \(S\);
\item \(x \in \ell\) is a pointed line in \(\PP\mathcal{V}_s\) with \(x \in \mathcal{P}_s\);
\item \(H\) is a \((q;a)\)-tic hypersurface parameterized by the linear system
\(\PP\mathcal{M}_s\) as in \parref{tlines-last-two-equations}; and
\item \(Y\) is the vanishing locus of the remaining equations so that
\(\mathcal{X}_s = Y \cap H\) is a presentation of \(\mathcal{X}_s\) as
a \((q;\mathbf{a})\)-tic scheme.
\end{itemize}
This data is subject to the conditions that \(\ell \subseteq Y\) and
\[
\mult_x(\ell \cap H) \geq a(q) - 1.
\]
Note \(H\) is not uniquely determined in this representation. However, any
other \(H'\) representing the same point of \(\mathcal{X}'\) has equation
differing from that of \(H\) by a \((q;a)\)-tic polynomial in the ideal of
\(Y\). In particular, since \(\ell \subseteq Y\), the multiplicity condition is
well-posed. The residual point map is now described, and constructed, as
follows:

\begin{Proposition}\label{tlines-residual-point}
In the above setting, there exists a morphism over \(S\)
\[
\res \colon
\mathcal{X}' \setminus \pr_1^{-1}(\mathcal{X}_1\rvert_{\mathcal{P}}) \longrightarrow
\mathcal{X}
\]
acting on geometric points as
\(\res(s,x,\ell,H,Y) = \ell \cap H - (a(q) - 1)x \in \mathcal{X}_s\).
\end{Proposition}

\begin{proof}
To construct this map globally, continue with the notation in
\parref{tlines-penta-family}, and observe that the section
\(\pr_1^*\sigma_1 \colon \sO_{\PP\mathcal{V}'} \to \pr_1^*\mathcal{E}_1\)
defining \(\mathcal{X}_1\) in \(\PP\mathcal{T}\) pulled back to
\(\mathcal{X}'\) factors through a section
\[
\tau \colon
\sO_{\mathcal{X}'} \to
\big(\mathcal{S}^\vee \otimes \sO_\rho(a(q) - 1)\big) \boxtimes \sO_\mu(-1).
\]
Its value on a point \((s,x,\ell,H,Y)\) may be identified as the degree
\(a(q)\) polynomial defining \(\ell \cap H\). Twisting down by
\(\sO_\rho(1 - a(q)) \boxtimes \sO_\mu(1)\) factors out the \((a(q)-1)\)-fold
zero at \(x \in \ell \cap H\), at which point \(\tau\) may be viewed as a
family of linear forms on lines in \(\mathcal{V}\). Composing \(\tau\) with the
wedge product isomorphism
\(\mathcal{S}^\vee \cong \mathcal{S} \otimes \sO_\rho(1) \otimes \rho^*\sO_\pi(1)\),
which sends a linear form to its zero locus, yields a section
\[
\tau' \colon
\big(\sO_\rho(-a(q)) \otimes \rho^*\sO_\pi(1)\big) \boxtimes \sO_\mu(1) \longrightarrow
\pr_1^*\mathcal{S}
\]
whose value at \((s, x, \ell, H, Y)\) is the zero locus of the linear form
given by \(\tau\); in other words, this is the residual point of intersection
between \(\ell\) and \(H\). Including \(\mathcal{S}\) into the pullback of
\(\mathcal{V}\) then provides the map to
\(\mathcal{X} \subseteq \PP\mathcal{V}\). It is defined at points where
\(\tau'\) does not vanish which, from the description, are points where
\(\ell \subseteq H\): this is the preimage of
\(\mathcal{X}_1\rvert_{\mathcal{P}}\) in \(\mathcal{X}'\).
\end{proof}

When the families \(\mathcal{X}_1\rvert_{\mathcal{P}} \to \mathcal{P}\) and
\(\mathcal{X}' \to S'\) of pointed lines and penultimate tangents associated
with an \(r\)-planed filtered family \(\mathcal{P} \subseteq \mathcal{X}\) of
\((q;\mathbf{a})\)-tic complete intersections are themselves
generically complete intersections, \parref{tlines-residual-point} provides
a rational map \(\res \colon \mathcal{X}' \dashrightarrow \mathcal{X}\) which
takes a penultimate tangent to its residual point of intersection. The next
goal is to show that \(\res\) is dominant whenever \(r\) is sufficiently large
depending only on \(\mathbf{a}\). In the following statement, a property is
said to hold \emph{fibrewise} in a family over \(S\) if the property holds
upon restriction to each closed point of \(S\).

\begin{Proposition}\label{unirational-res-dominant}
Let \(\mathcal{P} \subseteq \mathcal{X}\) be a family of \(r\)-planed
\((q;\mathbf{a})\)-tic complete intersections over a scheme \(S\). If
the family \(\mathcal{X}_1\) of pointed lines is of expected relative dimension
over \(\mathcal{P}\) fibrewise over \(S\) and
\[
r \geq r_0(\mathbf{a}) \coloneqq
\Big(\sum\nolimits_{a \in \mathbf{a}} \prod\nolimits_{j \geq 0} (a_j + 1)\Big)
- 2 \cdot \#\mathbf{a} - 1,
\]
then, for any associated family \(\mathcal{X}'\) of penultimate tangents, the
residual point map \(\res \colon \mathcal{X}' \dashrightarrow \mathcal{X}\)
exists and is dominant fibrewise over \(S\).
\end{Proposition}

Begin with a few reductions: If \(\mathbf{a}_{\mathrm{nlr}} = \varnothing\),
then the statement is empty and there is nothing to prove. So assume otherwise
and fix any family \(\mathcal{X}'\) of penultimate tangents as constructed in
\parref{tlines-penta-family}. Using the fact that the first projection
\(\mathcal{X} \times_S \PP\mathcal{M} \to \mathcal{X}\) is surjective,
\(S\) may be replaced by the linear system \(\PP\mathcal{M}\) of hypersurfaces
relative to which \(\mathcal{X}'\) is constructed so as to assume
\(\mathcal{M} = \sO_S\) and \(\mathcal{X}'\) is a family over
\(S' = \mathcal{P}\). Since, fibrewise over \(S\), \(\mathcal{X}'\) is obtained
from \(\mathcal{X}_1\) by omitting two relatively ample divisors over their
common base \(\mathcal{P}\), the hypothesis that \(\mathcal{X}_1\) is of
expected dimension implies that \(\mathcal{X}'\) is also of expected dimension.
This guarantees, via \parref{tlines-residual-point}, that the residual point
map \(\res \colon \mathcal{X}' \dashrightarrow \mathcal{X}\) exists, and even
that its indeterminacy locus does not contain any fibre over \(S\). Given this,
since the statement is fibrewise over \(S\), it suffices to consider the
absolute case over \(\Spec \kk\).

\medskip
For the remainder of the proof, then, let \(P \subseteq X\) be an \(r\)-planed
\((q;\mathbf{a})\)-tic complete intersection in \(\PP^n\) over \(\kk\). Write
\(\mathbf{a} = (a_1,\ldots,a_c)\) and fix a presentation of the form
\[
X = H_1 \cap \cdots \cap H_c
\;\;\text{where}\; H_i\;\text{is a \((q;a_i)\)-tic hypersurface}.
\]
Perhaps after reordering, assume that \(a_1 \in \Types\) is nonlinear and
reduced, and let \(X'\) be the scheme of penultimate tangents with respect to
\(H_1\).

For each point \(y \in X \setminus P\), consider the closed subscheme of \(P\)
given by
\[
Z_y \coloneqq
\{z \in P :
\mult_z(\ell_{y,z} \cap H_1) \geq a_1(q) - 1\;\text{and}\;
\ell_{y,z} \subset H_i \;\text{for}\; 2 \leq i \leq c
\}
\]
where \(\ell_{y,z}\) is the line between \(y\) and \(z\). The task is to show
that the open subscheme \(Z_y^\circ \subseteq Z_y\) parameterizing lines
intersecting \(H_1\) at \(z\) with multiplicity exactly \(a_1(q) - 1\) is
nonempty for general \(y\). Toward this, observe that:

\begin{Lemma}\label{unirational-res-dominant-absolute.empty}
If \(Z_y^\circ = \varnothing\), then \(\ell_{y,z} \subseteq X\) for every
\(z \in Z_y\).
\end{Lemma}

\begin{proof}
Emptiness of \(Z_y^\circ\) means \(\ell_{y,z}\) intersects \(H_1\) with
multiplicity at least \(a_1(q)\) at \(z\). Since \(\ell_{y,z}\) also intersects
\(H_1\) at \(y\), \(\ell_{y,z}\) must be contained in \(H_1\), whence also
\(X\).
\end{proof}

Equations for \(Z_y\) are simple to describe, and yield the following dimension
estimate:

\begin{Lemma}\label{unirational-res-dominant-absolute.estimate}
\(\displaystyle
\dim Z_y
\geq
r - r_0(\mathbf{a})
%+ 2c + 1 - \sum\nolimits_{i = 1}^c \prod\nolimits_{j = 0}^{m_i} (a_{ij} + 1)
\) for all \(y \in X \setminus P\).
\end{Lemma}

\begin{proof}
Write \(P_y \coloneqq \langle y, P \rangle\) for the \((r+1)\)-plane spanned by
\(y\) and the \(r\)-plane \(P\), and view linear projection in \(P_y\) centred
at \(y\) as a rational map \(P_y \dashrightarrow P\) to identify \(P\) with the
space of lines in \(P_y\) through \(y\). This is
resolved into a morphism \(a \colon \widetilde{P}_y \to P\) on the blowup of
\(P_y\) at \(y\). As is standard, \(a\) exhibits \(\widetilde{P}_y\) as the
projective bundle on
\[
\mathcal{E} \cong \sO_P \oplus \sO_P(-1) \subseteq
\sO_P \otimes \mathrm{H}^0(P_y, \sO_{P_y}(1))^\vee
\]
where \(\sO_P\) corresponds to the point \(y \in P_y\) and \(\sO_P(-1)\) is the
tautological line subbundle in the subspace corresponding to \(P \subset P_y\).

Each of the linear sections \(H_{i,y} \coloneqq H_i \cap P_y\) is a
\((q;a_i)\)-tic hypersurface in \(P_y = \PP^{r+1}\) by
\parref{qatic-linear-sections}. As in \parref{families-blowups}, the total
transforms \(b^{-1}(H_{i,y})\) are then families of \((q;a_i)\)-tic
hypersurfaces over \(P\) whose equations in \(\widetilde{P}_y\) correspond to a
section
\[
\sigma_i \colon \sO_P \to
\mathrm{S}^{a_i}(\mathcal{E}^\vee) \cong
\bigotimes\nolimits_{j = 0}^{m_i} \Sym^{a_{i,j}}(\sO_P \oplus \sO_P(q^j)).
\]
Each line bundle summand corresponds to a coefficient of the equation of
\(H_{i,y}\) restricted to the line \(\ell_{y,z} = \PP\mathcal{E}_z\) as a
function of \(z \in P\); thus \(Z_{i,y} \coloneqq \mathrm{V}(\sigma_i)\)
parameterizes points \(z \in P\) for which \(\ell_{y.z} \subset H_{i,y}\).
Observe that some components of \(\sigma_i\) vanish for \emph{a priori} reasons:
Write \(\xi\) and \(\eta\) for local fibre coordinates of \(\PP\mathcal{E}\) so that
\(\xi = 0\) and \(\eta = 0\) define the points \(z\) and \(y\) on
\(\ell_{y,z} = \PP\mathcal{E}_z\). Since \(\ell_{y,z}\) intersects \(H_{i,y}\)
at both \(y\) and \(z\), the coefficients of \(\xi^{a_i(q)}\) and \(\eta^{a_i(q)}\)
vanish, and so
\[
\codim(Z_{i,y} \subseteq P)
\leq \rank\mathrm{S}^{a_i}(\mathcal{E}^\vee) - 2
= \prod\nolimits_{j = 0}^{m_i} (a_{i,j} + 1) - 2.
\]
The condition on \(H_{1,y}\) requires only that \(\ell_{y,z}\) intersect it at
\(z\) with multiplicity \(a_1(q) - 1\), meaning that the scheme of interest is,
rather than \(Z_{1,y}\), the potentially larger locus
\[
Z_{1,y}' \coloneqq \{z \in P : \mult_z(\ell_{y,z} \cap H_{1,y}) \geq a_1(q) - 1\}.
\]
This is cut out by the vanishing of all components of \(\sigma_1\) other than
that corresponding to the coefficient of \(\xi^{a_1(q) - 1}\eta\). Combined with
the above, this shows that the codimension of \(Z_{1,y}'\) in \(P\) is at most
\(\rank\mathrm{S}^{a_1}(\mathcal{E}^\vee) - 3\).
Since \(Z_y = Z_{1,y}' \cap Z_{2,y} \cap \cdots \cap Z_{c,y}\), the
codimension estimates give
\begin{align*}
\dim Z_y
& = \dim P - \codim(Z_y \subseteq P) \\
& \geq \dim P - \codim(Z_{1,y}' \subseteq P) - \sum\nolimits_{i = 2}^c \codim(Z_{i,y} \subseteq P) \\
& \geq r + 2c + 1 - \sum\nolimits_{i = 1}^c \prod\nolimits_{j = 0}^{m_i} (a_{i,j} + 1)
= r - r_0(\mathbf{a}).
\qedhere
\end{align*}
\end{proof}

\begin{proof}[Proof of \parref{unirational-res-dominant}]
Comparing the numerical hypothesis with
\parref{unirational-res-dominant-absolute.estimate} shows that \(Z_y\) is
nonempty for every \(y \in X \setminus P\). Suppose, however, that
\(Z_y^\circ = \varnothing\) for general \(y \in X \setminus P\). Derive a
contradiction by estimating the dimension of \(X_1\rvert_P\) in two ways: On
the one hand, it has its expected dimension by assumption; viewing \(X_1\) as
the universal line over the Fano scheme \(\mathbf{F}_1(X)\) and using
\parref{fano-equations}, this is
\[
\dim X_1\rvert_P
= (\dim \mathbf{F}_1(X) + 1) - \dim X + \dim P
= n + c + r - 1 - \sum\nolimits_{i = 1}^c \prod\nolimits_{j = 0}^{m_i} (a_{i,j} + 1).
\]
On the other hand, emptiness of \(Z_y^\circ\) together with
\parref{unirational-res-dominant-absolute.empty} gives a morphism
\[
\{(y,z) \in (X \setminus P) \times P : z \in Z_y\} \to X_1\rvert_P \colon
(y,z) \mapsto (z,[\ell_{y,z}]).
\]
Fibres of this map are contained in the points of the lines \(\ell_{y,z}\) and
so have dimension at most \(1\). Combined with the dimension estimate
\parref{unirational-res-dominant-absolute.estimate}, this gives
\[
\dim X_1\rvert_P \geq
\dim X + \dim Z_y - 1 \geq
n+c+r - \sum\nolimits_{i = 1}^c \prod\nolimits_{j = 0}^{m_i} (a_{i,j} + 1).
\]
Comparing the two quantities yields a contradiction. Therefore
\(Z_y^\circ \neq \varnothing\) for general \(y \in X\), and this means that
the residual point map \(\res \colon X' \dashrightarrow X\) is dominant.
\end{proof}

\section{Unirationality}\label{section-unirational}
The goal of this section is to establish Theorem \parref{intro-unirationality},
showing that general \((q;\mathbf{a})\)-tic complete intersection is
unirational once its dimension is sufficiently large, depending only on the
multi-profile \(\mathbf{a}\). The proof proceeds by inductively simplifying the
equations of the tautological family of \((q;\mathbf{a})\)-tic complete
intersections via the constructions of
\S\S\parref{section-families}--\parref{section-residual}. Induction takes
place over the set \(\Pi\) of all multi-profiles equipped with a somewhat
complicated partial ordering; the ordering is designed to keep track of
the multi-profiles that appear after passing to one of the following three
constructions: Frobenius descent as in \parref{families-frobenius-descent};
removing linear equations as in \parref{families-remove-linear-equations}; and
passing to the family of penultimate tangents as in
\parref{tlines-lower-degrees}.

\subsectiondash{Ordering multi-profiles}\label{unirational-poset}
Let \(\Pi\) be the set of all multi-profiles, and consider the relation
\(\preceq^\Pi\) defined as follows: The cover relations
\(\mathbf{a}' \precdot^\Pi \mathbf{a}\) for this relation come in three flavours,
depending on which parts of the canonical type decomposition from
\parref{families-quotients} are present:
\begin{enumerate}
\item\label{unirational-poset.pow}
If \(\mathbf{a} = \mathbf{a}_{\mathrm{pow}}\), then set
\(\mathbf{a}' \coloneqq \mathbf{a}/t \coloneqq (a/t : a \in \mathbf{a})\).
\item\label{unirational-poset.lin}
If \(\mathbf{a} = \mathbf{a}_{\mathrm{lin}} \sqcup \mathbf{a}_{\mathrm{pow}}\)
and \(\mathbf{a}_{\mathrm{lin}} \neq \varnothing\), then set
\(\mathbf{a}' \coloneqq \mathbf{a} \setminus \mathbf{a}_{\mathrm{lin}}\).
\item\label{unirational-poset.nlr}
If \(\mathbf{a}_{\mathrm{nlr}} \neq \varnothing\), then let
\(a_0 \in \mathbf{a}_{\mathrm{nlr}}\) be any element with maximal
coefficient sum \(a_0(1)\), and set
\[
\mathbf{a}' \coloneqq
(b \in \Types : 0 \prec b \preceq a \;\;\text{for}\; a \in \mathbf{a}) \setminus (a_0,a_0 -1).
\]
\end{enumerate}
In general, two multi-profiles satisfy \(\mathbf{a}' \preceq^\Pi \mathbf{a}\) if and only if
\(\mathbf{a}' = \mathbf{a}\) or else they are connected by a finite sequence
of the above cover relations:
\[
\mathbf{a}' =
\mathbf{a}_n \precdot^\Pi
\mathbf{a}_{n-1} \precdot^\Pi
\cdots \precdot^\Pi
\mathbf{a}_1 \precdot^\Pi
\mathbf{a}_0 = \mathbf{a}.
\]

\subsectiondash{Examples}\label{unirational-poset-examples}
The following are a few examples illustrating properties of the poset
\((\Pi,\preceq^\Pi)\):
\begin{enumerate}
\item\label{unirational-poset-examples.constant}
If \(a(t) = d\) is a constant, then the interval \([\varnothing, d]^\Pi\)
between \(\varnothing\) and \(d\) is totally ordered. Explicitly, the
Hasse diagrams for \(3 \leq d \leq 5\) take the form:
\[
(3) - (1) - \varnothing,
\quad
(4) - (2, 1) - (1) - \varnothing, \quad
(5) - (3, 2, 1) - (2, 1,1,1) - (1,1,1) - \varnothing.
\]
To depict the Hasse diagram for \(d = 6\), write \(k^m\)
for the profile \(k\) appearing \(m\) times:
\[
(6)
- (4,3,2,1)
- (3,2^3,1^4)
- (2^3,1^8)
- (2^2,1^{10})
- (2,1^{11})
- (1^{11})
- \varnothing.
\]
\item\label{unirational-poset-examples.qbic}
The Hasse diagram for the interval \([\varnothing, 1+t]^\Pi\) is: \((1+t) - (1)
- \varnothing\).
\item\label{unirational-poset-examples.cubic}
The cover relation \parref{unirational-poset}\ref{unirational-poset.pow}
appears in the Hasse diagram for \([\varnothing, 1+2t]^\Pi\):
\[
(1+2t) - (1+t, t, 1) - (t,1,1) - (t) - (1) - \varnothing.
\]
\item\label{unirational-poset-examples.incomparable}
Consider the multi-profile \(\mathbf{a} \coloneqq ((q+1)t + 1, t^2+(q+1))\) of
\(\sqsubseteq\)-incomparable profiles of numerical degree \(q^2+q+1\) from
\parref{qatics-ordering-properties}\ref{qatics-ordering-properties.not-total}.
Specializing to \(q = 2\), the two multi-profiles covered by \(\mathbf{a}\) via
the relation \parref{unirational-poset}\ref{unirational-poset.nlr} are obtained
from
\[
\mathbf{a}_1 =
(3t+1, 3t, 2t+1, 2t, t+1, t, 1) \cup (t^2+3, t^2+2, t^2+1, t^2, 3, 2, 1)
\]
by omitting either \((3t+1, 3t)\) or else \((t^2+3, t^2+2)\).
\item\label{unirational-poset-examples.catenary}
By considering a slight variant of
\ref{unirational-poset-examples.incomparable}, it is possible to see that
lengths of paths from \(\mathbf{a}\) to \(\varnothing\) need not be the same.
For instance, this can be seen with the pair
\(\mathbf{a} = ((q+2)t + 1, t^2 + (2q+1))\) of \(\sqsubseteq\)-incomparable
profiles of numerical degree \(q^2+2q+1\).
\item\label{unirational-poset-examples.degree-drop}
Maximal coefficient sums in a multi-profile does not necessarily drop along
a cover relation of the form
\parref{unirational-poset}\ref{unirational-poset.nlr}. This is because the
maximum may be achieved by a member of \(\mathbf{a}_{\mathrm{pow}}\). For
example, consider the Hasse diagram for \(\mathbf{a} = (2,3t)\):
\[
(2,3t) - (3t) - (3) - (1) - \varnothing.
\]
\end{enumerate}

\subsectiondash{An invariant}\label{unirational-poset-degrees}
Basic properties of \((\Pi,\preceq^\Pi)\) require some effort to establish
since, for example, multi-profiles may grow bigger along the cover relations
defined in \parref{unirational-poset} and maximal coefficient sums of profiles
does not necessarily decrease along \(\preceq^\Pi\) as in
\parref{unirational-poset-examples}\ref{unirational-poset-examples.degree-drop}.
To address these difficulties, consider the three integer-valued functions
\[
\delta(\mathbf{a}) \coloneqq \max\{\deg_t a(t) : a \in \mathbf{a}\},
\;\;
\sigma(\mathbf{a}) \coloneqq \max\{a(1) : a \in \mathbf{a}_{\mathrm{nlr}}\},
\;\;
\mu(\mathbf{a}) \coloneqq \#\{a \in \mathbf{a}_{\mathrm{nlr}} : a(1) = \mu_{\mathrm{nlr}}(\mathbf{a})\},
\]
where the maximum of an empty set is \(0\). Define
\(\phi \colon \Pi \to \mathbf{Z}_{\geq 0}^4\) by
\(\phi(\mathbf{a}) \coloneqq (\delta(\mathbf{a}), \sigma(\mathbf{a}), \mu(\mathbf{a}), \#\mathbf{a}_{\mathrm{nlr}})\).
A case analysis shows that
\[
\text{if \(\mathbf{a}' \precdot^\Pi \mathbf{a}\) is a cover relation in
\(\Pi\), then \(\phi(\mathbf{a}') <^{\mathrm{lex}} \phi(\mathbf{a})\)},
\]
where \(<^{\mathrm{lex}}\) is the lexicographical ordering on
\(\mathbf{Z}_{\geq 0}^4\). Thus this invariant provides a relation preserving
function
\(\phi \colon (\Pi, \preceq^\Pi) \to (\mathbf{Z}_{\geq 0}^4, \leq^{\mathrm{lex}})\).
This makes it possible to establish the basic properties of \(\preceq^\Pi\):

\begin{Lemma}\label{unirational-poset-properties}
\(\preceq^\Pi\) defines a partial ordering on \(\Pi\) with a unique bottom
element \(\varnothing \in \Pi\).
\end{Lemma}

\begin{proof}
To see that \(\preceq^\Pi\) is a partial ordering, it remains to establish
antisymmetry. Assume \(\mathbf{a} \preceq^\Pi \mathbf{b}\) and
\(\mathbf{b} \preceq^\Pi \mathbf{a}\). Since
\((\mathbf{Z}_{\geq 0}^4, \leq^{\mathrm{lex}})\) is itself a poset,
\parref{unirational-poset-degrees} implies
\(\phi(\mathbf{a}) = \phi(\mathbf{b})\). But then the sequence of cover
relations witnessing \(\mathbf{a} \preceq^\Pi \mathbf{b}\) must be empty, since
would \(\phi\) strictly decrease along each step, so \(\mathbf{a} = \mathbf{b}\).

Since nothing can precede \(\varnothing\) in \(\preceq^\Pi\), the remaining
statement is that \(\varnothing\) can be reached from any multi-profile
\(\mathbf{a}\) through a finite sequence of cover relations. Since the
invariant \(\phi(\mathbf{a})\) is a sequence of nonnegative integers and
lexicographically strictly decreases along each cover relation, it must
eventually reach \((0,0,0,0)\) after finitely many steps. But
\(\phi^{-1}(0,0,0,0) = \{\varnothing\}\), so
\(\varnothing \preceq^\Pi \mathbf{a}\).
\end{proof}

\begin{Lemma}\label{unirational-poset-finiteness}
For any \(\mathbf{a} \in \Pi\), the interval
\(
[\varnothing, \mathbf{a}]^\Pi \coloneqq
\{ \mathbf{a}' \in \Pi : \varnothing \preceq^\Pi \mathbf{a}' \preceq^\Pi \mathbf{a} \}
\)
is finite.
\end{Lemma}

\begin{proof}
Since \(\preceq^\Pi\) is locally finite by construction, it suffices to see
that the length of any path from \(\mathbf{a}\) down to \(\varnothing\) is
bounded. Proceed by induction on
\(\phi(\mathbf{a}) \in \mathbf{Z}_{\geq 0}^4\) ordered lexicographically.
The base case is \(\phi(\varnothing) = (0,0,0,0)\), in which there is nothing
to prove. Given a nonzero \((\delta, \sigma, \mu, \lambda) \in \mathbf{Z}_{\geq 0}^4\),
inductively assume that for every multi-profile \(\mathbf{a}'\) with
\(\phi(\mathbf{a}') <^{\mathrm{lex}} (\delta, \sigma, \mu, \lambda)\),
there exists an integer \(L = L(\mathbf{a}')\) such that any sequence
of cover relations between \(\mathbf{a}'\) and \(\varnothing\) has length at
most \(L\). Consider a multi-profile
\(\mathbf{a} \in \phi^{-1}(\delta, \sigma, \mu, \lambda)\) and consider three
cases:

Suppose that \(\sigma = 0\), so that \(\mathbf{a}_{\mathrm{nlr}} = \varnothing\).
In the case \(\delta = 0\), then \(\mathbf{a} = \mathbf{a}_{\mathrm{nlr}}\)
and a single cover relation of type
\parref{unirational-poset}\ref{unirational-poset.lin} brings \(\mathbf{a}\)
to \(\varnothing\). If \(\delta > 0\),
then any path from \(\mathbf{a}\) to \(\varnothing\) begins with a step of
the relation \parref{unirational-poset}\ref{unirational-poset.pow}. This
produces a new multi-profile \(\mathbf{a}'\) with
\(\delta(\mathbf{a}') = \delta - 1 < \delta\) and so induction applies.
Therefore, in this case, any path from \(\mathbf{a}\) to \(\varnothing\) has
length at most \(L(\mathbf{a}) \coloneqq L(\mathbf{a}') + 1\).

If \(\sigma \neq 0\), then any sequence of cover relations from \(\mathbf{a}\)
down to \(\varnothing\) begins with the relation
\parref{unirational-poset}\ref{unirational-poset.nlr}. Each of the \(\mu\)
choices for applying this relation leads to a multi-profile \(\mathbf{a}'\)
with a lexicographically smaller \(\phi(\mathbf{a}')\). Taking \(L'\) to be the
maximum over these \(\mathbf{a}'\) of the bounds \(L(\mathbf{a}')\) provided by
induction, any path from \(\mathbf{a}\) to \(\varnothing\) is bounded by
\(L(\mathbf{a}) \coloneqq L' + 1\) cover relations. This concludes the induction.
\end{proof}

\subsectiondash{Numbers}\label{unirational-numbers}
Inductively define numerical functions on the poset \((\Pi, \preceq^\Pi)\) as
follows: Let \(r \geq 0\) be an integer and \(\mathbf{a}\) be a multi-profile.
Assuming that \(\mathbf{a}_{\mathrm{nlr}} \neq \varnothing\), define
\begin{align*}
r(\mathbf{a}) & \coloneqq
\max\{r_0(\mathbf{a})\} \cup
\{r(\mathbf{a}') + 1 :
  \mathbf{a}' \precdot^\Pi \mathbf{a}
\}, \\
n_1(\mathbf{a},r) & \coloneqq
\max\{2r-1+ \#\mathbf{a}_1\} \cup
\{n_1(\mathbf{a}', r-1) + 1 : \mathbf{a}' \precdot^\Pi \mathbf{a}\}, \;\text{and} \\
n_2(\mathbf{a},r) & \coloneqq
\max
\Big\{
\big\lceil
r+
\frac{1}{r}
\sum\nolimits_{a \in \mathbf{a}}
\prod\nolimits_{j \geq 0} \binom{a_j + r}{r}
-\frac{1}{r}\#\mathbf{a}\big\rceil\Big\}
\cup \{n_2(\mathbf{a}', r-1) + 1 : \mathbf{a}' \precdot^\Pi \mathbf{a}\},
\end{align*}
where \(r_0(\mathbf{a})\) is defined in \parref{unirational-res-dominant},
\(\mathbf{a}'\) in the second set ranges over elements of \(\Pi\)
covered by \(\mathbf{a}\) as in \parref{unirational-poset}, and
\(\lceil \cdot \rceil\) denotes the ceiling function. One may verify that
\(n_1(\mathbf{a},r) \leq n_2(\mathbf{a},r)\) unless \(\mathbf{a} = (2t^k) \cup
\mathbf{a}'\) where \(\mathbf{a}'\) consist of profiles of the form \(t^m\).
In the case that \(\mathbf{a}_{\mathrm{nlr}} = \varnothing\), define
\begin{align*}
r(\mathbf{a}) & \coloneqq
\max\{r_0(\mathbf{a})\} \cup \{r(\mathbf{a}') : \mathbf{a}' \precdot^\Pi \mathbf{a}\}, \\
n_1(\mathbf{a},r) & \coloneqq
\max\{2r-1+ \#\mathbf{a}_1\} \cup \{n_1(\mathbf{a}', r) : \mathbf{a}' \precdot^\Pi \mathbf{a}\},\;\text{and} \\
n_2(\mathbf{a},r) & \coloneqq
\max
\Big\{
\big\lceil
r+
\frac{1}{r}
\sum\nolimits_{a \in \mathbf{a}}
\prod\nolimits_{j \geq 0} \binom{a_j + r}{r}
-\frac{1}{r}\#\mathbf{a}\big\rceil\Big\}
\cup \{n_2(\mathbf{a}', r) : \mathbf{a}' \precdot^\Pi \mathbf{a}\},
\end{align*}
where the notation is as before. Here, there is precisely one multi-profile
\(\mathbf{a}'\) covered by \(\mathbf{a}\), and a direct
computation shows that \(r_0(\mathbf{a}) = r_0(\mathbf{a}')\), and hence
\(r(\mathbf{a}) = r(\mathbf{a}')\), and also
\(n_i(\mathbf{a}',r) = n_i(\mathbf{a},r) - \#\mathbf{a}_{\mathrm{lin}}\)
for \(i = 1,2\). Finally, for all \(\mathbf{a} \in \Pi\) and \(r \geq 0\), define
\[
n_0(\mathbf{a},r) \coloneqq \max\{n_1(\mathbf{a},r), n_2(\mathbf{a},r)\}
\]

With these, it is possible to formulate the precise unirationality result:

\begin{Proposition}\label{unirational-in-families}
Let \(\mathcal{P} \subseteq \mathcal{X}\) be a generic family of \(r\)-planed 
\((q;\mathbf{a})\)-tic complete intersections in a \(\PP^n\)-bundle over an
integral base scheme \(S\) with \(r \geq r(\mathbf{a})\). If
\(n \geq n_0(\mathbf{a}, r)\), then the general fibre of \(\mathcal{X}\) over
\(S\) is unirational.
\end{Proposition}

\begin{proof}
Proceed by induction on \(\mathbf{a}\) along the poset \((\Pi,\preceq^\Pi)\).
The base case is when \(\mathbf{a} = \varnothing\), in which case
\(\mathcal{X} = \PP\mathcal{V}\) and each fibre over \(S\) is even rational.
Let \(\mathbf{a} \neq \varnothing\) and inductively assume that the conclusion
holds for all multi-profiles preceding it in \(\Pi\). The task is to construct
a new generic \(r'\)-planed
family \(\mathcal{P}' \subseteq \mathcal{X}'\) of \((q;\mathbf{a}')\)-tic
complete intersections in a \(\PP^{n'}\)-bundle over an integral base \(S'\)
with \(\mathbf{a}' \prec^\Pi \mathbf{a}\), \(r' \geq r(\mathbf{a}')\), \(n'
\geq n_0(\mathbf{a}',r')\), and all fitting into a commutative diagram of the
form
\[
\begin{tikzcd}
\mathcal{X}' \rar[dashed] \dar & \mathcal{X} \dar \\
S' \rar & S
\end{tikzcd}
\]
where \(S' \to S\) is dominant and \(\mathcal{X}' \dashrightarrow
\mathcal{X}\) is fibrewise dominant. Induction would then give the result
since \(\varnothing\) is the unique bottom element by
\parref{unirational-poset-properties} and the interval
\([\varnothing, \mathbf{a}]^\Pi\) is finite by
\parref{unirational-poset-finiteness}. Decompose the multi-profile as
\(\mathbf{a} = \mathbf{a}_{\mathrm{lin}} \sqcup \mathbf{a}_{\mathrm{pow}} \sqcup \mathbf{a}_{\mathrm{nlr}}\)
into types as in \parref{families-quotients}, and proceed in three cases:

If \(\mathbf{a}_{\mathrm{nlr}} \neq \varnothing\), apply
\parref{families-maximal-subbundle} to choose a nonzero subbundle
\(\mathrm{S}^{a_0}(\mathcal{V}^\vee) \otimes \mathcal{M} \subseteq \mathcal{A}_{\mathrm{nlr}}\)
where \(a_0 \in \mathbf{a}_{\mathrm{nlr}}\) has maximal coefficient sum \(a_0(1)\).
Applying the penultimate line constructions of \parref{tlines-penta-family},
\parref{tlines-lower-degrees}, and \parref{tlines-induced-planing} provides
a family of \((r-1)\)-planed \((q;\mathbf{a}')\)-tic schemes in a 
\(\PP^{n-1}\)-bundle over \(\mathcal{P} \times_S \PP\mathcal{M}\), where
\[
\mathbf{a}' =
(a' \in \Types : 0 \prec a' \preceq a\;\text{for \(a \in \mathbf{a}\)})
\setminus (a_0, a_0-1).
\]
This family is generic by \parref{unirationality-generic-family}. The choice of
\(n\) together with \parref{tlines-generic-complete-intersection} provide a
nonempty open subscheme \(S' \subseteq \mathcal{P} \times_S \PP\mathcal{M}\)
over which the family \(\mathcal{P}' \subseteq \mathcal{X}'\) is a
\((q;\mathbf{a}')\)-tic complete intersection. The residual point map
\(\res \colon \mathcal{X}' \dashrightarrow \mathcal{X}\) from
\parref{tlines-residual-point} exists and is fibrewise dominant by
\parref{unirational-res-dominant} and the choice of \(r\), providing the
sought after rational map over the dominant map \(S' \to S\). The choice of
numbers in \parref{unirational-numbers} implies
\(r' = r - 1 \geq r(\mathbf{a}')\) and \(n' = n - 1 \geq n_0(\mathbf{a}',r')\),
completing the induction in this case.

If \(\mathbf{a}_{\mathrm{nlr}} = \varnothing\) and
\(\mathbf{a}_{\mathrm{lin}} \neq \varnothing\), then apply
\parref{families-remove-linear-equations} to view
\(\mathcal{P} \subseteq \mathcal{X}\)
as a family of \(r\)-planed \((q;\mathbf{a}')\)-tic complete intersections
in a projective subbundle \(\PP\mathcal{V}' \subseteq \PP\mathcal{V}\) where
\(\mathbf{a}' \coloneqq \mathbf{a} \setminus \mathbf{a}_{\mathrm{lin}}\).
Thus \(\mathcal{X}' = \mathcal{X}\) and \(S' = S\), but \(\mathcal{X}'\) is
equipped with the structure of a \((q;\mathbf{a}')\)-tic complete intersection
in a different projective bundle. By \parref{families-generic-in-subbundle},
the family \(\mathcal{P} \subseteq \mathcal{X}\) is also generic with this
structure. Now
\[
r' = r \geq r(\mathbf{a}) = r(\mathbf{a}')
\;\;\text{and}\;\;
n' = n - \#\mathbf{a}_{\mathrm{lin}}
\geq n_0(\mathbf{a},r) - \#\mathbf{a}_{\mathrm{lin}}
= n_0(\mathbf{a}',r'),
\]
completing the inductive step here.

Finally, if \(\mathbf{a} = \mathbf{a}_{\mathrm{pow}}\), apply
\parref{families-frobenius-descent} to obtain a family \(\mathcal{X}'\) of
\((q;\mathbf{a}')\)-tic complete intersections in a \(\PP^n\)-bundle
over \(S' = S\) with \(\mathbf{a}' \coloneqq \mathbf{a}/t\) and fitting into a
diagram
\[
\begin{tikzcd}
\mathcal{X}' \rar \dar & \mathcal{X} \dar \\ S \rar["\Fr"] & S
\end{tikzcd}
\]
where the horizontal map up top is a universal homeomorphism. Pulling back
the family \(\mathcal{P}\) of \(r\)-planes equips \(\mathcal{X}'\) with a family
of \(r\)-planes. It is straightforward to check that genericity of the
family \(\mathcal{P} \subseteq \mathcal{X}\) propagates to genericity of
\(\mathcal{P}' \subseteq \mathcal{X}'\). Since \(r' \coloneqq r\) and
\(n' \coloneqq n\), but also \(r(\mathbf{a}) = r(\mathbf{a}/t)\) and
\(n_0(\mathbf{a},r) = n_0(\mathbf{a}/t,r)\), the inductive step is settled
in this case too.
\end{proof}

\begin{proof}[Proof of Theorem \parref{intro-unirationality}]
Given a multi-profile \(\mathbf{a}\), set
\[
r \coloneqq r(\mathbf{a})
\;\;\text{and}\;\;
n \coloneqq n_0(\mathbf{a}) \coloneqq n_0(\mathbf{a}, r)
\]
where \(r(\mathbf{a})\) and \(n_0(\mathbf{a},r)\) are as in
\parref{unirational-numbers}. Consider the incidence correspondence
\[
\mathbf{Inc}_{n,r,\mathbf{a}} \coloneqq
\{
([U], [\boldsymbol{\alpha}]) \in
\mathbf{G}(r+1,n+1) \times \qaticci_{\PP^n} :
\PP U \subseteq X_{\boldsymbol{\alpha}}
\}
\]
between \(r\)-planes and \((q;\mathbf{a})\)-tic schemes in \(\PP^n\). This
is a projective bundle over the Grassmannian via the first projection, and so
it is integral.
Let \(S \subseteq \mathbf{Inc}_{n,r,\mathbf{a}}\) be the open subscheme
parameterizing \(r\)-planed \((q;\mathbf{a})\)-tic complete intersections
\(P \subseteq X\). Restricting the tautological family to \(S\) therefore
provides a family \(\mathcal{P} \subseteq \mathcal{X}\) of \(r\)-planed
\((q;\mathbf{a})\)-tic complete intersections in \(\PP^n_S\) satisfying the
hypotheses of \parref{unirational-in-families}. Therefore each fibre of
\(\mathcal{X} \to S\)
is unirational. Since the choice of \(n\) and \(r\) imply, via
\parref{fano-covered-in-planes}, that the projection \(S \to \qaticci_{\PP^n}\)
is dominant and the result follows.
\end{proof}

\subsectiondash{Example}\label{unirationality-estimates}
The bound \(n_0(\mathbf{a})\) may be computed for small multi-profiles
\(\mathbf{a}\). Some examples:
\[
n_0(t+1) = 4, \quad
n_0(2,t+1) = 9, \quad
n_0(t+1,t+1) = 13, \quad
n_0(t^2 + t + 1) = 48.
\]
The bounds obtained for constant profiles \(a(t) = d\), corresponding to
degree \(d\) hypersurfaces as in \parref{qatics-examples}\ref{qatics-examples.constant},
are surprisingly small:
\begin{gather*}
n_0(3) = 4, \quad
n_0(4) = 9, \quad
n_0(5) = 22, \quad
n_0(6) = 160, \quad
n_0(7) = 20376, \\
n_0(8) = 11914188890, \quad
n_0(9) = 8616199237736295920955120,
\end{gather*}
and
\(n_0(10) = 192884152577980851363553858004926940342106493833715693762179 \approx 10^{59}\).
In comparison, the best bounds \(n'_0(d)\) for the classical unirationality
construction are from \cite{Ramero}, and they give
\begin{gather*}
n'_0(3) = 3, \quad
n'_0(4) = 20, \quad
n'_0(5) = 8855, \quad
n'_0(6) = 454205040715033146, \\
n'_0(7) \approx 10^{103}, \quad
n'_0(8) \approx 10^{717}, \quad
n'_0(9) \approx 10^{5738}, \quad
n'_0(10) \approx 10^{51641}.
\end{gather*}
In the companion paper \cite{bounds}, it is shown that the integer \(n_0(d)\)
is on the order of \(2^{d2^d}\), growing significantly slower than \(n_0'(d) \sim 2^{d!}\).

\bibliographystyle{amsalpha}
\bibliography{main}
\end{document}